\documentclass[reqno,a4paper]{amsart}

\usepackage[cmyk,usenames,dvipsnames,svgnames,table]{xcolor}

\definecolor{myblue}{cmyk}{0.77, 0.23, 0.0, 0.37}
\definecolor{mycyan}{cmyk}{0.46, 0.0, 0.0, 0.22}
\definecolor{myyellow}{cmyk}{0.0, 0.12, 0.6, 0.0}
\definecolor{mydarkgray}{cmyk}{0.1, 0.0, 0.2, 0.78}
\definecolor{myred}{cmyk}{0.0, 0.6, 0.61, 0.05}
\definecolor{myviolet}{cmyk}{0.0, 0.76, 0.15, 0.5}

\usepackage{titletoc}
\contentsmargin{2.55em}
\dottedcontents{section}[1.5em]{}{2.8em}{1pc}
\dottedcontents{subsection}[3.7em]{}{3.4em}{1pc}
\dottedcontents{subsubsection}[7.6em]{}{3.2em}{1pc}
\dottedcontents{paragraph}[10.3em]{}{3.2em}{1pc}

\usepackage{enumitem}
\setenumerate{label=\textnormal{(\arabic*)}}
\newcommand{\ra}[1]{\renewcommand{\arraystretch}{#1}}

\allowdisplaybreaks[3]

\usepackage{amsmath,amssymb,dsfont,mathrsfs,verbatim,bm,geometry}
\usepackage{mathtools, longtable, tabu, booktabs}
\usepackage[latin1]{inputenc}
\usepackage{array,hhline}
\usepackage[raggedright]{titlesec}
\usepackage{tikz}
\usepackage{float,subfig}
\usetikzlibrary{arrows, backgrounds, fit, matrix, positioning}
\usepackage{amsrefs}

\titleformat{\section}
{\normalfont\Large\bfseries\center}{\thesection}{1em}{}
\titleformat{\subsection}
{\normalfont\large\bfseries}{\thesubsection}{1em}{}
\titleformat{\subsubsection}
{\normalfont\normalsize\bfseries}{\thesubsubsection}{1em}{}
\titleformat{\paragraph}[runin]
{\normalfont\normalsize\bfseries}{\theparagraph}{1em}{}
\titleformat{\subparagraph}[runin]
{\normalfont\normalsize\bfseries}{\thesubparagraph}{1em}{}

\titlespacing*{\chapter} {0pt}{50pt}{40pt}
\titlespacing*{\section} {0pt}{3.5ex plus 1ex minus .2ex}{2.3ex plus .2ex}
\titlespacing*{\subsection} {0pt}{3.25ex plus 1ex minus .2ex}{1.5ex plus .2ex}
\titlespacing*{\subsubsection}{0pt}{3.25ex plus 1ex minus .2ex}{1.5ex plus .2ex}
\titlespacing*{\paragraph} {0pt}{3.25ex plus 1ex minus .2ex}{1em}
\titlespacing*{\subparagraph} {\parindent}{3.25ex plus 1ex minus .2ex}{1em}

\newtheorem{theorem}{Theorem}[section]
\newtheorem{lemma}[theorem]{Lemma}
\newtheorem{proposition}[theorem]{Proposition}
\newtheorem{corollary}[theorem]{Corollary}

\theoremstyle{definition}
\newtheorem{definition}[theorem]{Definition}

\newtheorem{notation}[theorem]{Notation}
\newtheorem{example}[theorem]{Example}

\theoremstyle{remark}
\newtheorem{remark}[theorem]{Remark}

\DeclareMathOperator{\Diag}{Diag}

\DeclareMathOperator{\ide}{id}
\DeclareMathOperator{\ima}{Im}
\DeclareMathOperator{\In}{In}

\DeclareMathOperator{\J}{J}

\DeclareMathOperator{\Supp}{Supp}

\DeclareMathOperator{\rk}{rk}
\DeclareMathOperator{\rrank}{rrank}

\DeclareMathOperator{\Trans}{T}
\DeclareMathOperator{\Tr}{Tr}

\newcommand{\ov}{\overline}
\newcommand{\ot}{\otimes}

\newcommand{\hs}{\hspace{-0.5pt}}

\newcommand{\xcirc}{\hs \circ \hs}

\begin{document}

\title[Twisted tensor products of $K^n$ with $K^m$]{Twisted tensor products of $\bm{K^n}$ with $\bm{K^m}$}

\author[J. Arce]{Jack Arce}
\address{Pontificia Universidad Cat\'olica del Per\'u,  Secci\'on Matem\'aticas, PUCP, Av. Universitaria 1801, San Miguel, Lima 32, Per\'u.}
\email{jarcef@pucp.edu.pe}

\author{Jorge A. Guccione}
\address{Departamento de Matem\'atica\\ Facultad de Ciencias Exactas y Naturales-UBA, Pabell\'on~1-Ciudad Universitaria\\ Intendente Guiraldes 2160
(C1428EGA) Buenos Aires, Argentina.}
\address{Instituto de Investigaciones Matem\'aticas ``Luis A. Santal\'o"\\ Facultad de Ciencias Exactas y Natu\-ra\-les-UBA, Pabell\'on~1-Ciudad
Universitaria\\ Intendente Guiraldes 2160 (C1428EGA) Buenos Aires, Argentina.}
\email{vander@dm.uba.ar}

\author{Juan J. Guccione}
\address{Departamento de Matem\'atica\\ Facultad de Ciencias Exactas y Naturales-UBA\\ Pabell\'on~1-Ciudad Universitaria\\ Intendente Guiraldes 2160
(C1428EGA) Buenos Aires, Argentina.}
\address{Instituto Argentino de Matem\'atica-CONICET\\ Savedra 15 3er piso\\ (C1083ACA) Buenos Aires, Argentina.}
\email{jjgucci@dm.uba.ar}

\author[C. Valqui]{Christian Valqui}
\address{Pontificia Universidad Cat\'olica del Per\'u - Instituto de Matem\'atica y Ciencias Afi\-nes, Secci\'on Matem\'aticas, PUCP, Av. Universitaria 1801,
San Miguel, Lima 32, Per\'u.}
\email{cvalqui@pucp.edu.pe}

\subjclass[2010]{}

\begin{abstract}
We find three families of twisting maps of $K^m$ with $K^n$. One of them is related to truncated quiver algebras, the second one consists of
deformations of the first and the third one requires $m=n$ and yields algebras isomorphic to $M_n(K)$. Using these families and some exceptional
cases we construct all twisting maps of $K^3$ with $K^3$.
\end{abstract}

\maketitle

\tableofcontents

\section*{Introduction}
Let $A$, $C$ be unitary $K$-algebras, where $K$ is a field. By definition, a  twisted tensor product of $A$ with $C$ over $K$, is an algebra
structure defined on $A\ot_K C$, with unit $1\ot 1$, such that the canonical maps $i_A\colon A\to A\ot_K C$ and $i_C\colon C\to A\ot_K C$ are
algebra maps satisfying $a\ot c = i_A(a)i_C(c)$. When $K$ is a commutative ring this structure was introduced independently in \cite{Ma1} and
\cite{Tam}, and it has been formerly studied by many people with different motivations (In addition to the previous references see also \cite{B-M1},
\cite{B-M2}, \cite{Ca}, \cite{C-S-V}, \cite{C-I-M-Z}, \cite{G-G}, \cite{Ma2}, \cite{J-L-P-V}, \cite{VD-VK}).  A number of examples of classical and
recently defined constructions in ring theory fits into this construction. For instance, Ore extensions, skew group algebras, smash products,
etcetera (for the definitions and properties of these structures we refer to \cite{Mo} and \cite{Ka}). On the other hand, it has been applied to
braided geometry and it arises as a natural representative for the product of noncommutative spaces, this being based on the existing duality
between the categories of algebraic affine spaces and commutative algebras, under which the cartesian product of spaces corresponds to the tensor
product of algebras. And last, but not least, twisted tensor products arise as a tool for building algebras starting with simpler ones.

Given algebras $A$ and $C$, a basic problem is to determine all the twisted tensor products of $A$ with $C$. To our knowledge, the first paper in
which this problem was attacked in a systematic way was \cite{C}, in which C. Cibils studied and completely solved the case $C\coloneqq K\times K$.
In \cite{JLNS},the case $C\coloneqq K^n$ is analysed and some partial classification result were achieved.

In this paper we consider the case $A=K^m$ and $C=K^n$. It is well known that there is a canonical bijection between the twisted tensor products of
$A$ with $C$ and the so called twisting maps $\chi\colon C\ot_k A \to A\ot_k C$. So each twisting map $\chi$ is associated with a twisted tensor
product of $A$ with $C$ over $K$, which will be denoted by $A\ot_\chi C$.

It is evident that each $K$-linear map $\chi\colon K^n\ot K^m \to K^m\ot K^n$ determines and is determined by unique scalars $\lambda_{ij}^{kl}$,
such that
\begin{equation*}
\chi(e_i\ot f_j)=\sum_{k,l} \lambda_{ij}^{kl} f_k\ot e_l\qquad\text{for all $e_i$ and $f_j$.}
\end{equation*}
Given such a map $\chi$, for all $i,l\in \mathds{N}^*_m$ and $j,k\in \mathds{N}^*_n$, we let $A(i,l)\in M_n(K)$ and $B(j,k)\in M_m(K)$ denote the
matrices defined by
\begin{equation*}
A(i,l)_{kj}\coloneqq  \lambda_{ij}^{kl} \eqqcolon B(j,k)_{li}.
\end{equation*}

In Proposition~\ref{caracterizacion} we show that $\chi$ is a twisting map if and only if these matrices satisfy certain (easily verifiable)
conditions. This transforms the problem of finding all twisting maps into a problem of linear algebra. When one tries to find all twisting maps of
$K^3$ with $K^3$ using this linear algebra approach, one encounters that nearly all cases of twisting maps have a very special form. We call these
twisting maps standard and prove that the resulting twisted tensor product algebras are isomorphic to certain square zero radical truncated quiver
algebras. Moreover, there arises a second type of twisting maps, which we call quasi-standard twisting maps, which yield algebras which corresponds
to a formal deformation of the latter case, whenever the corresponding quiver has a triangle which is not a cycle. We also construct a third family
of twisting maps when $n=m$, and we show that the resulting algebras are isomorphic to $M_n(K)$.

These three families cover nearly all twisting maps of $K^3$ with $K^3$. We find additionally some extensions of the algebras corresponding to the
third family in the case $K^2$ with $K^2$, and one additional case.

The paper is organized as follows: in section~1 we make a quick review of twisting maps and the $n-1$-ary cross product of vectors. In section~2 we
present the characterization in terms of matrices of the twisting maps of $K^n$ with $K^m$ and some basic results, specially on isomorphism of
twisting maps and a basic representation on $M_n(K)$.

In section~3 we reprove the results of \cite{C} in our language. In section~4 we prove some basic results on the idempotent matrices $A(i,l)$, and
pay special attention to the case of $\rk=1$, where a family arises with algebras isomorphic to $M_n(K)$. In section~5 we define standard and
quasi-standard twisting columns and twisting maps and prove several results about them. In section~6 we classify completely the case of reduced
rank~$1$. In section~7 we explore the relation of standard twisting maps and quiver algebras, and also the case of quasi-standard twisting maps. In
section~8 we use all results in order to classify the twisting maps in low dimensional cases, including all the twisting maps which are not
quasi-standard in the case $K^3$ with $K^3$. In the appendix we list all standard and quasi-standard twisting maps of $K^3$ with $K^3$.

\section{Preliminaries}
Let $K$ be a field. From now on we assume implicitly that all the maps whose domain and codomain are $K$-vector spaces are $K$-linear maps and that
all the algebras are over $K$. Next we introduce some notations and make some comments.

\begin{itemize}

\smallskip

\item[-] $K^{\times}\coloneqq K\setminus \{0\}$.

\smallskip

\item[-] For each natural number $i$, we set $\mathds{N}^*_i\coloneqq \{1,\dots, i\}$.

\smallskip

\item[-] The tensor product over $K$ is denoted by $\ot$, without any subscript.

\smallskip

\item[-] Given a matrix $X$ we let $X^{\Trans}$ denote the transpose matrix of $X$. Moreover, we denote with a juxtaposition the multiplication
of two matrices and with a bullet the multiplication in $K^n$. So, $(a_1,\dots, a_n) \centerdot  (b_1,\dots, b_n)= (a_1b_1,\dots, a_nb_n)$. Note
that $\mathbf{a}=(a_1,\dots, a_n)$ is invertible respect to the multiplication map $\centerdot$ if and only if $\mu_n(\mathbf{a})\coloneqq
a_1\cdots a_n\ne 0$. In this case we let $\mathbf{a}^{\centerdot}$ denote the inverse $(a_1^{-1},\dots, a_n^{-1})$ of $\mathbf{a}$.

\smallskip

\item[-] We let $E^{ij}\in M_n(K)$ denote the matrix with~$1$ in the $i,j$-entry and $0$ otherwise. So, $\{E^{ij}: 1 \le i,j \le n\}$ is the
canonical basis of $M_n(K)$.

\smallskip

\item[-] For the sake of simplicity we write $\mathds{1}=\mathds{1}_n=\mathds{1}_{K^n}\coloneqq (1,\dots,1)^{T}$.

\smallskip

\item[-] The symbol $\tau_{nm}$ denotes the flip $K^n\ot K^m \longrightarrow K^n\ot K^m$.

\end{itemize}

\subsection{Twisting maps}
Let $A$, $C$ be unitary algebras. Let $\mu_A$, $\eta_A$, $\mu_C$ and $\eta_C$ be the multiplication and unit maps of $A$ and $C$, respectively. A
{\em twisted tensor product} of $A$ with $C$ is an algebra structure on the $K$-vector space $A\ot C$, such that the canonical maps
$$
i_A\colon A\longrightarrow A\ot C\quad\text{and}\quad i_C\colon C \longrightarrow A\ot C
$$
are algebra homomorphisms and $\mu\xcirc (i_A\ot i_C) = \ide_{A\ot C}$, where $\mu$ denotes the multiplication map of the twisted tensor product.

\smallskip

Assume we have a twisted tensor product of $A$ with $C$. Then, the map
$$
\chi\colon C\ot A\longrightarrow A\ot C,
$$
defined by $\chi \coloneqq \mu \xcirc (i_C\ot i_A)$, satisfies:

\begin{enumerate}

\smallskip

\item $\chi\xcirc (\eta_C\ot A) = A\ot \eta_C$,

\smallskip

\item $\chi \xcirc (C\ot \eta_A) = \eta_A\ot C$,

\smallskip

\item $\chi \xcirc (\mu_C\ot A) = (A\ot \mu_C)\xcirc (\chi\ot C)\xcirc (C\ot \chi)$,

\smallskip

\item $\chi\xcirc (C\ot \mu_A) = (\mu_A\ot C)\xcirc (A\ot \chi)\xcirc (\chi\ot A)$.

\smallskip

\end{enumerate}
A map satisfying these conditions is called a {\em twisting map} of $C$ with $A$. Conversely, if
$$
\chi\colon C\ot A\longrightarrow A\ot C
$$
is a twisting map, then $A\ot C$ becomes a twisted tensor product via
$$
\mu_{\chi} \coloneqq (\mu_A\ot\mu_C)\xcirc (A\ot \chi \ot C).
$$
This algebra will be denoted by $A\ot_{\chi} C$. Furthermore, these constructions are inverse one to each other.

\begin{definition}\label{morfismo de twisting map} Let $\chi \colon C\ot A \longrightarrow A\ot C$ and $\chi' \colon C'\ot A' \longrightarrow A'\ot
C'$ be twisting maps. A {\em morphism} $F_{gh}\colon \chi \to \chi'$, from $\chi$ to $\chi'$, is a pair $(g,h)$ of algebra maps $g\colon C\to C'$
and $h\colon A\to A'$ such that $\chi'\xcirc (g\ot h) = (h\ot g)\xcirc \chi$.
\end{definition}

\begin{remark}\label{morfismos de twisting maps a morfiamos de algebras} Let $\chi$ and $\chi'$ be as above. If $F_{gh}\colon \chi \to \chi'$ is a
morphism of twisting maps, then the map $h\ot g\colon A\ot_{\chi} C \longrightarrow A'\ot_{\chi'} C'$ is a morphism of algebras. Moreover this
correspondence is functorial in an evident sense.
\end{remark}

\begin{remark}\label{twisting map modificado por isomorfismos}  Let $h\colon A\to A'$ and $g\colon C\to C'$ be isomorphisms of algebras. If
$$
\chi' \colon C'\ot A' \longrightarrow A'\ot C'
$$
is a twisting map, then $\chi\coloneqq (h^{-1}\ot g^{-1}) \xcirc \chi'\xcirc (g\ot h)$ is also. Moreover $F_{gh}\colon \chi\to \chi'$ is an
isomorphism.
\end{remark}

\begin{proposition}\label{extension de twistings} Let $\chi\colon (B\times C)\otimes A\longrightarrow A\otimes (B\times C)$ be a twisting map.
Denote by $\iota_B$, $\iota_C$, $p_B$, $p_C$ be the evident inclusions and projections. The map $\chi_B\colon B\otimes A\longrightarrow A\otimes
B$, defined by
$$
\chi_B\coloneqq (A\otimes p_B)\circ \chi \circ (\iota_B\otimes A),
$$
is a twisting map if and only if $(A\otimes p_B)\circ \chi\circ (\iota_C\otimes A)=0$. Moreover in this case $F_{p_B,\ide_A}$ is a morphism of
twisting maps from $\chi$ to $\chi_B$. We say that $p_B(\chi)\coloneqq \chi_B$ is \emph{the twisting map of $B$ with $A$ induced by $\chi$} and
that $\chi$ is an {\em extension} of $\chi_B$.
\end{proposition}

\begin{proof} Since $\chi$ is a twisting map
\begin{align*}
\chi\bigl((1_B,0)\ot a\bigr) &= \chi\bigl((1_B,1_C)\ot a\bigr) - \chi\bigl((0,1_C)\ot a\bigr)\\
& = a\ot (1_B,1_C) - \chi\bigl((0,1_C)\ot a\bigr)\\
& = a\ot (1_B,1_C) + a\ot (0,1_C) - \chi\bigl((0,1_C)\ot a\bigr).
\end{align*}
Consequently, if $\chi_B$ is also a twisting map, then
$$
a\ot 1_B = \chi_B(1_B\ot a) = a\ot 1_B - (A\ot p_B)\circ \chi \bigl((0,1_C)\ot a\bigr),
$$
or, equivalently, $(A\ot p_B)\circ \chi \bigl((0,1n_C)\ot a\bigr) =0$. Evaluating now the equalities
$$
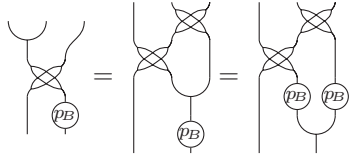
\begin{tikzpicture}[scale=0.5]
\def\mult(#1,#2)[#3]{\draw (#1,#2) .. controls (#1,#2-0.555*#3/2) and (#1+0.445*#3/2,#2-#3/2) .. (#1+#3/2,#2-#3/2) .. controls (#1+1.555*#3/2,#2-#3/2) and (#1+2*#3/2,#2-0.555*#3/2) .. (#1+2*#3/2,#2) (#1+#3/2,#2-#3/2) -- (#1+#3/2,#2-2*#3/2)}
\def\map(#1,#2)[#3]{\draw (#1,#2-0.5)  node[name=nodemap,inner sep=0pt,  minimum size=9.5pt, shape=circle,draw]{$#3$} (#1,#2)-- (nodemap)  (nodemap)--(#1,#2-1)}
\def\twisting(#1,#2)[#3]{\draw (#1+#3,#2) .. controls (#1+#3,#2-0.05*#3) and (#1+0.96*#3,#2-0.15*#3).. (#1+0.9*#3,#2-0.2*#3) (#1,#2-1*#3) .. controls (#1,#2-0.95*#3) and (#1+0.04*#3,#2-0.85*#3).. (#1+0.1*#3,#2-0.8*#3) (#1+0.1*#3,#2-0.8*#3) ..controls (#1+0.25*#3,#2-0.8*#3) and (#1+0.45*#3,#2-0.69*#3) .. (#1+0.50*#3,#2-0.66*#3) (#1+0.1*#3,#2-0.8*#3) ..controls (#1+0.1*#3,#2-0.65*#3) and (#1+0.22*#3,#2-0.54*#3) .. (#1+0.275*#3,#2-0.505*#3) (#1+0.72*#3,#2-0.50*#3) .. controls (#1+0.75*#3,#2-0.47*#3) and (#1+0.9*#3,#2-0.4*#3).. (#1+0.9*#3,#2-0.2*#3) (#1,#2) .. controls (#1,#2-0.05*#3) and (#1+0.04*#3,#2-0.15*#3).. (#1+0.1*#3,#2-0.2*#3) (#1+0.5*#3,#2-0.34*#3) .. controls (#1+0.6*#3,#2-0.27*#3) and (#1+0.65*#3,#2-0.2*#3).. (#1+0.9*#3,#2-0.2*#3) (#1+#3,#2-#3) .. controls (#1+#3,#2-0.95*#3) and (#1+0.96*#3,#2-0.85*#3).. (#1+0.9*#3,#2-0.8*#3) (#1+#3,#2) .. controls (#1+#3,#2-0.05*#3) and (#1+0.96*#3,#2-0.15*#3).. (#1+0.9*#3,#2-0.2*#3) (#1+0.1*#3,#2-0.2*#3) .. controls  (#1+0.3*#3,#2-0.2*#3) and (#1+0.46*#3,#2-0.31*#3) .. (#1+0.5*#3,#2-0.34*#3) (#1+0.1*#3,#2-0.2*#3) .. controls  (#1+0.1*#3,#2-0.38*#3) and (#1+0.256*#3,#2-0.49*#3) .. (#1+0.275*#3,#2-0.505*#3) (#1+0.50*#3,#2-0.66*#3) .. controls (#1+0.548*#3,#2-0.686*#3) and (#1+0.70*#3,#2-0.8*#3)..(#1+0.9*#3,#2-0.8*#3) (#1+#3,#2-1*#3) .. controls (#1+#3,#2-0.95*#3) and (#1+0.96*#3,#2-0.85*#3).. (#1+0.9*#3,#2-0.8*#3) (#1+0.72*#3,#2-0.50*#3) .. controls (#1+0.80*#3,#2-0.56*#3) and (#1+0.9*#3,#2-0.73*#3)..(#1+0.9*#3,#2-0.8*#3)(#1+0.72*#3,#2-0.50*#3) -- (#1+0.50*#3,#2-0.66*#3) -- (#1+0.275*#3,#2-0.505*#3) -- (#1+0.5*#3,#2-0.34*#3) -- (#1+0.72*#3,#2-0.50*#3)}
\begin{scope}[xshift=0cm, yshift=-0.5cm]
\mult(0,0)[1]; \draw (2,0) .. controls (2,-0.5) and (1.5,-0.5)..(1.5,-1); \twisting(0.5,-1)[1]; \draw (0.5,-2) -- (0.5,-3); \map(1.5,-2)[\scriptstyle p_{\!\scriptscriptstyle B}];
\end{scope}
\begin{scope}[xshift=2.5cm, yshift=-0.5cm]
\node at (0,-1.5){=};
\end{scope}
\begin{scope}[xshift=3.3cm, yshift=0cm]
\draw (0,0) -- (0,-1);  \twisting(1,-0)[1]; \twisting(0,-1)[1];\draw (2,-1) -- (2,-2); \mult(1,-2)[1]; \draw (0,-2) -- (0,-4); \map(1.5,-3)[\scriptstyle p_{\!\scriptscriptstyle B}];
\end{scope}
\begin{scope}[xshift=5.8cm, yshift=-0.5cm]
\node at (0,-1.5){=};
\end{scope}
\begin{scope}[xshift=6.6cm, yshift=0cm]
\draw (0,0) -- (0,-1);  \twisting(1,-0)[1]; \twisting(0,-1)[1];\draw (2,-1) -- (2,-2); \draw (0,-2) -- (0,-4);  \map(1,-2)[\scriptstyle p_{\!\scriptscriptstyle B}]; \map(2,-2)[\scriptstyle p_{\!\scriptscriptstyle B}]; \mult(1,-3)[1];
\end{scope}
\end{tikzpicture}
$$
in $(0,1_C)\ot (0,c) \ot a$ for all $c\in C$ and $a\in A$, we conclude that $(A\otimes p_B)\circ \chi\circ (\iota_C\otimes A)=0$. We leave to the reader the task to check the other assertions.
\end{proof}

\subsection{Cross product}
We recall that the {\em cross product} is the $(n\!-\!1)$-ary operation
$$
(\mathbf{v}_1,\dots, \mathbf{v}_{n-1})  \mapsto \mathbf{v}_1 \times \cdots \times \mathbf{v}_{n-1}
$$
on $K^n$, determined by
$$
(\mathbf{v}_1 \times \cdots \times \mathbf{v}_{n-1})\mathbf{x}^{\Trans} = \det \begin{pmatrix} \mathbf{x} \\ \mathbf{v}_1\\ \vdots \\
\mathbf{v}_{n-1}\end{pmatrix}
$$
for all $\mathbf{x}\in K^n$. From this definition it follows immediately that $\mathbf{v}_1 \times \cdots \times \mathbf{v}_{n-1}$ is orthogonal to
the subspace $\langle \mathbf{v}_1,\dots, \mathbf{v}_{n-1} \rangle$ generated by $\mathbf{v}_1,\dots, \mathbf{v}_{n-1}$, and that $\mathbf{v}_1
\times \cdots \times \mathbf{v}_{n-1}=0$ if $\mathbf{v}_1,\dots, \mathbf{v}_{n-1}$ are not linearly independent. It is well known (and very easy to
check) that
$$
\mathbf{v}_1 \times \cdots \times \mathbf{v}_{n-1} = \det \begin{pmatrix} e_1 & \dots & e_n \\ v_{11} & \cdots & v_{1n}\\ \vdots & \ddots & \vdots
\\ v_{n-1,1} & \cdots & v_{n-1,n} \end{pmatrix},
$$
where $\{e_1,\dots,e_n\}$ is the standard basis of $K^n$, $\mathbf{v}_i = (v_{i1},\dots,v_{in})$ and the determinant is computed by the Laplace
expansion along the first row. From this it follows immediately that if $X$ is the matrix with rows $\mathbf{x}_1, \dots, \mathbf{x}_n$ and columns
$\mathbf{y}_1, \dots, \mathbf{y}_n$, then

\begin{equation}\label{eq cross prod y transpuesta}
(\mathbf{y}_1^{\Trans} \times\cdots\times\widehat{\mathbf{y}_j^{\Trans}}\times \cdots \times \mathbf{y}_n^{\Trans})\centerdot e_j = (\mathbf{x}_1
\times\cdots \times \widehat{\mathbf{x}_j} \times\cdots\times \mathbf{x}_n)\centerdot e_j \qquad\text{for all $j$.}
\end{equation}

\begin{proposition}\label{relacion entre el producto cruzado y el producto en K^n} If $\mathbf{x}\in K^n$ is invertible, then
$$
\mathbf{x}\centerdot (\mathbf{v}_1 \times \cdots \times \mathbf{v}_{n-1})= \mu_n(\mathbf{x}) \, (\mathbf{x}^{\centerdot} \centerdot \mathbf{v}_1)
\times \cdots \times (\mathbf{x}^{\centerdot} \centerdot \mathbf{v}_{n-1}),
$$
for all $\mathbf{v}_1,\dots, \mathbf{v}_{n-1}\in K^n$, where $\mu_n(\mathbf{x})=x_1\cdots x_n$, as in the introduction.
\end{proposition}

\begin{proof} This assertion is an immediate consequence of the fact that
\begin{align*}
\mathbf{y} \bigl(\mathbf{x}\centerdot (\mathbf{v}_1 \times \cdots \times \mathbf{v}_{n-1})\bigr)^{\Trans} &= (\mathbf{x}\centerdot \mathbf{y})
(\mathbf{v}_1 \times \cdots \times \mathbf{v}_{n-1})^{\Trans}\\
&= \det \begin{pmatrix} (\mathbf{x}\centerdot \mathbf{y})^{\Trans} & \mathbf{v}_1^{\Trans} & \cdots &  \mathbf{v}_{n-1}^{\Trans}\end{pmatrix}\\
&= \tau(\mathbf{x}) \det \begin{pmatrix} \mathbf{y}^{\Trans} & (\mathbf{x}^{\centerdot}\centerdot \mathbf{v}_1)^{\Trans} & \cdots &
(\mathbf{x}^{\centerdot}\centerdot \mathbf{v}_{n-1})^{\Trans}\end{pmatrix}\\
&= \tau(\mathbf{x}) \, \mathbf{y} \bigl((\mathbf{x}^{\centerdot}\centerdot \mathbf{v}_1) \times \cdots \times (\mathbf{x}^{\centerdot} \centerdot
\mathbf{v}_{n-1})\bigr)^{\Trans}
\end{align*}
for all $\mathbf{y}\in K^n$.
\end{proof}

\section[Twisted tensor products of $K^n$ with $K^m$]{Twisted tensor products of $\bm{K^n}$ with $\bm{K^m}$}\label{Twisted tensor products of K^n with K^m}
Let $\chi\colon K^m\ot K^n \longrightarrow K^n \ot K^m$ be a map and let $\{e_1,\dots, e_m\}$ and $\{f_1,\dots, f_n\}$ be the canonical bases of
$K^m$ and $K^n$, respectively. There exist unique scalars $\lambda_{ij}^{kl}$, such that
\begin{equation}\label{formula para chi}
\chi(e_i\ot f_j)=\sum_{k,l} \lambda_{ij}^{kl} f_k\ot e_l\qquad\text{for all $e_i$ and $f_j$.}
\end{equation}
Given such a map $\chi$, for all $i,l\in \mathds{N}^*_m$ and $j,k\in \mathds{N}^*_n$, we let $A(i,l)\in M_n(K)$ and $B(j,k)\in M_m(K)$ denote the matrices defined by
\begin{equation}
A(i,l)_{kj}\coloneqq  \lambda_{ij}^{kl} \eqqcolon B(j,k)_{li}.\label{eq:rel entre A(i,l) y B(j,k)}
\end{equation}
If necessary we will specify these map with a subscript, writing $A_{\chi}(i,l)$ and $B_{\chi}(j,k)$. Moreover, we let
$\mathcal{A}=\mathcal{A}_{\chi}$ denote the family $(A(i,l))_{i,l\in \mathds{N}^*_m}$ and $\mathcal{B}=\mathcal{B}_{\chi}$ denote the family
$(B(j,k))_{j,k\in \mathds{N}^*_n}$.

\begin{notation}\label{j sub i(l)} For each $i,l\in \mathds{N}^*_m$ we set $J_i(l)\coloneqq \{j\in \mathds{N}^*_n: A(i,l)_{jj}=1\}$. If there is no
danger of confusion (as is the case, for example, when we work with the matrices $A(1,l),\dots,A(m,l)$ of a fixed column of $\mathcal{A}$), we
write $J_i$ instead of $J_i(l)$. Similarly, for each $i,l\in \mathds{N}^*_n$ we set $\tilde{J}_u(k)\coloneqq \{i\in \mathds{N}_m^* :
B_{\chi}(u,k)_{ii}=1\}$, and we write $\tilde{J}_u$ instead of $\tilde{J}_u(k)$ whenever there is no danger of confusion.
\end{notation}

\begin{remark}\label{A sub Chi es B sub tau chi tau} Let $\tilde{\chi}\coloneqq \tau_{mn}\circ \chi \circ \tau_{nm}$ An immediate computation shows
that
$$
A_{\tilde{\chi}}(i,l)_{kj} = B_{\chi}(i,l)_{kj}\qquad\text{and}\qquad B_{\tilde{\chi}}(j,k)_{li} = A_{\chi}(j,k)_{li}
$$
for each map $\chi\colon K^m\ot K^n\longrightarrow K^n \ot K^m$. Moreover $\tilde{\chi}$ is a twisting map if and only if $\chi$ is. In this case we
say that $\chi$ and $\tilde{\chi}$ are {\em dual} of each other.
\end{remark}

\begin{proposition}\label{caracterizacion} The map $\chi$ is a twisting map if and only if the following facts hold:

\begin{enumerate}

\smallskip

\item $\delta_{ii'} A(i,l) = A(i,l)A(i',l)$ for all $i$, $i'$ and $l$,

\smallskip

\item $\delta_{jj'} B(j,k) = B(j,k)B(j',k)$ for all $j$, $j'$ and $k$,

\smallskip

\item $A(i,l)\mathds{1}=\delta_{il}\mathds{1}$ for all $i$ and $l$,

\smallskip

\item $B(j,k)\mathds{1}=\delta_{jk}\mathds{1}$ for all $j$ and $k$.

\end{enumerate}

\end{proposition}

\begin{proof} A direct computation shows that
$$
\chi \xcirc (\mu_{K^m}\ot K^n) = (K^n\ot \mu_{K^m})\xcirc (\chi\ot K^m)\xcirc (K^m\ot \chi)
$$
if and only if
$$
\delta_{ii'}\lambda_{ij}^{kl}=\sum_{u=1}^n \lambda_{iu}^{kl}\lambda_{i'j}^{ul}\quad\text{for all $i,i',j,k,l$,}
$$
which is equivalent to condition~(1), and that
$$
\chi\xcirc (K^m\ot \mu_{K^n}) = (\mu_{K^n}\ot K^m)\xcirc (K^n\ot \chi)\xcirc (\chi\ot K^n)
$$
if and only if
$$
\delta_{jj'}\lambda_{ij}^{kl}=\sum_{u=1}^m \lambda_{ij}^{ku} \lambda_{uj'}^{kl}\quad\text{for all $i,j,j',k,l$,}
$$
which is equivalent to condition~(2). Finally it is easy to check that
$$
\chi \xcirc (K^m\ot \eta_{K^n}) = \eta_{K^n}\ot K^m\qquad\text{and}\qquad \chi\xcirc (\eta_{K^m}\ot K^n) = K^n\ot \eta_{K^m}
$$
if and only if conditions~(3) and~(4) are fulfilled.
\end{proof}

\begin{remark} Statement~(1) says that for each $l\in \mathds{N}_m^*$, the matrices $A(1,l),\dots, A(m,l)$ are a family of orthogonal idempotents,
and Statement~(2) says that for each $k\in \mathds{N}_n^*$, the matrices $B(1,k),\dots, B(n,k)$ are also a family of orthogonal idempotents.
Statement~(1) implies that Statement~(3) holds if and only if  $\mathds{1}_{K^n}$ belongs to the image of $A(i,i)$ for all $i$. Similarly, if
Statement~(2) is fulfilled, then Statement~(4) is true if and only if $\mathds{1}_{K^m}\in \ima B(j,j)$ for all $j$.
\end{remark}

\begin{corollary}\label{caracterizacion en terminos de las A(i,j)s} The map $\chi$ is a twisting map if and only if the following conditions are
fulfilled:

\begin{enumerate}

\smallskip

\item $\delta_{ii'} A(i,l) = A(i,l)A(i',l)$ for all $i$, $i'$ and $l$,

\smallskip

\item $A(i,l)\mathds{1}=\delta_{il}\mathds{1}$ for all $i$ and $l$,

\smallskip

\item $\sum_{i=1}^m A(i,l)=\ide$ for all $l$,

\smallskip

\item $\sum_{h=1}^m A(i,h)_{kj} A(h,l)_{kj'}=\delta_{jj'}A(i,l)_{kj}$ for all $i$, $j$, $j'$, $k$ and $l$.

\end{enumerate}
\end{corollary}

\begin{proof} Conditions~(1) and~(2) are conditions~(1) and~(3) of Proposition~\ref{caracterizacion}. Since by~\eqref{eq:rel entre A(i,l) y B(j,k)},
$$
\sum_{i=1}^m A(l,i)_{kj}=\sum_{i=1}^m B(j,k)_{li},
$$
condition~(3) is equivalent to condition~(4) of that proposition, and since, again by~\eqref{eq:rel entre A(i,l) y B(j,k)},
$$
\sum_{u=1}^m A(u,l)_{kj}A(i,u)_{kj'} = \sum_{u=1}^m B(j,k)_{lu}B(j',k)_{ui},
$$
condition~(4) is equivalent to condition~(2) of the same proposition.
\end{proof}

\begin{remark}\label{corolario con las B} By Remark~\ref{A sub Chi es B sub tau chi tau} and the fact that $\chi$ is a twisting map if and only if
$\tilde{\chi}$ is, there is  a similar corollary with the matrices $A(i,l)$ replaced by the matrices  $B(j,k)$.
\end{remark}

\begin{remark} Corollary~\ref{caracterizacion en terminos de las A(i,j)s}(4) says in particular that the vector $\bigl(A(i,1)_{kj},\dots,A(i,m)_{kj}
\bigr)$ is orthogonal to the vector $\bigl(A(1,l)_{kj'},\dots, A(m,l)_{kj'}\bigr)$ for each $i$, $j$, $j'$, $k$ and $l$ with $j\ne j'$.
\end{remark}

\begin{remark}\label{cuando son idempotentes} Let $X_1,\dots,X_k\in M_n(K)$ such that $\sum_{j=1}^k X_j = \ide_n$. A straightforward computation
shows that if $\sum_{j=1}^k \rk(X_j) \le n$, then the $X_i$'s are orthogonal idempotents, which means that $X_iX_j = \delta_{ij} X_i$ for all $i$
and $j$.
\end{remark}

\begin{remark}\label{cuando son idempotentes ortogonales} Let $X_1,\dots,X_k\in M_n(K)$ be idempotent matrices such that $\sum_{i=1}^k X_i=\ide_n$.
Then the $X_i$'s are orthogonal idempotents. In fact, since $\rk(X_i)=\Tr(X_i)$ and
$$
\sum_{i}\Tr(X_i)=\Tr\bigg(\sum_{i}X_{i}\bigg)=\Tr(\ide)=n,
$$
this follows from the Remark~\ref{cuando son idempotentes}.
\end{remark}

\begin{remark}\label{Simplificacion Corolario 2.4(4)} Fix $k\in \mathds{N}_n^*$ and assume that $\sum_j A(i,l)_{kj} = \delta_{il}$ for all $i$ and
$l$ (which is Corollary~\ref{caracterizacion en terminos de las A(i,j)s}~(2) for this $k$). If the equality in Corollary~\ref{caracterizacion en
terminos de las A(i,j)s}(4) holds for all $i$, $l$ and $j=j'$, then it holds for all $i$, $j$, $j'$ and $l$. In fact, the assumptions guarantee
that $B(j,k)$ is idempotent for each $j\in \mathds{N}^*_n$ and that $\sum_{j=1}^n  B(j,k) = \ide$. So, by  Remark~\ref{cuando son idempotentes
ortogonales}, the family of idempotent matrices $(B(j,k))_{j\in \mathds{N}^*_n}$ is orthogonal, which is equivalent to
Corollary~\ref{caracterizacion en terminos de las A(i,j)s}(4) for this fixed $k$.
\end{remark}

\begin{definition}\label{matriz de rangos} The matrices $\Gamma_{\chi}\in M_m(K)$, of $\mathcal{A}$-ranks, and $\tilde\Gamma_{\chi}\in M_n(K)$,
of~$\mathcal{B}$-ranks, are~def\-ined by
$$
\Gamma_{\chi}\coloneqq\begin{pmatrix} \gamma_{11} & \dots & \gamma_{1m} \\ \vdots & \ddots & \vdots \\ \gamma_{m1} & \cdots & \gamma_{mm}
\end{pmatrix}\quad\text{and}\quad \tilde{\Gamma}_{\chi}\coloneqq\begin{pmatrix} \tilde\gamma_{11} & \dots & \tilde\gamma_{1n} \\ \vdots & \ddots &
\vdots \\ \tilde\gamma_{n1} & \cdots & \tilde\gamma_{nn} \end{pmatrix},
$$
where $\gamma_{il}\coloneqq \rk(A(i,l))$ and $\tilde\gamma_{jk}\coloneqq \rk(B(j,k))$.
\end{definition}

\begin{corollary}\label{propiedades de la matriz de rangos} If $\chi$ is a twisting map, then the rank matrices have the following properties:
\begin{enumerate}

\item $\delta_{il}\le \gamma_{il} \le n$ for all $i$ and $l$.

\smallskip

\item $\sum_{i=1}^m \gamma_{il}=n$ for all $l$.

\smallskip

\item $\delta_{jk}\le \tilde\gamma_{jk} \le m$ for all $j$ and $k$.

\smallskip

\item $\sum_{j=1}^n \tilde\gamma_{jk}=m$ for all $k$.

\end{enumerate}
\end{corollary}

\begin{proof} Items~(1) and~(2) follow from Corollary~\ref{caracterizacion en terminos de las A(i,j)s} and~(3) and~(4) from the corresponding properties of the $B(j,k)$'s.
\end{proof}

\subsection{Isomorphisms of twisting maps}

\begin{proposition}\label{matrices de rango equivalentes} Two twisting maps $\chi,\chi' \colon K^m\ot K^n\longrightarrow K^n\ot K^m$ are isomorphic
if and only if there exists $\sigma\in S_m$ and $\varsigma\in S_n$ such that
$$
A_{\chi'}(i,l)_{kj}=A_{\chi}(\sigma(i),\sigma(l))_{\varsigma(k)\varsigma(j)}
$$
or, equivalently,
$$
B_{\chi'}(j,k)_{li}=B_{\chi}(\varsigma(j),\varsigma(k))_{\sigma(l)\sigma(i)}.
$$
\end{proposition}

\begin{proof} By definition $\chi$ and $\chi'$ are isomorphic if and only if there are algebra automorphisms $g\colon K^m\to K^m$ and $h\colon K^n
\to K^n$ such that $\chi'=(h^{-1}\otimes g^{-1})\circ \chi\circ (g\otimes h)$. Since the automorphisms of $K^n$ and $K^m$ are given by permutation
of the entries, there exist $\varsigma\in S_n$ and $\sigma\in S_m$ such that $g(e_i)=e_{\sigma(i)}$ and $h(f_j)=f_{\varsigma(j)}$ for all $i\in
\mathds{N}^*_m$ and $j\in \mathds{N}^*_n$, and so
\begin{align*}
\chi'(e_i\otimes f_j) & = (h^{-1}\otimes g^{-1})\chi(e_{\sigma(i)}\otimes f_{\varsigma(j)})\\
& = \sum_{k,l}\lambda_{\sigma(i)\varsigma(j)}^{\varsigma(k)\sigma(l)} (h^{-1}\otimes g^{-1})( f_{\varsigma(k)}\otimes e_{\sigma(l)})\\
& = \sum_{k,l}\lambda_{\sigma(i)\varsigma(j)}^{\varsigma(k)\sigma(l)}f_k\otimes e_l .
\end{align*}
Now the result follows immediately from~\eqref{formula para chi} and~\eqref{eq:rel entre A(i,l) y B(j,k)}.
\end{proof}

\subsection{Representations in matrix algebras}

In this subsection $\chi\colon K^m\ot K^n \longrightarrow K^n \ot K^m$ denotes a twisting map and $\lambda_{ij}^{kl}$, $A(i,l)$ and $B(j,k)$ are as
at the beginning of Section~\ref{Twisted tensor products of K^n with K^m}.

\begin{proposition}\label{representacion} For each $1\le u\le m$ the formulas
$$
\rho_u(f_j\otimes 1)\coloneqq E^{jj} \quad\text{and}\quad \rho_u (1\otimes e_i)\coloneqq A(i,u)
$$
define a representation $\rho_u\colon K^n\otimes_{\chi} K^m \longrightarrow M_n(K)$. Similarly, for each $1\le v\le n$ the formulas
$$
\tilde\rho_v( 1\otimes e_i)\coloneqq E^{ii} \quad\text{and}\quad\tilde\rho_v (f_j\otimes 1)\coloneqq B(j,v)
$$
define a representation $\tilde\rho_v\colon K^n\otimes_{\chi} K^m \longrightarrow M_m(K)$.
\end{proposition}

\begin{proof} Clearly the restriction of $\rho_u$ to $K^n\otimes K\cdot\mathds{1}$ is a morphism of algebras. Moreover, by items~(1) and~(3) of
Corollary~\ref{caracterizacion en terminos de las A(i,j)s}, the restriction of $\rho_u$ to $K\cdot\mathds{1}\otimes K^m$ is also a morphism of
algebras. So, in order to prove that $\rho_u$ defines a representation, it only remains to check that
$$
\rho_u\bigl((1\otimes e_i)(f_j\otimes 1)\bigr)=\rho_u(1\otimes e_i)\rho_u(f_j\otimes 1).
$$
But this is true, since, on one hand,
$$
(1\otimes e_i)(f_j\otimes 1)=\sum_{k,l}\lambda_{ij}^{kl}f_k\otimes e_l,
$$
and, on the other hand,
$$
A(i,u) E^{jj}=\sum_{k,l} \lambda_{ij}^{kl}E^{kk}A(l,u),
$$
because
\begin{align*}
\sum_{k,l} \lambda_{ij}^{kl}E^{kk}A(l,u)&= \sum_{k,l,s} \lambda_{ij}^{kl} A(l,u)_{ks}E^{ks}\\
&=  \sum_{k,l,s}  A(i,l)_{kj} A(l,u)_{ks}E^{ks}\\
&= \sum_{k} A(i,u)_{kj}E^{kj}\\
&= A(i,u) E^{jj},
\end{align*}
where the first and the last equality are straightforward, the second equality is true by~\eqref{eq:rel entre A(i,l) y B(j,k)} and the third one by
Corollary~\ref{caracterizacion en terminos de las A(i,j)s}(4). The proof for $\tilde\rho_v$ is similar.
\end{proof}

\begin{remark}\label{imagen de rho} We can give a complete description of the image of $\rho_u$ and $\tilde\rho_v$. For this, note that if
$A(i,u)_{kj}\ne 0$ for some $i$, $j$ and $k$, then $E^{kj}\in \ima(\rho_u)$. In fact,
$$
E^{kj}A(i,u)_{kj}=E^{kk}A(i,u)E^{jj}=\rho_u((f_k\otimes1)(1\otimes e_i)(f_j\otimes 1)).
$$
Hence
$$
E^{kj}=\rho_u\left(\frac{(f_k\otimes1)(1\otimes e_i)(f_j\otimes 1)}{A(i,u)_{kj}}\right).
$$
so, the image of $\rho_u$ is the matrix incidence algebra of the preorder on $\{1,\dots,n\}$ given by $k\le j$ if and only if $E^{kj}\in\ima
(\rho_u)$. In particular, if for all $k,j$ there exists a $i$ with $A(i,u)_{kj}\ne 0$, then $\rho_u$ is surjective. Similarly, the image of
$\tilde\rho_v$ is the  matrix incidence algebra of the preorder on $\{1,\dots,m\}$ given by $l\le i$ if and only if $E^{li}\in\ima (\tilde\rho_u)$.
\end{remark}

\begin{remark}\label{tabla e ideales} Set $x_{ji}\coloneqq f_j\otimes e_i$. A straightforward computation shows that in $K^n\otimes_{\chi} K^m$
$$
x_{ki}x_{jl}=\lambda_{ij}^{kl}x_{kl}=A(i,l)_{kj} x_{kl}= B(j,k)_{li} x_{kl}.
$$
We also can prove that all two sided ideals of the algebra $K^n\otimes_{\chi} K^m$ are generated by monomials. In fact, let $I$ be an ideal and let
$x=\sum_{r,s}\alpha_{rs} x_{rs}$. Then
$$
(f_j\otimes 1)\left(\sum_{r,s}\alpha_{rs} x_{rs}\right)(1\otimes e_i)=\sum_{r,s}\alpha_{rs}(f_j\otimes 1)(f_r\otimes 1)(1\otimes e_s)(1\otimes e_i)=\alpha_{ji} x_{ji},
$$
and so, if $\alpha_{ji}\ne 0$ for some element $x\in I$, then $x_{ji}\in I$. This shows that the ideal $I$ is linearly generated by a set of elements $x_{ji}$.
\end{remark}

\section[Twisting maps of $K^m$ with $K^2$]{Twisting maps of $\bm{K^m}$ with $\bm{K^2}$}\label{section twisting maps m2}
The proofs given in this section could be lightly simplified using some of the results given in Section 4, but we prefer to use the least machinery
possible in order to give a flavour of how our methods work, reproducing the beautiful result of Cibils in \cite{C}. Therefore we restrict
ourselves to use the results established in the previous sections and the following remark:

\begin{remark}\label{rango uno} Let $A\in M_2(K)$ be such that $A^2=A$, $A\mathds{1}=\mathds{1}$ and $\rk(A)=1$. There exists $a\in K$ such that
$$
A=\left(\begin{array}{cc}
a&1-a\\
a&1-a
\end{array}\right).
$$
\end{remark}

\smallskip

The twisting maps of $K^m$ with $K^2$ have been classified completely by Cibils, who shows that they correspond to colored quivers $Q_{f,\delta}$.
The first step is to describe the quiver $Q_f$. We can obtain this quiver directly from our $\mathcal{A}$-rank matrix. Given an algebra map $f
\colon C\to C$, where $C\coloneqq K^m$, we let ${}^f\! C$ denote that $C$-bimodule structure on $C$ given by $c\cdot c'\cdot c''\coloneqq
f(c)c'c''$. Let $(e_i)_{i\in \mathds{N}^*_m}$ be the canonical basis of~$C$.

Consider a map
$$
\chi\colon C \otimes \frac{K[X]}{\langle X(1-X)\rangle}\longrightarrow  \frac{K[X]}{\langle X(1-X)\rangle} \otimes  C.
$$
In~\cite{C}*{Section 3} it was proved that $\chi$ is a twisting map if and only if there exists an algebra morphism $f\colon C\to C$ and an
idempotent derivation $\delta \colon C\to {^f \!C}$, satisfying $f=f^2+\delta f+f\delta$, such that
$$
\chi(e_i \otimes X) = X \otimes f(e_i) + 1 \otimes \delta(e_i) = X \otimes (f+\delta)(e_i) + (1-X) \otimes \delta(e_i).
$$
With our notations, we have
$$
\chi(e_i\otimes f_1) =\sum_l\bigl(\lambda_{i1}^{1l}f_1 \otimes e_l+\lambda_{i1}^{2l} f_2 \otimes e_l\bigr)=\sum_l \bigl(A(i,l)_{11} f_1\otimes e_l +
A(i,l)_{21} f_2 \otimes e_l\bigr),
$$
where $f_1$ is the class of $X$ in $k[X]/\langle X(1-X) \rangle$ and $f_2$ is the class of $1-X$ in $k[X]/\langle X(1-X) \rangle$. Hence
\begin{equation}\label{ecuacion de f}
f(e_i)=\sum_l \bigl(A(i,l)_{11} - A(i,l)_{21}\bigr)e_l \quad \text{and} \quad \delta(e_i)=\sum_l A(i,l)_{21} e_l.
\end{equation}

\smallskip

The quiver $Q_f$ in~\cite{C} is constructed in the following way. Since $f$ is an algebra map, there exists a unique set map $\varphi\colon
\mathds{N}^*_m\to  \mathds{N}^*_m$, such that
\begin{equation}\label{formula de f}
f(e_l)=\sum_{\{i : \varphi(i)=l\}} e_i.
\end{equation}
By definition, the quiver $Q_f$ of $f$ has set of vertices $\mathds{N}^*_m$ and an arrow from $i$ to $\varphi(i)$ for each $i\in \mathds{N}^*_m$.

\begin{proposition}\label{quiver de f} Let $\chi$ be a twisting map and let $f$ be as above. The adjacency matrix of the quiver $Q_f$ is
$M(\chi)\coloneqq (\Gamma_{\chi}-\ide)^T$, where $\Gamma_{\chi}$ is as in Definition~\ref{matriz de rangos}.
\end{proposition}

\begin{proof} Let $l\in \mathds{N}_m^*$. By Corollary~\ref{propiedades de la matriz de rangos} we know that $\rk(A(l,l)) = 2$ and $A(i,l) = 0$ for
all $i\ne l$, or $\rk(A(l,l)) = 1$ and there exists a unique $i\ne l$ such that $\rk(A(i,l) = 1$ and $A(j,l) = 0$ for all $j\notin\{i,l\}$. Thus,
if $\rk(A(l,l)) = 2$ then $A(l,l) = \ide$, and so $A(l,l)_{11}-A(l,l)_{21}=1$. On the other hand if $\rk(A(l,l)) = 1$, then by
Proposition~\ref{caracterizacion} and Remark~\ref{rango uno} there exists $a_l\in K$ such that $A(l,l)=\begin{psmallmatrix}  a_l & 1-a_l\\a_l &
1-a_l\end{psmallmatrix}$, and hence $A(l,l)_{11} - A(l,l)_{21}=0$. Moreover, since $A(i,l) + A(l,l) = \ide$, we have $A(i,l)=\begin{psmallmatrix}
1-a_l & a_l-1\\-a_l & a_l \end{psmallmatrix}$, and so $A(i,l)_{11}-A(i,l)_{21}=1$. Finally, if $\rk(A(j,l)) = 0$, then (of course)
$A(j,l)_{11}-A(j,l)_{21}=0$. Consequently, by the first equality in~\eqref{ecuacion de f} and equality~\eqref{formula de f},
$$
M(\chi)_{il} = \begin{cases}1 & \text{if $\varphi(i)=l$,} \\ 0 & \text{otherwise,}\end{cases}
$$
which finishes the proof.
\end{proof}

\begin{corollary} A vertex $i$ of $Q_f$ is a loop vertex if and only if $\rk (A(i,i)) = 2$.
\end{corollary}

In the rest of this section, for each $i\in \mathds{N}^*_m$ we let $a_i$ denote $A(i,i)_{11}$. We want to determine the possible matrices $A(i,l)$
which can occur in a twisting map of $K^m$ with~$K^2$:

\begin{enumerate}

\smallskip

\item If $\rk(A(l,l))=2$, then $A(l,l)=\ide$ and $A(i,l)=0$ for all $i\ne l$.

\smallskip

\item If $\rk(A(l,l))=1$, then there exists $i\ne l$ such that
$$
\quad\qquad A(l,l)=\begin{pmatrix}  a_l & 1-a_l\\ a_l & 1-a_l\end{pmatrix},\quad A(i,l)=\begin{pmatrix}  1-a_l & a_l-1\\ -a_l & a_l \end{pmatrix}
\quad\text{and}\quad  A(h,l)=0 \text{ for $h\notin\{i,l\}$.}
$$
Now we have several possibilities:

\begin{itemize}

\smallskip

\item[-] If $\rk(A(i,i))=2$, then $A(l,i)=0$, and so, by~\eqref{eq:rel entre A(i,l) y B(j,k)} and Proposition~\ref{caracterizacion}(2),
\begin{equation}\label{al es 0 o 1}
\qquad\qquad a_l-a_l^2 = B(1,1)_{ll}-(B(1,1)^2)_{ll}=0,
\end{equation}
which implies that $a_l\in\{0,1\}$.

\smallskip

\item[-] If $\rk(A(i,i))=1$, then we have $A(i,i)=\begin{psmallmatrix}  a_i & 1-a_i\\ a_i & 1-a_i \end{psmallmatrix}$, and, again
by~\eqref{eq:rel entre A(i,l) y B(j,k)} and Proposition~\ref{caracterizacion}(2),
\begin{align}
\qquad\qquad  &(1-a_l)(1-a_i-a_l)  = B(1,1)_{li} - (B(1,1)^2)_{li} = 0\label{conectar bk y bj}
\shortintertext{and}
\qquad\qquad  &a_l(a_i+a_l-1) = B(2,2)_{li} - (B(2,2)^2)_{li} = 0.\label{conectar bk y bj 2}
\end{align}
Hence $a_i + a_l = 1$. If $A(l,i)\ne 0$, then we do not obtain additional conditions on $a_l$, while if $A(l,i)=0$, then, by~\eqref{al es 0 o 1}, we have $a_l\in\{0,1\}$, and so there are only two cases: $a_l=0$ and $a_i=1$ or $a_l=1$ and $a_i=0$.
\end{itemize}
\end{enumerate}

Next we recall the definition of a coloration on $Q_f$ in \cite{C}*{Definition 3.12}, but we take the opposite coloration.

\begin{definition}\label{coloracion} A coloration of $Q_f$ is an element $c =\sum_i c_i e_i\in C$ such that:

\begin{enumerate}

\smallskip

\item For a connected component reduced to the round trip quiver with vertices $i$ and $j$ the coefficients $c_i$ and $c_j$ satisfy $c_i + c_j=1$.

\smallskip

\item For other connected components:

\smallskip

\begin{enumerate}[label=(\alph*)]

\item In case $i$ is a non loop vertex $c_i\in\{0,1\}$.

\smallskip

\item For each arrow having no loop vertex target, one extremity value is $0$ and the other is $1$.

\smallskip

\item At a loop vertex $i$ we have $c_i=0$.
\end{enumerate}
\end{enumerate}
\end{definition}

Given a twisting map $\chi\colon K^m\otimes K^2\longrightarrow K^2\otimes  K^m$ consider the matrices $A(i,l)\coloneqq A_{\chi}(i,l)$. By
Proposition~\ref{quiver de f} and the discussion above Definition~\ref{coloracion}, the element $c\coloneqq (c_1,\dots,c_m)\in C$ given by
$c_l\coloneqq A(l,l)_{21}$ is a coloration. Conversely, given a coloration $c=(c_1,\dots,c_m)\in C$ on a one-valued quiver $Q_f$ with set of
vertices $\mathds{N}_m^*$, we can construct matrices $A(i,l)\in M_2(K)$ in the following way: if $l$ is a loop vertex, then $A(l,l)\coloneqq \ide$
and  $A(i,l)\coloneqq 0$ for $i\ne l$. Otherwise
\begin{itemize}

\smallskip

\item[-] we set $A(l,l)\coloneqq \begin{psmallmatrix}  a_l & 1-a_l\\a_l & 1-a_l\end{psmallmatrix}$, where $a_l\coloneqq c_l$,

\smallskip

\item[-] for the target $t(l)$ of the arrow starting at $l$, we set $A(t(l),l)\coloneqq \begin{psmallmatrix}  1-a_l & a_l-1\\ -a_l & a_l
\end{psmallmatrix}$,

\smallskip

\item[-] for all $i\notin\{t(l),l\}$, we set $A(h,l)\coloneqq 0$.

\end{itemize}
In order to verify that these matrices define a twisting map, we must check the conditions of Proposition~\ref{caracterizacion}, where the matrices
$B(j,k)$ are defined by~\eqref{eq:rel entre A(i,l) y B(j,k)}. Conditions~(1) and~(3) are satisfied by construction. Condition~(2) is equivalent to
$$
\sum_i A(i,l)_{kj}=\delta_{jk}\qquad \text{for all $l$, $j$ and $k$,}
$$
which holds, because
$$
\sum_i A(i,l)_{kj}=\begin{cases} A(l,l)_{kj} = \delta_{jk}, & \text{if $\rk(A(l,l))=2$}\\ A(l,l)_{kj} + A(t(l),l)_{kj}=\delta_{jk}, & \text{if
$\rk(A(l,l))=1$.}\end{cases}
$$
Finally we check condition~(4), which is equivalent to
\begin{equation}\label{la condicion 2}
\delta_{jj'}A(i,l)_{kj} = \sum_u A(i,u)_{kj'} A(u,l)_{kj} \qquad\text{for all $i$, $j$ , $j'$, $k$ and $l$}.
\end{equation}
When  $t(l) = l$, then $A(u,l) = \delta_{ul}\ide$ for all $u$, which implies that equality~\eqref{la condicion 2} holds. Assume that $t(l)\ne l$.
We consider three cases: $i=l$, $i=t(l)$ and $i\notin\{l,t(l)\}$. If $i=l$, then equality~\eqref{la condicion 2} reads
$$
\delta_{jj'}A(l,l)_{kj} = A(l,l)_{kj'} A(l,l)_{kj} + A(l,t(l))_{kj'} A(t(l),l)_{kj}\qquad\text{for all $j$, $j'$ and $k$;}
$$
if $i=t(l)$, then equality~\eqref{la condicion 2} reads
$$
\delta_{jj'}A(t(l),l)_{kj} = A(t(l),l)_{kj'}A(l,l)_{kj} + A(t(l),t(l))_{kj'}A(t(l),l)_{kj} \qquad\text{for all $j$, $j'$ and $k$;}
$$
and finally, if $i\notin \{l,t(l)\}$, then equality~\eqref{la condicion 2} reads
$$
0 = A(i,t(l))_{kj'}A(t(l),l)_{kj}\qquad\text{for all $j$, $j'$ and $k$.}
$$
All these equalities are easily verified using that, since $(c_1,\dots,c_m)$ is a coloration,
$$
A(l,l)=\begin{pmatrix}  a_l & 1-a_l\\ a_l & 1-a_l\end{pmatrix},\qquad A(t(l),l)=\begin{pmatrix}  1-a_l & a_l-1\\-a_l & a_l \end{pmatrix},
$$
and that

\begin{itemize}

\smallskip

\item[-] if $t(t(l))=t(l)$, then $A(t(l),t(l))=\ide$;

\smallskip

\item[-] if $t(t(l))\ne t(l)$, then $A(t(l),t(l))=\begin{psmallmatrix} 1-a_l & a_l\\1-a_l & a_l\end{psmallmatrix}$;

\smallskip

\item[-] if $t(t(l))=l$, then $A(l,t(l))= \begin{psmallmatrix} a_l & -a_l\\a_l-1 & 1-a_l\end{psmallmatrix}$ and $A(u,t(l))=0$ for all
$u\notin\{l,t(l)\}$;

\smallskip

\item[-] if $t(t(l))\ne l$, then $a_l\in \{0,1\}$, $A(t(t(l)),t(l)) = \begin{psmallmatrix} a_l & -a_l\\a_l-1 & 1-a_l\end{psmallmatrix}$ and
$A(u,t(l)) = 0$ for all $u\notin \{t(l),t(t(l))\}$.

\end{itemize}

\section{Miscellaneous results}\label{seccion resultados generales}
Throughout this section $\chi\colon K^m\ot K^n \longrightarrow K^n \ot K^m$ denotes a map and $\lambda_{ij}^{kl}$, $A(i,l)$ and $B(j,k)$ are as at
the beginning of Section~\ref{Twisted tensor products of K^n with K^m}. We also assume that $A(i,l)$ and $B(j,k)$ are idempotent matrices for all
$i,l\in \mathds{N}^*_m$ and $j,k\in \mathds{N}^*_n$.  The following results are useful in our quest of classifying the twisted tensor products
$K^n\otimes_{\chi} K^m$.

\subsection{General properties}
\begin{remark}\label{suma de rangos} The rank matrices $\Gamma_{\chi}$ and $\tilde{\Gamma}_{\chi}$, introduced in Definition~\ref{matriz de rangos},
have the same trace. In fact,
$$
\Tr(\Gamma_{\chi})=\sum_i\rk(A(i,i))=\sum_{i,j}\lambda_{ij}^{ji}=\sum_j\rk(B(j,j))=\Tr(\tilde{\Gamma}_{\chi}).
$$
\end{remark}

\begin{remark} Since the matrices $A(i,l)$ are idempotent, we know that $\rk(A(i,l))=\Tr(A(i,l))$. Consequently,
$$
\rk(A(i,l)) = \sum_j A(i,l)_{jj} = \sum_j B(j,j)_{li}.
$$
Similarly, $\rk(B(j,k))=\sum_i A(i,i)_{kj}$.
\end{remark}

\subsection[Standard idempotent $0,1$-matrices]{Standard idempotent $\bm{0,1}$-matrices}
\begin{definition}\label{definicion matriz estandar} A $0,1$-matrix $A\in M_n(K)$ is called a \emph{standard idempotent $0,1$-matrix} if there
exist $r\in \mathds{N}_n^*$ and a matrix $C\in M_{n-r\times r}(K)$ that has exactly one non-zero entry in each row, such that
\begin{equation}\label{forma de matriz}
A=\begin{pmatrix}\ide_r&0\\ C&0 \end{pmatrix},
\end{equation}
where $\ide_r$ is the identity of $M_r(K)$.
\end{definition}

\begin{definition} Two matrices $A,A'\in M_n(K)$ are \emph{equivalent via identical permutations in rows and columns} if there exists a permutation
$\sigma\in S_n$ such that $A_{\sigma(k)\sigma(j)}=A'_{kj}$ for all $k,j$.
\end{definition}

\begin{remark}\label{matrices equivalentes a estandars} A matrix $A\in M_n(K)$ is equivalent via identical permutations in rows and columns to a
standard idempotent $0,1$-matrix if and only if it is a $0,1$-matrix with exactly one nonzero entry in every row, that satisfies the following
condition: for each $j$, if $A_{jj}=0$, then $A_{kj}=0$ for all $k$.
\end{remark}

\begin{notation} Let $A\in M_n(K)$ be a $0,1$-matrix such that $A\mathds{1}=\mathds{1}$. For each $k$ such that~$A_{kk}=0$, we let $c_k=c_k(A)$
denote the unique index such that $A_{kc_k}$ is non-zero.
\end{notation}

\begin{proposition}\label{01 matrix} Let $A\in M_n(K)$ be a $0,1$-matrix. If $A$ is idempotent and $A\mathds{1}=\mathds{1}$, then $A$ is equivalent
via identical permutations in rows and columns to a standard idempotent $0,1$-matrix.
\end{proposition}

\begin{proof} Since $r\coloneqq \rk(A)=\Tr(A)$, we have $r$ times the entry~$1$ and $n-r$ times the entry~$0$ on the diagonal of $A$. Applying an
identical permutations in rows and columns we can assume that the~$1$'s are in the first $r$ entries. Since $A\mathds{1}=\mathds{1}$, each row of
this matrix has only one~$1$, and the other entries are zero. Thus, the first~$r$ rows of $A$ are as in~\eqref{forma de matriz}. Now the fact that
$\rk(A)=r$ implies that, again as in~\eqref{forma de matriz}, the right lower block of~$A$ is the zero matrix and its left lower block is a matrix
$C$ that satisfies the required properties.
\end{proof}

\begin{remark}\label{c sub k esta en J sub k} By Proposition~\ref{01 matrix} we have $A_{c_kc_k} = 1$ for each $k$ such that $A_{kk} = 0$.
\end{remark}

\begin{corollary}\label{lema 01 matrix previo} Assume that $\chi$ is a twisting map. If $A(l,l)$ is a $0,1$-matrix, then $A(l,l)$ is equivalent via
identical permutations in rows and columns to a standard idempotent $0,1$-matrix.
\end{corollary}

\begin{proposition}\label{forma general casi triangular} Assume that $\chi$ is a twisting map and let $l\in \mathds{N}^*_m$. If
$$
\rk(A(i,l))\rk(A(l,i))=0\quad\text{for all $i\ne l$,}
$$
then $A(l,l)$ is a $0,1$-matrix.
\end{proposition}

\begin{proof} By Corollary~\ref{caracterizacion en terminos de las A(i,j)s}(4) and the fact that $A(i,l)A(l,i)=0$ for all $i\ne l$,
$$
A(l,l)_{kj}=\sum_{i=1}^m A(l,i)_{kj} A(i,l)_{kj}= A(l,l)_{kj}^2.
$$
So, $A(l,l)_{kj}\in\{0,1\}$ for all $k,j$.
\end{proof}

\begin{corollary}\label{forma general triangular} If $\chi$ is a twisting map and $\Gamma_{\chi}$ is upper or lower triangular, then each of the
matrices $A(l,l)$ is a $0,1$-matrix.
\end{corollary}

\begin{remark}\label{lo mismo para B} Proposition~\ref{forma general casi triangular} and Corollaries~\ref{lema 01 matrix previo} and~\ref{forma
general triangular} are valid for the matrices $B(j,j)$ (in the second corollary we replace $\Gamma_{\chi}$ by $\tilde\Gamma_{\chi}$).
\end{remark}

\subsection[Rank $1$ idempotent matrices]{Rank $\bm{1}$ idempotent matrices}
\begin{remark}\label{rango uno generalizado} At the beginning of Section~\ref{section twisting maps m2} we noted that if $A\in M_2(K)$ satisfies
$A^2=A$, $A\mathds{1}=\mathds{1}$ and $\rk(A)=1$, then there exists $a\in K$ such that
$$
A=\left(\begin{array}{cc} a&1-a\\ a&1-a \end{array}\right).
$$
More generally, if $A\in M_n(K)$ such that $A^2=A$, $A\mathds{1}=\mathds{1}$ and $\rk(A)=1$, then there exists $a_1,\dots,a_n\in K$ with $\sum
a_j=1$, such that
$$
A=\left(\begin{array}{ccc} a_1&\dots&a_n\\ \vdots&\ddots&\vdots\\ a_1&\dots&a_n \end{array}\right).
$$
\end{remark}

\begin{proposition}\label{Gamma en caso de rango de Aii igual a 1} If $\rk(A(i,i))=1$ for some $i\in\mathds{N}_m^*$, then there exists
$j\in\mathds{N}_n^*$ such that $\tilde\Gamma_{jk}\ne 0$ for all $k$. Moreover, if such $j$ is unique, then $A(i,i)_{st}=\delta_{tj}$ for all $s,t$.
A similar statement holds for $B(j,j)$ and $\Gamma$.
\end{proposition}

\begin{proof} Since $\Tr(A(i,i))=\rk(A(i,i))=1$, there exists $j$ such that $A(i,i)_{jj}\ne 0$. By Remark~\ref{rango uno generalizado}
$$
B(j,k)_{ii}=A(i,i)_{kj}=A(i,i)_{jj}\ne 0,\quad\text{for all $k$.}
$$
This implies that $B(j,k)\ne 0$ for all $k$, and so $\tilde\Gamma_{jk}\ne 0$ for all $k$.  If $j$ is unique, then for each $l \ne j$ there
exists~$k$ such that $\tilde{\Gamma}_{lk}=0$, and so, again by Remark~\ref{rango uno generalizado}, we have
$$
A(i,i)_{hl}=A(i,i)_{kl}=B(l,k)_{ii}=0 \quad\text{for all $h$.}
$$
The argument for $B(j,j)$ and $\Gamma$ is the same.
\end{proof}

\subsection[Columns of $1$'s in $\Gamma_{\chi}$]{Columns of $\bm{1}$'s in $\bm{\Gamma_{\chi}}$}

\begin{proposition}\label{diagonales uno} Assume that $\chi$ is a twisting map and that $n=m$. If $\Diag\bigl(\Gamma_{\chi}\bigr)=(1,1,\dots,1)$,
then $\Gamma_{\chi}=\tilde{\Gamma}_{\chi}$ is the matrix $\mathfrak{J}_n$ whose entries are all $1$.
\end{proposition}

\begin{proof} By Remark~\ref{suma de rangos} and Proposition~\ref{propiedades de la matriz de rangos}(3), we know that
$\Diag\bigl(\tilde{\Gamma}_{\chi})=(1,\dots,1)$. In other words, $\rk(B(j,j))=1$ for all $j$. Assume by contradiction that $\Gamma_{\chi}\ne
\mathfrak{J}_n$. Then by items~(1) and~(3) of Corollary~\ref{caracterizacion en terminos de las A(i,j)s} there exist $i,l$ such that $A(i,l)=0$.
Hence, by Remark~\ref{rango uno generalizado} the $i$-th column of $B(j,j)$ is zero for all $j$. But then $\Diag(A(i,i))=(0,\dots,0)$, which, since
$A(i,i)$ is idempotent, implies that $A(i,i)=0$, a contradiction. For $\tilde{\Gamma}_{\chi}$ proceed in a similar way.
\end{proof}

\begin{proposition}\label{proposicion determinante} Let $l\in \mathds{N}^*_m$. Assume that $\chi$ is a twisting map, that
$\Gamma_{\chi}=\tilde{\Gamma}_{\chi}$ is the matrix $\mathfrak{J}_n$ whose entries are all $1$'s, and that there exists $k$ such that $A(l,l)_{kj}
\ne 0$ for all $j$. Let $\mathbf{v}_i=(v_{i1},\dots,v_{in})\in K^n\setminus\{0\}$.  If $\mathbf{v}_i^{\bot} \in \ima (A(i,l))$, then $v_{ik}\ne 0$
for all $k$.
\end{proposition}

\begin{proof} Since $\rk(A(i,l))=1$ there exists $\mathbf{w}_i=(w_{i1},\dots,w_{in})\in K^n$ such that $A(i,l)= \mathbf{v}_i^{\Trans}\mathbf{w}_i$.
Assume by contradiction that there exists $k$ such that $v_{ik}=0$. Then  $A(i,l)_{kj}=v_{ik}w_{ij}=0$ for all $j$. By~\eqref{eq:rel entre A(i,l) y
B(j,k)} this means that $B(j,k)_{li}=0$ for all $j$, and so
\begin{equation}\label{determinante}
\det\begin{pmatrix}B(1,k)_{l1}&\dots & B(1,k)_{li}&\dots & B(1,k)_{ln}\\ \vdots &\ddots&\vdots &\ddots& \vdots \\ B(n,k)_{l1}&\dots &
B(n,k)_{li}&\dots & B(n,k)_{ln}
\end{pmatrix}=0.
\end{equation}
On the other hand, By Remark~\ref{corolario con las B} we know that $(B(1,k),\dots,B(n,k))$ is a complete family of orthogonal idempotent matrices
of rank~$1$. But then, also $(B(1,k)^{\Trans},\dots,B(n,k)^{\Trans})$ is. Since $B(j,k)_{ll}=A(l,l)_{kj}\ne 0$ implies that the vector
$(B(j,k)_{l1},\dots,B(j,k)_{ln})$ generates $\ima(B(j,k)^{\Trans})$, the determinant of~\eqref{determinante} cannot be zero, a contradiction which
concludes the proof.
\end{proof}

\begin{theorem}\label{existencia, mas manejable} Let $\mathbf{v}_1, \dots, \mathbf{v}_n$ be $n$ invertible elements of $K^n$ with
$\mathbf{v}_1=1_{K^n}$, such that
$$
\det \begin{pmatrix} \mathbf{v}_1^{\Trans} & \dots & \mathbf{v}_n^{\Trans} \end{pmatrix} = 1.
$$
There exists a unique twisting map $\xi\colon K^n \ot K^n \longrightarrow K^n \ot K^n$ with
$$
A_{\xi}(i,l)\coloneqq (-1)^{i-1} (\mathbf{v}_l^{\centerdot} \centerdot \mathbf{v}_i)^{\Trans} (\mathbf{v}_l \centerdot (\mathbf{v}_1 \times \cdots
\times \widehat{\mathbf{v}_i} \times\cdots \times \mathbf{v}_n))\quad\text{for all $i,l$,}
$$
where, as usual, $\widehat{\mathbf{v}_i}$ means that the term $\mathbf{v}_i$ is omitted. Moreover, the twisted tensor product algebra $K^n\ot_\xi
K^n$ is isomorphic to $M_n(K)$.
\end{theorem}

\begin{proof}  We assert that the $A_{\xi}(i,j)$'s are idempotent matrices of rank~$1$ satisfying:

\begin{enumerate}

\smallskip

\item $A_{\xi}(i,o)A_{\xi}(j,o)=\delta_{ij}A_{\xi}(i,o)$,

\smallskip

\item $A_{\xi}(i,j) 1_{K^n}^{\Trans} = \delta_{ij}1_{K^n}^{\Trans}$,

\smallskip

\item $\sum_{i=1}^n A_{\xi}(i,o)= \ide$.

\smallskip

\end{enumerate}
In fact, since by Proposition~\ref{relacion entre el producto cruzado y el producto en K^n}
$$
\mathbf{v}_l \centerdot (\mathbf{v}_1 \times \cdots \times \widehat{\mathbf{v}_i} \times\cdots \times \mathbf{v}_n) = \tau(\mathbf{v}_l) \,
(\mathbf{v}_l^{\centerdot} \centerdot \mathbf{v}_1) \times \cdots \times  \widehat{(\mathbf{v}_l^{\centerdot} \centerdot \mathbf{v}_i)} \times
\cdots \times (\mathbf{v}_l^{\centerdot} \centerdot \mathbf{v}_n),
$$
we have
\begin{align*}
\bigl(\mathbf{v}_l \centerdot (\mathbf{v}_1 \times \cdots \times \widehat{\mathbf{v}_i} \times\cdots \times \mathbf{v}_n) \bigr)
(\mathbf{v}_l^{\centerdot} \centerdot \mathbf{v}_j)^{\Trans} &= \tau(\mathbf{v}_l)\det \begin{pmatrix} \mathbf{v}_l^{\centerdot} \centerdot
\mathbf{v}_j \\ \mathbf{v}_l^{\centerdot} \centerdot \mathbf{v}_1\\ \vdots  \\ \mathbf{v}_l^{\centerdot} \centerdot \mathbf{v}_{i-1}\\
\mathbf{v}_l^{\centerdot} \centerdot \mathbf{v}_{i+1} \\ \vdots \\ \mathbf{v}_l^{\centerdot} \centerdot \mathbf{v}_n\end{pmatrix}\\
&=(-1)^{i-1}\tau(\mathbf{v}_l)\delta_{ij}\det \begin{pmatrix} \mathbf{v}_l^{\centerdot} \centerdot \mathbf{v}_1\\ \vdots  \\
\mathbf{v}_l^{\centerdot} \centerdot \mathbf{v}_n \end{pmatrix}\\
&= (-1)^{i-1}\delta_{ij}\det \begin{pmatrix} \mathbf{v}_1\\ \vdots \\  \mathbf{v}_n\end{pmatrix}\\
&= (-1)^{i-1}\delta_{ij}.
\end{align*}
This implies that $A(i,l)$ is the idempotent with image $K(\mathbf{v}_l^{\centerdot} \centerdot \mathbf{v}_i)^{\Trans}$ and kernel $\langle
(\mathbf{v}_l^{\centerdot} \centerdot \mathbf{v}_j)^{\Trans}: j\ne i\rangle$, which implies items~(1), (2) y (3) (for (2) use that
$\mathbf{v}_l^{\centerdot}\centerdot \mathbf{v}_l=1_{K^n}$).

Now we consider the vectors $\mathbf{w_u}$ ($1\le i \le n$) determined by the equality
$$
\begin{pmatrix}\mathbf{w}_1\\  \vdots \\ \mathbf{w}_n \end{pmatrix} \coloneqq \begin{pmatrix} \mathbf{v}_1^{\Trans} & \dots &
\mathbf{v}_n^{\Trans}\end{pmatrix},
$$
and we define the matrices
$$
B_{\xi}(j,k)\coloneqq (-1)^{j-1} (\mathbf{w}_k^{\centerdot} \centerdot \mathbf{w}_j)^{\Trans} (\mathbf{w}_k\centerdot (\mathbf{w}_1\times \cdots
\times\widehat{\mathbf{w}_j} \times \dots \times \mathbf{w}_n))
$$
One checks that $A_{\xi}(i,l)_{kj}=B_{\xi}(j,k)_{li}$. Moreover, arguing as above for the $A_{\xi}(i,j)$'s, it can be proven that
$$
B_{\xi}(i,o)B_{\xi}(j,o)=\delta_{ij}B_{\xi}(i,o)\quad\text{for all $i,j,o$.}
$$
From this it follows immediately that the matrices $A_{\xi}(i,l)$ satisfy condition~(4) in Corollary~\ref{caracterizacion en terminos de las
A(i,j)s}, which finishes the proof of the existence of $\chi$. The uniqueness is clear, so it remains to prove that $K^n\ot_\xi K^n$ is isomorphic
to $M_n(K)$. By Remark~\ref{imagen de rho} for this it suffices to prove that for any $l$ and all $k,j$ there exists $i$ such that $A(i,l)_{jk}\ne
0$, since then the representation $\rho_l$ is a surjective morphism between two algebras of the same dimension, and hence is an isomorphism. So fix
$l$, $k$, $j$. From $\sum_i A(i,l)=\ide$ it follows that there exists $i$ such that $A(i,l)_{kk}\ne 0$. But then
$$
A(i,l)_{jk}=\frac{(v_i)_j}{(v_i)_k}A(i,l)_{kk}\ne 0,
$$
as desired.
\end{proof}

\begin{remark}\label{unicidad mejorada} The uniqueness part in Theorem~\ref{existencia, mas manejable} can be improved. If two twisting maps $\chi$
and $\check{\chi}$ with $\Gamma_{\chi}=\Gamma_{\check{\chi}}=\mathfrak{J}_n$ satisfy $A_{\chi}(i,l)=A_{\check{\chi}}(i,l)$ for a fixed~$l$ and
all~$i$, and  all the entries of $A_{\chi}(i,l)$ are non null, then $\chi=\check{\chi}$. The proof is left to the reader (use~\eqref{eq:rel entre
A(i,l) y B(j,k)}, Proposition~\ref{diagonales uno} and Remark~\ref{rango uno generalizado}).
\end{remark}

\section{Standard and quasi-standard columns}

\begin{definition} The {\em support} of a matrix $A\in M_n(K)$ is the set
$$
\Supp(A)\coloneqq \{(i,j)\in \mathds{N}_{n}^*\times\mathds{N}_n^* : a_{ij}\ne 0\},
$$
and the support of the $k$-th row of $A$ is the set $\Supp(A_{k*})\coloneqq \{j\in \mathds{N} : a_{kj}\ne 0\}$.
\end{definition}

\begin{definition}\label{pretwisting} A family $(A(i,l))_{i,l\in \mathds{N}_m^*}$ of matrices $A(i,l)\in M_n(K)$, is called a {\em pre-twisting} of
$K^m$ with $K^n$ if it satisfies conditions~(1), (2) and~(3) of Corollary~\ref{caracterizacion en terminos de las A(i,j)s}.
\end{definition}

Throughout this section $\mathcal{A}=\left(A(i,l)\right)_{i,l\in \mathds{N}^*_m}$ denotes a pre-twisting of $K^m$ with $K^n$.

\begin{definition}\label{definicion columna estandar} We say that the $l_0$-th column of $\mathcal{A}$ is a {\em standard column} if
\begin{enumerate}

\smallskip

\item $A(l_0,l_0)$ is a $0,1$-matrix,

\smallskip

\item $\Supp(A(i,l_0))\subseteq\Supp(A(l_0,l_0))\cup \Supp(\ide)$ for all $i$.

\end{enumerate}
\end{definition}

\begin{remark}\label{Como son las columnas estandar} Assume that $\left(A(i,l_0)\right)_{i\in \mathds{N}^*_m}$ is a standard column of
$\mathcal{A}$ and let $k\in \mathds{N}_n^*$. The following facts hold:
\begin{enumerate}

\smallskip

\item For each index $i$, we have $A(i,l_0)_{kk}\in\{0,1\}$.

\smallskip

\item $A(i,l_0)_{kk}=1$ for exactly one $i$. We let $i(k) = i(k,l_0)$ denote this index.

\smallskip

\item If $i\ne i(k)$ and $i\ne l_0$, then $A(i,l_0)_{kj}=0$ for all $j$.

\smallskip

\item $A(i,l_0)_{kj} = -1$ if and only if $i= i(k) \ne l_0$ and $j=c_k(A(l_0,l_0))$. Moreover $A(i,l_0)_{kj'}=0$ for all
$j'\notin\{k,c_k(A(l_0,l_0))\}$.

\smallskip

\item $A(i,l_0)_{kj}\in \{1,0,-1\}$ for all $i$, $k$, $j$, and $A(i,l_0)_{kj}=1$ implies $i=l_0$ or $j=k$.

\end{enumerate}
\end{remark}

\begin{remark} From Remark~\ref{Como son las columnas estandar} it follows that each standard column $A(i,l_0)_{i\in \mathds{N}_m^*}$ of a
pre-twisting of $K^m$ with $K^n$ can be obtained in the following way:
\begin{enumerate}

\smallskip

\item Take a matrix $A\in M_n(K)$, which is equivalent via identical permutations in rows and columns to a standard idempotent $0,1$-matrix, and
set $A(l_0,l_0)\coloneqq A$.

\smallskip

\item Set $J_{l_0}\coloneqq \{k\in \mathds{N}_n^*: A(l_0,l_0)_{kk}=1\}$.

\smallskip

\item For all $i\in\mathds{N}_m^*\setminus\{l_0\}$ choose $J_i\subseteq\mathds{N}_n^*\setminus J_{l_0}$ such that
$$
\bigcup_{i=1}^m J_i=\mathds{N}_n^*\quad\text{and}\quad J_i\cap J_{i'}=\emptyset\quad\text{if $i\ne i'$}.
$$

\smallskip

\item For $i\ne l_0$ define $A(i,l_0)\in M_n(K)$ by
$$
A(i,l_0)_{kj}\coloneqq \begin{cases} \phantom{-} 1  & \text{if $k\in J_i$ and $j=k$,}\\-1  & \text{if $k\in J_i$ and $j=c_k$,}\\ \phantom{-} 0&
\text{otherwise.}
\end{cases}
$$
\end{enumerate}
\end{remark}

Next we generalize the notation introduced in Remark~\ref{Como son las columnas estandar}(2).

\begin{remark}\label{suma en la diagonal} Let $l_0\in\mathds{N}_m^*$ and $k\in \mathds{N}_n^*$. If $A(i,l_0)_{kk}\in\{0,1\}$ for all~$i$, then there is a unique index $i_0$, which is denoted $i(k)=i(k,l_0)=i(k,l_0,\mathcal{A})$, such that $A(i_0,l_0)_{kk}=1$. So, $A(i,l_0)_{kk}=\delta_{ii_0}$.
\end{remark}

\begin{definition}\label{twisting standard} We say that a twisting map $\chi\colon K^m\ot K^n\longrightarrow K^n\ot K^m$ is {\em standard} if the
columns of $\mathcal{A}_{\chi}$ are standard columns. In this case we also say that the twisted tensor product $K^n\ot_{\chi} K^m$ is {\em
standard}.
\end{definition}

\begin{proposition}\label{A standard equivale a B standard} A twisting map $\chi$ is a standard twisting map if and only if the map $\tilde{\chi}$,
in\-tro\-duced in Remark~\ref{A sub Chi es B sub tau chi tau}, is.
\end{proposition}

\begin{proof} By Remark~\ref{A sub Chi es B sub tau chi tau} we know that $\mathcal{A}_{\tilde{\chi}} = \mathcal{B}_{\chi}$. Thus, since
$\tilde{\chi}$ is a twisting map, we only must check that the $l_0$-th column of $\mathcal{B}_{\chi}$ is a standard column for all $l_0\in
\mathds{N}^*_n$. Item~(1) of Definition~\ref{definicion columna estandar} is an immediate consequence of Remark~\ref{Como son las columnas
estandar}(1). For item~(2) it suffices to consider the case $i\ne l_0$. By Remark~\ref{Como son las columnas estandar}(4), we know that
$B_{\chi}(i,l_0)_{kj}\in\{1,0,-1\}$ for all $j,k$ and that $B_{\chi}(i,l_0)_{kj}\ne 1$ for $j\ne k$. Since $\sum_{j=1}^m B_{\chi}(i,l_0)_{kj}=0$,
this implies that if $B_{\chi}(i,l_0)_{kk}=0$, then the $k$-th row vanishes. Else $B_{\chi}(i,l_0)_{kk}=1$ and there exists exactly one index $j'$
such that $B_{\chi}(i,l_0)_{kj'}=-1$. It remains to check that $j'=c_k(B_{\chi}(l_0,l_0))$. Using that $B_{\chi}(i,l_0)$ is idempotent, we obtain
that
$$
-1 = B_{\chi}(i,l_0)_{kj'}=\sum_{j=1}^m B_{\chi}(i,l_0)_{kj}B_{\chi}(i,l_0)_{jj'}=B_{\chi}(i,l_0)_{kj'}-B_{\chi}(i,l_0)_{j'j'}= -1-B_{\chi}(i,l_0)_{j'j'}.
$$
Set $i_0\coloneqq i\bigl(j',l_0,\mathcal{A}_{\tilde{\chi}}\bigr)$. Since $B_{\chi}(i_0,l_0)_{j'j'}=1$ the above equality implies that $i\ne i_0$. Thus,
$$
0=\sum_{j=1}^m B_{\chi}(i,l_0)_{kj}B_{\chi}(i_0,l_0)_{jj'}=B_{\chi}(i_0,l_0)_{kj'}-B_{\chi}(i_0,l_0)_{j'j'}=B_{\chi}(i_0,l_0)_{kj'}-1,
$$
where the first equality holds because $B_{\chi}(i,l_0)B_{\chi}(i_0,l_0)=0$. Therefore $B_{\chi}(i_0,l_0)_{kj'}=1$, and so $i_0=l_0$, because $j'\ne k$. Hence, $j'=c_k(B_{\chi}(l_0,l_0))$, as desired.
\end{proof}

\begin{remark}\label{condicion para twisting standard}
Let $\chi$ be a standard twisting map and let $i\ne l$ and $k\ne j$. Then $A_{\chi}(i,l)_{kj} = -1$ if and only if $B_{\chi}(k,k)_{li} = 1$ and
$A_{\chi}(l,l)_{kj} = 1$. In fact, by Remark~\ref{Como son las columnas estandar}(4),
$$
A_{\chi}(i,l)_{kj} = -1 \Rightarrow B_{\chi}(k,k)_{li} = A_{\chi}(i,l)_{kk} = 1.
$$
Since, by Proposition~\ref{A standard equivale a B standard} and Remark~\ref{A sub Chi es B sub tau chi tau} we know that the map $\tilde{\chi}$ is
a standard twisting map and $\mathcal{A}_{\tilde{\chi}} = (B_{\chi}(i,l))_{i,l\in \mathds{N}_n^*}$, we also have $A_{\chi}(l,l)_{kj} = 1$.
Conversely,
$$
1 = B_{\chi}(k,k)_{li} = A_{\chi}(i,l)_{kk} \Rightarrow \exists ! j \text{ such that } A_{\chi}(i,l)_{kj} = -1.
$$
So $j = c_k(A_{\chi}(l,l))$.
\end{remark}

\begin{theorem}\label{A y B determina el stm} Let $(A(i))_{i\in N^*_m}$ and $(B(k))_{k\in N^*_n}$ be two families of idempotent $0,1$-matrices
$A(i)\in M_n(K)$ and $B(k)\in M_m(K)$, such that, for all $i$ and $k$,

\begin{enumerate}

\smallskip

\item $A(i)\mathds{1} = \mathds{1}$ and $B(k)\mathds{1} = \mathds{1}$,

\smallskip

\item $A(i)_{kk} = B(k)_{ii}$.

\smallskip

\end{enumerate}
The family $\mathcal{A}_{\chi}=(A_{\chi}(i,l))_{i,l\in \mathds{N}^*_m}$, of matrices $A_{\chi}(i,l)\in M_n(K)$ defined by
$$
A_{\chi}(i,l)_{kj}\coloneqq \begin{cases} A(l)_{kj} & \text{if $i=l$,}\\ B(k)_{li} & \text{if $k=j$,}\\ -1 & \text{if $i\ne l$, $k\ne j$ and
$A(l)_{kj}=B(k)_{li}= 1$,}\\ \phantom{-} 0 &\text{otherwise,}\end{cases}
$$
gives the unique standard twisting map
$$
\chi\colon K^m\ot K^n\longrightarrow K^n\ot K^m,
$$
such that $A_{\chi}(i,i)=A(i)$ and $B_{\chi}(k,k)=B(k)$.
\end{theorem}

\begin{proof} The uniqueness holds since the definition of $\mathcal{A}_{\chi}$ is forced. Set $B_{\chi}(j,k)_{li}\coloneqq A_{\chi}(i,l)_{kj}$.
Note that $B_{\chi}(k,k)= B(k)$. We must check that conditions (1)--(4) of Proposition~\ref{caracterizacion} are fulfilled and that $\chi$ is
standard. For condition~(3) we must verify that

\begin{equation}\label{A cumple condicion 3}
\delta_{il} = \sum_j A_{\chi}(i,l)_{kj}\qquad\text{for all $i$, $l$ and $k$.}
\end{equation}
When $i=l$ this is true by assumption. When $i\ne l$ and $B(k)_{li} = 0$, we have $A_{\chi}(i,l)_{kj} = 0$ for all~$j$, and thus equality~\eqref{A
cumple condicion 3} is true. Finally, when~$i\ne l$ and~$B(k)_{li} = 1$, we have~$A_{\chi}(i,l)_{kk} = 1$, $A_{\chi}(i,l)_{kc_k} = -1$ (where $c_k
= c_k(A(l))$) and $A_{\chi}(i,l)_{kj}=0$ for $j\notin\{k,c_k\}$, and again equality~\eqref{A cumple condicion 3} is true. The proof of
condition~(4) is similar. Since $B_{\chi}(j,k)_{li} = A_{\chi}(i,l)_{kj}$, conditions~(3) and~(4) say that $\sum_i A_{\chi}(i,l) = \ide$ and
$\sum_j B_{\chi}(j,k) = \ide$ for all $l$ and for all $k$. Hence, by Remark~\ref{cuando son idempotentes}, in order to check condition~(1) it
suffices to prove that
\begin{equation}\label{es menor o igual que n}
\sum_i \rk\bigl(A_{\chi}(i,l)\bigr) \le n\qquad\text{for all $l$.}
\end{equation}
Fix $l\in \mathds{N}_m^*$. Since the $B(k)$'s are equivalent, via identical permutations in rows and columns, to a standard idempotent
$0,1$-matrices, we know that for each $k$ there exists a unique $i$ such that $A_{\chi}(i,l)_{kk} = B(k)_{li} = 1$. Thus $\sum_i \# \{k:
A_{\chi}(i,l)_{kk} = 1\} = n$. Consequently, to conclude that inequality~\eqref{es menor o igual que n} holds it is enough to show that
$$
\rk\bigl(A_{\chi}(i,l)\bigr)\le \# \bigl\{k: A_{\chi}(i,l)_{kk} = 1\bigr\}\qquad\text{for all $i$.}
$$
But, for $i=l$ we know that $\rk\bigl(A_{\chi}(l,l)\bigr)= \# \bigl\{k: A_{\chi}(l,l)_{kk} = 1\bigr\}$, because $A(l)$ is an idempotent
$0,1$-matrix, while, for $i\ne l$, from the fact that
$$
A_{\chi}(i,l)_{kk}\in \{0,1\}\quad\text{and}\quad A_{\chi}(i,l)_{kk} = 0 \text{ implies that } A_{\chi}(i,l)_{kj} = 0\text{ for all $j$,}
$$
it follows that $\# \bigl\{k: A_{\chi}(i,l)_{kk} = 1\bigr\}$ is the number of non zero rows of $A_{\chi}(i,l)$, which is greater than or equal to
$\rk\bigl(A_{\chi}(i,l)\bigr)$. This concludes the proof of condition~(1) of Proposition~\ref{caracterizacion}. The proof of condition~(2) is
similar.
\end{proof}

\begin{notation} For all $l\in \mathds{N}_m^*$, we set
$$
F_0(\mathcal{A},l)\coloneqq \{k\in\mathds{N}^*_n : A(i,l)_{kj}=\delta_{il}\delta_{kj}, \text{ for all $i$ and $j$}\}.
$$
and for all $i,l\in \mathds{N}_m^*$, we set $F(A(i,l))\coloneqq \{j\in\mathds{N}^*_n  : A(i,l)_{jj}=1\}$.
\end{notation}

\begin{remark} The set $F(A(i,l))$ was introduced in Notation~\ref{j sub i(l)}, where was denoted $J_i(l)$, but in some places we prefer to use the
longer but more precise notation $F(A(i,l))$.
\end{remark}

\begin{definition} We will say that Corollary~\ref{caracterizacion en terminos de las A(i,j)s}(4) is satisfied in the $l_0$-th column of
$\mathcal{A}$ if
\begin{equation}\label{condicion 4 rr2}
\sum_{h=1}^m A(i,h)_{kj}A(h,l_0)_{kj'}=\delta_{jj'}A(i,l_0)_{kj}\qquad\text{for all $i$, $j$, $j'$ and $k$.}
\end{equation}
\end{definition}

\begin{proposition}\label{influencia de columna estandar} If the $l_0$-th column of $\mathcal{A}$ is a standard column, then
Corollary~\ref{caracterizacion en terminos de las A(i,j)s}(4) is satisfied in the $l_0$-th column of $\mathcal{A}$ if and only if
$F(A(v,l_0))\subseteq F_0(\mathcal{A},v)$ for all~$v\in\mathds{N}^*_m$.
\end{proposition}

\begin{proof} \noindent $\Rightarrow$)\enspace Let $v\in\mathds{N}^*_m$ and $k\in\mathds{N}^*_n$. If $k\in F(A(v,l_0))$, then
$A(u,l_0)_{kk}=\delta_{uv}$ for all $u\in\mathds{N}^*_m$ (see Remark~\ref{Como son las columnas estandar}). So, from~\eqref{condicion 4 rr2} with
$j=k$, we obtain that
$$
A(i,v)_{kj}=\sum_{u=1}^m A(i,u)_{kj}A(u,l_0)_{kk}=\delta_{jk}A(i,l_0)_{kj}=\delta_{jk}A(i,l_0)_{kk}=\delta_{jk}\delta_{iv}
$$
for all $i,j$, which says that $k\in F_0(\mathcal{A},v)$, as desired.

\smallskip

\noindent $\Leftarrow$)\enspace Fix $k\in \mathds{N}_n^*$. If $i(k,l_0)=l_0$, then $k \in F(A(l_0,l_0)) \subseteq F_0(\mathcal{A},l_0)$, and so
condition~\eqref{condicion 4 rr2} holds if and only~if
$$
A(i,l_0)_{kj} \delta_{kj'} = \delta_{jj'}A(i,l_0)_{kj}\qquad\text{for all $i$, $j$ and $j'$.}
$$
But this is true for $i\ne l_0$, since then $A(i,l_0)_{kj}=0$, and also for $i=l_0$, since $A(l_0,l_0)_{kj}=\delta_{kj}$.

If $h_0\coloneqq i(k,l_0)\ne l_0$, then equality~\eqref{condicion 4 rr2} holds if and only if
\begin{equation}\label{35 para l distinto a i0}
A(i,h_0)_{kj}A(h_0,l_0)_{kj'}+A(i,l_0)_{kj}A(l_0,l_0)_{kj'}=\delta_{jj'}A(i,l_0)_{kj}\qquad\text{for all $i$, $j$ and $j'$,}
\end{equation}
since for $h\notin\{h_0,l_0\}$ we have $A(h,l_0)_{kj'}=0$ for all $j'$. In order to prove that~\eqref{35 para l distinto a i0} is true, we
con\-sider the cases $j=k$, $j=c_k=c_k(A(l_0,l_0))$ and $j\notin \{k,c_k\}$. We will use that $A(i,h_0)_{kj}=\delta_{ih_0}\delta_{kj}$ for all
$i,j$, which is true, because $k\in F(A(h_0,l_0))\subseteq F_0(\mathcal{A},h_0)$.
\begin{itemize}

\smallskip

\item[-] If $j=k$, then we must prove that
$$
\qquad\qquad A(i,h_0)_{kk}A(h_0,l_0)_{kj'}+A(i,l_0)_{kk}A(l_0,l_0)_{kj'}=\delta_{kj'}A(i,l_0)_{kk}\quad\text{for all $i$ and all $j'$.}
$$
But this is true, since by the above discussion, Remark~\ref{Como son las columnas estandar} and Proposition~\ref{01 matrix},
$$
\qquad\qquad A(i,h_0)_{kk}\!=\delta_{ih_0},\!\!\quad A(h_0,l_0)_{kj'}\!=\delta_{kj'}-\delta_{j'c_k},\!\!\quad A(i,l_0)_{kk}\!
=\delta_{ih_0}\!\!\quad \text{and}\!\!\quad A(l_0,l_0)_{kj'}\!=\delta_{j'c_k}.
$$

\smallskip

\item[-] Since $A(i,h_0)_{kc_k}=0$ for all $i$, when $j=c_k$ we are reduced to prove that
$$
A(i,l_0)_{kc_k}A(l_0,l_0)_{kj'}=\delta_{c_kj'}A(i,l_0)_{kc_k}\quad\text{for all $i$ and all $j'$.}
$$
But this is true, since $A(l_0,l_0)_{kj'}=\delta_{j'c_k}$.

\smallskip

\item[-] If $j\notin\{k,c_k\}$, then both sides of~\eqref{35 para l distinto a i0} vanish.

\smallskip

\end{itemize}
Thus,~\eqref{condicion 4 rr2} holds in all the cases.
\end{proof}

\begin{corollary} Let $\chi\colon K^m\ot K^n\longrightarrow K^n\ot K^m$ be a $k$-linear map such that $\mathcal{A}_{\chi}$ is a pre-twisting. If
each column of $\mathcal{A}_{\chi}$ is standard, then $\chi$ is a twisting map if and only if $F(A(i,l))\subseteq F_0(A,i)$ for all $i,l\in
\mathds{N}_m^*$.
\end{corollary}

Given sets $X,Y$, in the sequel we let $M_{X,Y}(K)$ denote the set of functions from $X\times Y$ to $K$. We also denote by $\ide_X$ the identity
matrix in $M_X(K)\coloneqq M_{X,X}(K)$.

\begin{proposition}\label{reduccion al cuadrado inferior} Let $l\in\mathds{N}_k^*$ and  let $A(1),\dots, A(k)\in M_n(K)$ be matrices such that
$A(l)$ is an idempotent $0,1$-matrix with $A(l)\mathds{1}=\mathds{1}$. Set $J_l\coloneqq \{k : A(l)_{kk}=1\}$ and  $J_l^c\coloneqq
\mathds{N}_n^*\setminus J_l$. For each $i$ set
$$
X_i\coloneqq A(i)|_{J_l\times J_l},\quad Y_i\coloneqq A(i)|_{J_l\times J_l^c},\quad U_i\coloneqq A(i)|_{J_l^c\times J_l} \quad\text{and}\quad W_i\coloneqq A(i)|_{J_l^c\times J_l^c}.
$$
The matrices $A(i)$'s are orthogonal idempotents satisfying $\sum_i A(i)=\ide$ if and only if the following facts hold:
\begin{enumerate}

\smallskip

\item $X_i=0$ for all $i\ne l$,

\smallskip

\item  $Y_i=0$ for all $i$,

\smallskip

\item $W_iW_j=\delta_{ij}W_i$ for all $i$,

\smallskip

\item $U_i=-W_i U_l$ for all $i\ne l$,

\smallskip

\item $\sum_i W_i=\ide_{J_l^c}$.

\smallskip

\end{enumerate}
Moreover, if the $A(i)$'s satisfy the required conditions, then $A(i)\mathds{1}=\delta_{il}\mathds{1}$.
\end{proposition}

\begin{proof} Without loss of generality we can assume that $J_l=\mathds{N}_r^*$, where $r\coloneqq \rk(A(l))$. Then
$A(l)=\begin{psmallmatrix}\ide_r&0\\ U_l& 0 \end{psmallmatrix}$ and $A(i)=\begin{psmallmatrix}X_i&Y_i\\ U_i& W_i \end{psmallmatrix}$. Let $i\ne l$.
A direct computation shows that $A(l)A(i)=0$ if and only if $X_i=0$ and $Y_i=0$. Under this condition, $A(i)A(l)=0$ if and only if $U_i=-W_i U_l$.
Assuming all the previous conditions for all $i\ne l$, we have $A(i)A(j)=\delta_{ij}A(i)$ if and only if $W_iW_{j}=\delta_{ij}W_i$, and, under the
same conditions, $\sum_i A(i)=\ide_{n}$ if and only if $\sum_i W_i=\ide_{J_l^c}$. The last assertion follows from the fact that $U_i=-W_i U_l$ and
$U_l\mathds{1}_{J_l}=\mathds{1}_{J_l^c}$.
\end{proof}

\begin{definition}\label{quasi standard} Let $l_0\in\mathds{N}_m^*$. For all $i,u,v\in \mathds{N}_m^*$, set $D_{(i,l_0)}^{uv}=D_{(i)}^{uv}\coloneqq
A(i,l_0)|_{J_u\times J_v}$, where $J_i\coloneqq J_i(l_0)$. We say that $(A(i,l_0))_{i\in \mathds{N}_m^*}$ is a {\em quasi-standard column} of
$\mathcal{A}$ if

\begin{enumerate}

\smallskip

\item $A(l_0,l_0)$ is a $0,1$-matrix,

\smallskip

\item $A(i,l_0)_{kk}\in \{0,1\}$ for all $i$ and $k$,

\smallskip

\item $D_{(i)}^{uv} = 0$ if $u\ne i$ and $v\notin \{i,l_0\}$,

\smallskip

\item For $u,i\in\mathds{N}_m^*$, $v\in \mathds{N}_m^*\setminus\{l_0\}$ and $k\in J_u$, we have $\#\Supp\bigl(\bigl(D_{(i)}^{uv}\bigr)_{k*}\bigr)
\le 1$. Moreover if $d\in \Supp\bigl((D_{(i)}^{uv})_{k*}\bigr)$, then $c_d=c_k$, where $c_d\coloneqq c_d(A(l_0,l_0))$ and $c_k\coloneqq
c_k(A(l_0,l_0))$\footnote{Note that $c_k$ exists since necessarily $u\ne l_0$ (see Remark~\ref{columnas que se anulan en l_0} and the beginning
of Remark~\ref{simplificacion})}. If necessary we will write $d^{(v)}$ or $d_k^{(v)}$ instead of $d$.

\end{enumerate}
\end{definition}

\begin{remark}\label{columnas que se anulan en l_0} Let $k\in J_{l_0}$ and let $i\ne l_0$. By items~(1) and~(2) of Proposition~\ref{reduccion al
cuadrado inferior} we know that $A(i,l_0)_{kj} = 0$ for all~$j$. Consequently $D^{l_0v}_{(i)} = 0$ for all $v\in \mathds{N}_m^*$. Note that this
implies that $F(A(l_0,l_0)) = F_0(\mathcal{A},l_0)$.
\end{remark}

\begin{remark}\label{identidad en la diagonal} Since $\sum_i A(i,l_0) = \ide$, we have $\sum_i D^{uu}_{(i)} = \ide$ for all $u\in \mathds{N}_m^*$,
which by condition~(3) implies that $D^{uu}_{(u)} = \ide$ for all $u\ne l_0$ (by Proposition~\ref{01 matrix}, also $D^{l_0l_0}_{(l_0)} = \ide$).
\end{remark}

\begin{remark}\label{cero fuera} Since $\sum_i A(i,l_0)\! =\! \ide$, we have $\sum_i D^{uv}_{(i)}\! =\! 0$ for all $u\ne v$ in $\mathds{N}_m^*$,
which by con\-dition~(3) implies that $D^{uv}_{(u)} = - D^{uv}_{(v)}$ for all $u\in\mathds{N}_m^*$ and $v\in\mathds{N}_m^*\setminus\{u,l_0\}$.
\end{remark}

Remark~\ref{columnas que se anulan en l_0} is valid for pre-twistings that satisfy condition~(1) of Definition~\ref{quasi standard}, while
Remarks~\ref{identidad en la diagonal} and~\ref{cero fuera} are true for pre-twistings that satisfy conditions (1) and (3) of the same definition.

\begin{remark}\label{simplificacion} From the fact that $A(l_0,l_0)$ is a $0,1$-matrix it follows immediately that $D_{(l_0)}^{uv} = 0$ for all $u
\in \mathds{N}_m^*$ and $v\in \mathds{N}_m^*\setminus\{l_0\}$. Combining this with Remarks~\ref{columnas que se anulan en l_0} and~\ref{cero fuera}
we obtain that Conditions~(3) and~(4) in Definition~\ref{quasi standard} could be replaced by

\begin{enumerate}

\item[(3')] $D_{(i)}^{uv} = 0$ if $i\ne l_0$ and $u,v\notin\{i,l_0\}$,

\smallskip

\item[(4')] $\#\Supp\bigl(\bigl(D_{(v)}^{uv}\bigr)_{k*}\bigr) \le 1$ for $u,v\in \mathds{N}_m^*\setminus \{l_0\}$ and $k\in J_u$. Moreover if
$d\!\in \Supp\bigl((D_{(v)}^{uv})_{k*}\bigr)$, then $c_d=c_k$, where $c_d\coloneqq c_d(A(l_0,l_0))$ and $c_k\coloneqq c_k(A(l_0,l_0))$,

\end{enumerate}
respectively.
\end{remark}

\begin{remark}\label{estandar implica cuasi estandar} Each standard column of $\mathcal{A}$ is a quasi-standard column of $\mathcal{A}$.
\end{remark}

\begin{example} Assume for example that $n=10$, $J_{l_0}=\{1,2\}$ and $J_i=\{5,6,7\}$. If the $l_0$-th column of $\mathcal{A}$ is a quasi-standard
column, then the only entries where the matrix $A (i, l_0)$ may have nonzero values are the entries indicated by stars. In this example and in
Example~\ref{ejemplo de columna cuasiestandar} below, the elements of each family $J_u$ are consecutive, but of course this need not be the case.
\begin{center}
\begin{tikzpicture}
\begin{scope}
\draw (-2.8,-0.2)  node {$J_i$}; \draw (-2.8,1.65)  node {$J_{l_0}$}; \draw (0.1,2.4)  node {$J_i$}; \draw (-1.7,2.4)  node {$J_{l_0}$}; \draw (-4,0)  node {$A(i,l_0) =$};
\end{scope}
\begin{scope}
\matrix [matrix of math nodes,left delimiter=(, right delimiter=),nodes in empty cells,every node/.style={anchor=base,text depth=.1ex,text height=1ex,text width=0.5em}] (m)
{
0 & 0 & 0 & 0 & 0 & 0 & 0 & 0 & 0 & 0\\
0 & 0 & 0 & 0 & 0 & 0 & 0 & 0 & 0 & 0\\
$\star$ & $\star$ & 0 & 0 & $\star$ & $\star$ & $\star$ & 0 & 0 & 0\\
$\star$ & $\star$ & 0 & 0 & $\star$ & $\star$ & $\star$ & 0 & 0 & 0\\
$\star$ & $\star$ & $\star$ & $\star$ & $\star$ & $\star$ & $\star$ & $\star$ & $\star$ & $\star$\\
$\star$ & $\star$ & $\star$ & $\star$ & $\star$ & $\star$ & $\star$ & $\star$ & $\star$ & $\star$\\
$\star$ & $\star$ & $\star$ & $\star$ & $\star$ & $\star$ & $\star$ & $\star$ & $\star$ & $\star$\\
$\star$ & $\star$ & 0 & 0 & $\star$ & $\star$ & $\star$ & 0 & 0 & 0\\
$\star$ & $\star$ & 0 & 0 & $\star$ & $\star$ & $\star$ & 0 & 0 & 0\\
$\star$ & $\star$ & 0 & 0 & $\star$ & $\star$ & $\star$ & 0 & 0 & 0\\
};
%
%
%
%
\end{scope}
     \end{tikzpicture}
\end{center}
\end{example}

\begin{lemma}\label{D distinto de cero} Assume that the $l_0$-th column of $\mathcal{A}$ satisfies conditions~(1)--(3) of Definition~\ref{quasi standard}. Take $i,u \in \mathds{N}_m^*\setminus\{l_0\}$ and $k\in J_u$. If $A(i,l_0)_{kc_k} \ne 0$, then there exist indices $v\in \mathds{N}_m^*\setminus \{l_0\}$ and $j\in J_v$ such that $(D_{(i)}^{uv})_{kj}\ne 0$. Moreover, if $u\ne i$, then necessarily~$v=i$.
\end{lemma}

\begin{proof} By Remark~\ref{columnas que se anulan en l_0} we know that $A(i,l_0)_{jc_k}=0$ for all $j\in J_{l_0}$. So
$$
\sum_{v\in \mathds{N}_m^*\setminus \{l_0\} } \sum_{j\in J_v}  A(i,l_0)_{kj}  A(i,l_0)_{jc_k}= \sum_{j\in \mathds{N}_n^*}  A(i,l_0)_{kj}  A(i,l_0)_{jc_k} = A(i,l_0)_{kc_k} \ne 0.
$$
Consequently, there exists $v\in \mathds{N}_m^*\setminus \{l_0\}$ and $j\in J_v$ such that $(D_{(i)}^{uv})_{kj}=A(i,l_0)_{kj}\ne 0$. The last
assertion is true by item~(3) of Definition~\ref{quasi standard}.
\end{proof}

For $u\in \mathds{N}_m^*\setminus\{l_0\}$ and $k\in J_u=J_u(l_0)$, we set
$$
\mathscr{X}_k\coloneqq \bigl\{v\in \mathds{N}_m^*\setminus\{u,l_0\} :\Supp\bigl((D_{(u)}^{uv})_{k*}\bigr)\ne\emptyset\bigr\}\qquad\text{and}\qquad
d^{(\mathscr{X}_k)}\coloneqq \{d^{(v)}: v\in \mathscr{X}_k\}.
$$

\begin{lemma}\label{condicion} Assume that the $l_0$-th column of $\mathcal{A}$ is quasi-standard. For each $k\in \mathds{N}_n^*\setminus J_{l_0}$
and $v\in \mathds{N}_m^*$, we have $\Supp\bigl(A(v,l_0)_{k*}\bigr)\subseteq \{k,c_k\}\cup d^{(\mathscr{X}_k)}$.

\end{lemma}

\begin{proof} When $v=l_0$ this is clear. So, we can assume that $v\ne l_0$. Let $u\coloneqq i(k,l_0)$. We consider two cases:

\smallskip

\noindent $u\ne v$)\enspace By the very definition of quasi-standard column and Remark~\ref{cero fuera}
$$
\Supp\bigl(A(v,l_0)_{k*}\bigr) \subseteq \Supp\bigl((D_{(v)}^{ul_0})_{k*}\bigr)\cup \Supp\bigl((D_{(v)}^{uv})_{k*}\bigr)\qquad\text{and}\qquad
\Supp\bigl((D_{(v)}^{uv})_{k*}\bigr)\subseteq d^{(\mathscr{X}_k)}.
$$
Hence it suffice to prove that $\Supp\bigl((D_{(v)}^{ul_0})_{k*}\bigr) \subseteq \{c_k\}$. Since $D^{ui}_{(v)} = 0$ for $i\notin\{v,l_0\}$,
$$
D^{ul_0}_{(v)}D^{l_0l_0}_{(l_0)}+D^{uv}_{(v)}D^{vl_0}_{(l_0)}= A(v,l_0)A(l_0,l_0)|_{J_u\times J_{l_0}} = 0.
$$
Since $D^{l_0l_0}_{(l_0)} = \ide$, this yields
$$
D^{ul_0}_{(v)}= - D^{uv}_{(v)}D^{vl_0}_{(l_0)}.
$$
Thus, if $\Supp((D^{uv}_{(v)})_{k*})=\emptyset$, then $\Supp((D^{ul_0}_{(v)})_{k*})=\emptyset$. Else $\Supp((D^{uv}_{(v)})_{k*})=\{d^{(v)}\}$ and so
$$
\bigl(D^{ul_0}_{(v)}\bigr)_{k*} = - \bigl(D^{uv}_{(v)}\bigr)_{kd^{(v)}}\bigl(D^{vl_0}_{(l_0)}\bigr)_{d^{(v)}*}.
$$
Combining this with the fact that
$$
\Supp \bigl(\bigl(D^{vl_0}_{(l_0)}\bigr)_{d^{(v)}*}\bigr) = \Supp \bigl(A(l_0,l_0)_{d^{(v)}*}\bigr) = \{c_{d^{(v)}}\} = \{c_k\},
$$
we obtain that $\Supp \bigl(\bigl(D^{ul_0}_{(v)}\bigr)_{k*}\bigr) = \{c_k\}$, as desired.

\smallskip

\noindent $u=v$)\enspace Using that $A(u,l_0)_{k*} = \delta_{k*}-\sum_{i\ne u} A(i,l_0)_{k*}$  we obtain that
$$
\Supp\bigl(A(u,l_0)_{k*}\bigr)\subseteq \{k\}\cup \bigcup_{i\ne u} \Supp\bigl(A(i,l_0)_{k*}\bigr),
$$
which finishes the proof.
\end{proof}

\begin{example}\label{ejemplo de columna cuasiestandar}
The matrices
\begin{equation*}
A(1,1)\coloneqq \begin{pmatrix}
\textcolor{myblue}{1} & \textcolor{myblue}{0} & \textcolor{myviolet}{0} & \textcolor{myviolet}{0} & \textcolor{myviolet}{0} & \textcolor{green}{0} &
\textcolor{green}{0} & \textcolor{green}{0} \\
\textcolor{myblue}{0} & \textcolor{myblue}{1} & \textcolor{myviolet}{0} & \textcolor{myviolet}{0} & \textcolor{myviolet}{0} & \textcolor{green}{0} &
\textcolor{green}{0} & \textcolor{green}{0} \\
\textcolor{brown}{1} & \textcolor{brown}{0} & \textcolor{magenta}{0} & \textcolor{magenta}{0} & \textcolor{magenta}{0} & \textcolor{myred}{0} &
\textcolor{myred}{0} & \textcolor{myred}{0} \\
\textcolor{brown}{1} & \textcolor{brown}{0} & \textcolor{magenta}{0} & \textcolor{magenta}{0} & \textcolor{magenta}{0} & \textcolor{myred}{0} &
\textcolor{myred}{0} & \textcolor{myred}{0} \\
\textcolor{brown}{0} & \textcolor{brown}{1} & \textcolor{magenta}{0} & \textcolor{magenta}{0} & \textcolor{magenta}{0} & \textcolor{myred}{0} &
\textcolor{myred}{0} & \textcolor{myred}{0} \\
\textcolor{mydarkgray}{1} & \textcolor{mydarkgray}{0} & \textcolor{mycyan}{0} & \textcolor{mycyan}{0} & \textcolor{mycyan}{0} & \textcolor{myyellow}{0} &
\textcolor{myyellow}{0} & \textcolor{myyellow}{0} \\
\textcolor{mydarkgray}{0} & \textcolor{mydarkgray}{1} & \textcolor{mycyan}{0} & \textcolor{mycyan}{0} & \textcolor{mycyan}{0} & \textcolor{myyellow}{0} &
\textcolor{myyellow}{0} & \textcolor{myyellow}{0} \\
\textcolor{mydarkgray}{0} & \textcolor{mydarkgray}{1} & \textcolor{mycyan}{0} & \textcolor{mycyan}{0} & \textcolor{mycyan}{0} & \textcolor{myyellow}{0} &
\textcolor{myyellow}{0} & \textcolor{myyellow}{0}
 \end{pmatrix},
\quad
A(2,1)\coloneqq %
 \begin{pmatrix}
\textcolor{myblue}{0} & \textcolor{myblue}{0} & \textcolor{myviolet}{0} & \textcolor{myviolet}{0} & \textcolor{myviolet}{0} & \textcolor{green}{0} &
\textcolor{green}{0} & \textcolor{green}{0} \\
\textcolor{myblue}{0} & \textcolor{myblue}{0} & \textcolor{myviolet}{0} & \textcolor{myviolet}{0} & \textcolor{myviolet}{0} & \textcolor{green}{0} &
\textcolor{green}{0} & \textcolor{green}{0} \\
\mathbin{\textcolor{brown}{-}} \textcolor{brown}{1} \mathbin{\textcolor{brown}{-}} \textcolor{brown}{\lambda_1}& \textcolor{brown}{0} &
\textcolor{magenta}{1} & \textcolor{magenta}{0} & \textcolor{magenta}{0} & \textcolor{myred}{\lambda_1} & \textcolor{myred}{0} & \textcolor{myred}{0} \\
\mathbin{\textcolor{brown}{-}} \textcolor{brown}{1} \mathbin{\textcolor{brown}{-}} \textcolor{brown}{\lambda_2}& \textcolor{brown}{0} &
\textcolor{magenta}{0} & \textcolor{magenta}{1} & \textcolor{magenta}{0} & \textcolor{myred}{\lambda_2} & \textcolor{myred}{0} & \textcolor{myred}{0} \\
\textcolor{brown}{0} & \!\!\!\textcolor{brown}{-1} & \textcolor{magenta}{0} & \textcolor{magenta}{0} & \textcolor{magenta}{1} & \textcolor{myred}{0} &
\textcolor{myred}{0} & \textcolor{myred}{0} \\
\textcolor{mydarkgray}{0} & \textcolor{mydarkgray}{0} & \textcolor{mycyan}{0} & \textcolor{mycyan}{0} & \textcolor{mycyan}{0} &
\textcolor{myyellow}{0} &
\textcolor{myyellow}{0} & \textcolor{myyellow}{0} \\
\textcolor{mydarkgray}{0} &  \!\!\!\!\mathbin{\textcolor{mydarkgray}{-}} \textcolor{mydarkgray}{\lambda_3} & \textcolor{mycyan}{0} &
\textcolor{mycyan}{0} &
\textcolor{mycyan}{\lambda_3} & \textcolor{myyellow}{0} & \textcolor{myyellow}{0} & \textcolor{myyellow}{0} \\
\textcolor{mydarkgray}{0} &  \!\!\!\!\mathbin{\textcolor{mydarkgray}{-}} \textcolor{mydarkgray}{\lambda_4} & \textcolor{mycyan}{0} &
\textcolor{mycyan}{0} &
\textcolor{mycyan}{\lambda_4} & \textcolor{myyellow}{0} & \textcolor{myyellow}{0} & \textcolor{myyellow}{0}
\end{pmatrix}
\end{equation*}
and $A(3,1)\coloneqq \ide - A(1,1) - A(2,1)$ form a quasi-standard column of each pre-twisting of~$K^3$ with~$K^8$ that include them (for instance
we can take  $A(1,2)=A(3,2)=A(1,3)=A(2,3)=0$ and $A(2,2)=A(3,3)=\ide$). In this example $J_1=\{1,2\}$, $J_2=\{3,4,5\}$ and $J_3=\{6,7,8\}$.
\end{example}

\begin{theorem}\label{proto quasi estandar} Assume that the $l_0$-th column of $\mathcal{A}$ is quasi-standard. Then Corollary~\ref{caracterizacion
en terminos de las A(i,j)s}(4) is satisfied in the $l_0$-th column of $\mathcal{A}$ (that is, condition~\eqref{condicion 4 rr2} is fulfilled) if
and only if the following conditions hold:

\begin{enumerate}

\smallskip

\item $J_i=F(A(i,l_0))\subseteq F_0(\mathcal{A},i)$ for all $i\in \mathds{N}_m^*$.

\smallskip

\item If $\bigl(D_{(u)}^{uv}\bigr)_{kd}\ne 0$ and $u\ne v\ne l_0$, then

\begin{enumerate}[label=\emph{(}\alph*{\emph{)}}]

\smallskip

\item $A(u,v)_{kj} = \delta_{kj}-\delta_{jd}$ for all $j$,

\smallskip

\item $A(v,v)_{kj} = \delta_{jd}$ for all $j$,

\smallskip

\item $A(i,v)_{kj}=0$ for $i\notin\{u,v\}$ and for all $j$.
\end{enumerate}
\end{enumerate}
\end{theorem}

\begin{proof} $\Rightarrow$)\enspace The arguments given in the proof of Proposition~\ref{influencia de columna estandar} show that condition~(1)
is fulfilled. So we only must prove condition~(2). By Definition~\ref{quasi standard}(3)
\begin{equation}\label{eq2}
A(i,l_0)_{kd} = 0\qquad\text{for $i\notin \{u,v\}$,}
\end{equation}
which, since $\sum_i A(i,l_0)=\ide$ and $k\ne d$, implies that
\begin{equation}\label{opuestos}
A(v,l_0)_{kd} = -A(u,l_0)_{kd} = -\bigl(D_{(u)}^{uv}\bigr)_{kd}\ne 0.
\end{equation}
Moreover, by condition~(1) we know that $k\in J_u \subseteq F_0(\mathcal{A},u)$, and so
\begin{equation}\label{eq3}
A(i,u)_{kj} = \delta_{iu}\delta_{kj}\qquad\text{for all $i$ and $j$.}
\end{equation}
By~\eqref{eq2} and~\eqref{eq3}, the equality~\eqref{condicion 4 rr2} with $j'=d$ reads
\begin{equation}\label{eq4}
\delta_{iu}\delta_{kj}A(u,l_0)_{kd} + A(i,v)_{kj}A(v,l_0)_{kd}=\delta_{jd}A(i,l_0)_{kd}\qquad\text{for all $i$ and $j$.}
\end{equation}
When $i=u$, from~\eqref{opuestos} and~\eqref{eq4}, we obtain that
\begin{equation*}
\delta_{kj}A(u,l_0)_{kd} - A(u,v)_{kj}A(u,l_0)_{kd}=\delta_{jd} A(u,l_0)_{kd}\qquad\text{for all $j$,}
\end{equation*}
which gives~(a) since $A(u,l_0)_{kd}\ne 0$. On the other hand, when $i\ne u$, equality~\eqref{eq4} reduces to
\begin{equation*}
A(i,v)_{kj} A(v,l_0)_{kd}=\delta_{jd}A(i,l_0)_{kd}\qquad\text{for all $j$,}
\end{equation*}
which, combined with \eqref{eq2} and \eqref{opuestos}, gives items~(b) and~(c).

\smallskip

\noindent $\Leftarrow$)\enspace By Remark~\ref{Simplificacion Corolario 2.4(4)} it suffices to prove that
\begin{equation}\label{eq5}
\sum_{h\in \mathds{N}_m^*} A(i,h)_{kj} A(h,l_0)_{kj} = A(i,l_0)_{kj}\qquad\text{for all $i$, $k$ and $j$.}
\end{equation}
Fix $k\in \mathds{N}_n^*$ and set $u\coloneqq i(k,l_0)$. If $k\in F(A(l_0,l_0))=F_0(\mathcal{A},l_0)$, then $A(i,l_0)_{kj}=\delta_{il_0}
\delta_{kj}$ for all $i$ and $j$, and equality~\eqref{eq5} is trivially true. Consequently we can assume that $u\ne l_0$. So, by
Lemma~\ref{condicion},
$$
\Supp\bigl(A(i,l_0)_{k*}\bigr)\subseteq \{k,c_k\}\cup d^{(\mathscr{X}_k)}\qquad\text{for all $i$.}
$$
Thus, if $j\notin \{k,c_k\}\cup d^{(\mathscr{X}_k)}$ both sides of the equality~\eqref{eq5} are zero.

\smallskip

Assume that $j=k$. By Remark~\ref{suma en la diagonal} equality~\eqref{eq5} reads
$$
A(i,u)_{kk} = \delta_{iu}\qquad\text{for all $i$.}
$$
But this is true since, by condition~(1), we have $k\in J_u\subseteq F(\mathcal{A},u)$.

\smallskip

Suppose now that $j=c_k$. By Remark~\ref{cero fuera}, Lemma~\ref{D distinto de cero} and conditions~(1) and~(2), equality~\eqref{eq5} reduces to
$$
A(i,l_0)_{kc_k}A(l_0,l_0)_{kc_k}= A(i,l_0)_{kc_k}\quad \text{for all $i$},
$$
which is true.

\smallskip

If $j=d^{(v)}$ for $v\notin\{u,l_0\}$, then $0\ne A(h,l_0)_{kd^{(v)}}= (D_{(h)}^{uv})_{kd^{(v)}}$ implies that $h\in\{u,v\}$, by item~(3) of
Definition~\ref{quasi standard}. But by condition~(1) we know that $A(i,u)_{k d^{(v)}}=\delta_{iu}\delta_{kd^{(v)}}=0$. So equality~\eqref{eq5}
reduces to
$$
A(i,v)_{kd^{(v)}}A(v,l_0)_{kd^{(v)}}= A(i,l_0)_{kd^{(v)}},
$$
which can be verified easily using  that
$$
A(u,l_0)_{kd^{(v)}}=\bigl(D_{(u)}^{uv}\bigr)_{kd^{(v)}} = -\bigl(D_{(v)}^{uv}\bigr)_{kd^{(d)}}=-A(v,l_0)_{kd^{(v)}}\ne 0
$$
and condition~(2).
\end{proof}

\begin{definition}\label{columna de rango reducido r} We say that the $l_0$-th column of $\mathcal{A}$ has {\em reduced rank} $r$ if there are
exactly $r$ indices $i\ne l_0$ such that $A(i,l_0)\ne 0$. In this case we write~$\rrank_{\mathcal{A}}(l_0)=r$. If $\mathcal{A}$ is associated with
a map $\chi$ as at the beginning of Section~\ref{Twisted tensor products of K^n with K^m}, then we use $\rrank_{\chi}(l_0)$ as a synonym of
$\rrank_{\mathcal{A}}(l_0)$.
\end{definition}

\begin{remark} Let $l_0,u\in \mathds{N}_m^*$ and let and $k\in J_u$. Assume that $\mathcal{A}$ is a family of matrices asso\-ciated with a twisting
map of $K^m$ with $K^n$ and that conditions~(1) and~(2) of Definition~\ref{quasi standard} are fulfilled for the $l_0$-th column of $\mathcal{A}$.
By Remark~\ref{suma en la diagonal} we have $A(v,l_0)_{kk}=\delta_{uv}$ for all $v$.~Con\-se\-quent\-ly, from Corollary~\ref{caracterizacion en
terminos de las A(i,j)s}(4) with $j'=k$, it follows that
\begin{equation}\label{Aiu es delta}
A(i,u)_{kj} = \delta_{iu}\delta_{kj}\quad\text{for all $i$ and $j$.}
\end{equation}
\end{remark}

\begin{proposition}\label{cuando es cuasiestandard} Let $l_0\in \mathds{N}_m^*$. Assume that $\mathcal{A}$ is a family of matrices associated with
a~twist\-ing map of $K^m$ with $K^n$ and that conditions~(1) and~(2) of Definition~\ref{quasi standard} are fulfilled for the $l_0$-th column of
$\mathcal{A}$.
\begin{enumerate}[label=\emph{(}\alph*\emph{)}]

\smallskip

\item If condition~(3) is also fulfilled, then the $l_0$-th column of $\mathcal{A}$ is quasi-standard.

\smallskip

\item If the reduced rank of the $l_0$-th column  of $\mathcal{A}$ is lower than or equal to~$2$, then the $l_0$-th column of $\mathcal{A}$ is
quasi-standard.

\smallskip

\end{enumerate}
\end{proposition}

\begin{proof} As in Definition~\ref{quasi standard}, for all $i,u,v\in \mathds{N}_m^*$ we set $D_{(i)}^{uv}\coloneqq A(i,l_0)|_{J_u\times J_v}$,
where $J_i\coloneqq J_i(l)$.

\smallskip

\noindent (a)\enspace We must prove that if condition~(3) of Definition~\ref{quasi standard} is fulfilled, then condition~(4') of
Remark~\ref{simplificacion} is also. We begin by proving that
\begin{equation}\label{pepe}
\# \Supp\bigl(\bigl(D^{uv}_{(v)}\bigr)_{k*}\bigr)\le 1\quad\text{for all $u,v\in \mathds{N}_m^*\setminus \{l_0\}$ and $k\in J_u$.}
\end{equation}
For $u=v=i$ it is true by Remark~\ref{identidad en la diagonal}. So, we can assume that $u\ne v$. Assume on the contrary that there exist $d_1\ne
d_2$ in $\Supp\bigl(\bigl(D^{uv}_{(v)}\bigr)_{k*}\bigr)$. Since $A(i,l_0)_{kd_1} = 0$ for $i\notin\{u,v\}$ and $A(v,l_0)_{kd_1} = -A(u,l_0)_{kd_1}
\ne 0$,~Corol\-lary~\ref{caracterizacion en terminos de las A(i,j)s}(4) with $i = u$, $l=l_0$ and $j=j'=d_1$ gives
$$
A(v,v)_{kd_1} - A(v,u)_{kd_1} = 1.
$$
A similar argument shows that Corollary~\ref{caracterizacion en terminos de las A(i,j)s}(4) with $i=u$, $l=l_0$, $j=d_1$ and $j'=d_2$, gives
$$
A(v,v)_{kd_1} - A(v,u)_{kd_1} = 0,
$$
a contradiction.

\smallskip

It remains to check that if $d\in \Supp\bigl(\bigl(D^{uv}_{(v)}\bigr)_{k*}\bigr)$, then $c_d = c_k$. When $v=u$ this follows again from
Remark~\ref{identidad en la diagonal}. Assume that $v\ne u$. We assert that
\begin{equation}\label{zazaza}
\Supp\bigl(A(v,l_0)_{k*}\bigr) = \{d,c_d\}.
\end{equation}
Since $\Supp\bigl(\bigl(D^{uv}_{(v)}\bigr)_{k*}\bigr)=\{d\}$ and $D^{ui}_{(v)} = 0$ for $i\notin\{v,l_0\}$, this is true if and only if
$$
\Supp \bigl(\bigl(D^{ul_0}_{(v)}\bigr)_{k*}\bigr) = \{c_d\}.
$$
In order to check this, note that $A(v,l_0)A(l_0,l_0) = 0$  imply
$$
D^{ul_0}_{(v)}D^{l_0l_0}_{(l_0)}+D^{uv}_{(v)}D^{vl_0}_{(l_0)}=\sum_i D^{ui}_{(v)}D^{il_0}_{(l_0)} = 0.
$$
Since $D^{l_0l_0}_{(l_0)} = \ide$ and $\Supp\bigl(\bigl(D^{uv}_{(v)}\bigr)_{k*}\bigr) = \{d\}$, this yields
$$
\bigl(D^{ul_0}_{(v)}\bigr)_{k*} = - \bigl(D^{uv}_{(v)}\bigr)_{kd}\bigl(D^{vl_0}_{(l_0)}\bigr)_{d*}.
$$
Combining this with the fact that
$$
\Supp \bigl(\bigl(D^{vl_0}_{(l_0)}\bigr)_{d*}\bigr) = \Supp \bigl(A(l_0,l_0)_{d*}\bigr) = \{c_d\},
$$
we obtain that $\Supp \bigl(\bigl(D^{ul_0}_{(v)}\bigr)_{k*}\bigr) = \{c_d\}$, as we need.

By Lemma~\ref{D distinto de cero}, if $A(h,l_0)_{kc_k} \ne 0$, then $h\in\{u,v,l_0\}$. So Corollary~\ref{caracterizacion en terminos de las
A(i,j)s}(4) with $j=d$, $j'=c_k$ and $i=v$ gives
$$
A(v,l_0)_{kd} + A(v,v)_{kd} A(v,l_0)_{kc_k} + A(v,u)_{k} A(u,l_0)_{kc_k} = 0,
$$
where we use that $A(l_0,l_0)_{kc_k}=1$. But by~\eqref{Aiu es delta} we have $A(v,u)_{kd} = \delta_{vu}\delta_{kd} = 0$, and so, necessarily
$A(v,l_0)_{kc_k}\ne 0$, which, by equality~\eqref{zazaza}, implies $c_k=c_d$.

\smallskip

\noindent (b)\enspace If the reduced rank of the $l_0$-th column of $\mathcal{A}$ is lower than~$2$, then that column is standard and the result is
trivial (see Remark~\ref{estandar implica cuasi estandar}). So we can assume that its reduced rank is~$2$. By item~(a) and
Remark~\ref{simplificacion} it suffices to prove that $D_{(i)}^{uv}=0$, if $i\ne l_0$ and $u,v\notin\{i,l_0\}$. Since the reduced rank of the
$l_0$-th column is~$2$, there exist two indices $i_0,i_1\ne l_0$ such that $A(i,l_0)\ne 0$ if and only if $i\in\{l_0,i_0,i_1\}$. So we must prove
that $D_{(i_a)}^{i_bi_b}=0$ for $a\in \{0,1\}$ and $b\coloneqq 1-a$. Take $k\in J_{i_b}$. We first prove that either
\begin{equation}\label{unico indice}
\Supp(A(i_b,l_0)_{k*})\subseteq\{k,c_k\}\quad\text{or}\quad \exists!\ d\ \text{such that $d\in\Supp(A(i_b,l_0)_{k*})\setminus\{k,c_k\}$}.
\end{equation}
Assume by contradiction that there exist $d\ne e\in \Supp(A(i_b,l_0)_{k*})\setminus\{k,c_k\}$. First note that since
$A(l_0,l_0)+A(i_a,l_0)+A(i_b,l_0)=\ide$ and $\Supp(A(l_0,l_0)_{k*})=\{c_k\}$, if $f\notin\{k,c_k\}$, then
\begin{equation}\label{opuestos'}
A(i_a,l_0)_{kf}=-A(i_b,l_0)_{kf}.
\end{equation}
By equation~\eqref{Aiu es delta} we know that $A(i_b,i_b)_{kd}=0$. Moreover, since $\Supp(A(l_0,l_0)_{k*})=\{c_k\}$, we have
$$
A(l_0,l_0)_{kd}= A(l_0,l_0)_{ke} = 0.
$$
Consequently from Corollary~\ref{caracterizacion en terminos de las A(i,j)s}(4) with $j=j'=d$ and $i=i_b$, we obtain that
$$
A(i_b,i_a)_{kd} A(i_a,l_0)_{kd}=\sum_u A(i_b,u)_{kd}A(u,l_0)_{kd} = A(i_b,l_0)_{kd}\ne 0,
$$
which implies $A(i_b,i_a)_{kd} \ne 0$. On the other hand from Corollary~\ref{caracterizacion en terminos de las A(i,j)s}(4) with $j=d$, $j'=e$ and
$i=i_b$, we obtain that
$$
A(i_b,i_a)_{kd} A(i_a,l_0)_{ke} = \sum_u A(i_b,u)_{kd}A(u,l_0)_{ke}=0,
$$
and so, necessarily $A(i_a,l_0)_{ke}=0$. But this is impossible since $A(i_a,l_0)_{ke}=-A(i_b,l_0)_{ke}\ne 0$. Hence condition~\eqref{unico indice}
is satisfied. We claim that if it exists, then $d\in J_{i_a}$. In fact, since $k\in J_{i_b}$ we have $A(i_a,l_0)_{kk}=0$, and thus, by
equality~\eqref{opuestos'}, if $d\in \Supp(A(i_b,l_0)_{k*})\setminus\{k,c_k\}$, then $\Supp(A(i_a,l_0)_{k*})=\{c_k,d\}$. Using now that
$A(i_a,l_0)$ is idempotent, we obtain that
\begin{align*}
A(i_a,l_0)_{kd}&=A(i_a,l_0)_{k*}A(i_a,l_0)_{*d} \\
& =A(i_a,l_0)_{kc_k}A(i_a,l_0)_{c_kd}+A(i_a,l_0)_{kd}A(i_a,l_0)_{dd}\\
& =A(i_a,l_0)_{kd}A(i_a,l_0)_{dd},
\end{align*}
since $A(i_a,l_0)_{c_kd}=0$ by Remark~\ref{columnas que se anulan en l_0}. But then $A(i_a,l_0)_{dd}=1$, which means that $d\in J_{i_a}$, as we
claim. Thus
$$
\Supp(A(i_a,l_0)_{k*})\subseteq J_{l_0}\cup J_{i_a},
$$
which implies that $\Supp((D_{(i_a)}^{i_bi_b})_{k*})=\Supp(A(i_a,l_0)_{k*})\cap J_{i_b}=\emptyset$ for all $k\in J_{i_b}$, as desired.
\end{proof}

\begin{definition} We say that a twisting map $\chi \colon K^m\ot K^n\longrightarrow K^n\ot K^m$ is {\em quasi-standard} if the columns of
$\mathcal{A}_{\chi}$ are quasi-standard.
\end{definition}

\begin{proposition}\label{A quasi standard equivale a B quasi standard} A twisting map $\chi \colon K^m\ot K^n\longrightarrow K^n\ot K^m$ is
quasi-standard if and only if the map $\tilde{\chi}$, introduced in Remark~\ref{A sub Chi es B sub tau chi tau}, is a quasi-standard twisting map.
\end{proposition}

\begin{proof} By Proposition~\ref{cuando es cuasiestandard}, Remark~\ref{A sub Chi es B sub tau chi tau} and the fact that $\chi$ is a twisting map
if and only if~$\widetilde{\chi}$ is, in order to prove the proposition it suffices to check that if every column of $\mathcal{A}_{\chi}$ is
quasi-standard, then each column of~$\mathcal{A}_{\widetilde{\chi}}= \mathcal{B}_{\chi}$ satisfies items~(1), (2) and~(3) of Definition~\ref{quasi
standard}. Assume that each column of $\mathcal{A}_{\chi}$ is quasi-standard. Using equality~(2.3) it is easy to check that items~(1) and~(2) are
satisfied by the columns of $\mathcal{B}_{\chi}$. Consequently, by Remark~\ref{simplificacion} we only must prove that $\widetilde
D^{uv}_{(j)}\coloneqq B_{\chi}(j,k)|_{\tilde J_u\times \tilde J_v}$ (where $\tilde{J}_u\coloneqq \tilde{J}_u(k)$) are null matrices for $j\ne k$
and $u,v\notin\{j,k\}$. So we are reduced to prove that
$$
B_{\chi}(j,k)_{ls}=A_{\chi}(s,l)_{kj}=0\quad\text{for all $k\notin\{j,u,v\}$, $j\notin\{u,v\}$, $l\in \tilde J_u$ and $s\in \tilde J_v$.}
$$
But, since
$$
l\in \tilde J_u \text{ and }s\in \tilde J_v\quad\text{ if and only if }\quad A_{\chi}(l,l)_{ku}=1 \text{ and } A_{\chi}(s,s)_{kv}=1,
$$
and, in that case,
$$
j\notin\{u,v\} \quad\text{ if and only if }\quad  A_{\chi}(l,l)_{kj}=0\text{ and } A_{\chi}(s,s)_{kj}=0,
$$
for this it suffices to check that if $k\notin\{j,u,v\}$ and the $l$-th column of $\mathcal{A}_{\chi}$ is quasi-standard, then
$$
\setlength{\tabcolsep}{1pt} \left. \begin{tabular}{>{$}c<{$}  >{$}c<{$}}  A_{\chi}(l,l)_{ku}&=1\\ A_{\chi}(s,s)_{kv}&=1 \\ A_{\chi}(l,l)_{kj}&=0 \\
A_{\chi}(s,l)_{kj}& \ne 0 \end{tabular} \right\} \Rightarrow A_{\chi}(s,s)_{kj}\ne 0.
$$
Clearly $s\ne l$.  Moreover $k\in J_w(l)$ with $w\ne s$ since, otherwise $A_{\chi}(s,l)_{kj}=0$ by Theorem~\ref{proto quasi estandar}(1). Suppose
that $k\notin J_l(l)$ and $j\notin\bigl\{k, c_k(A_{\chi}(l,l))\bigr\}$. Then, by Lemma~\ref{condicion} we have $j\notin J_l(l)\cup J_w(l)$.
Consequently, by Definition~\ref{quasi standard}(3) and Remark~\ref{cero fuera},
$$
j\in \Supp((D^{ws}_{(s)})_{k*})= \Supp((D^{ws}_{(w)})_{k*}) \qquad\text{and} \qquad w\ne s\ne l.
$$
Thus, from Theorem~\ref{proto quasi estandar}(2b) we obtain that $A_{\chi}(s,s)_{kj}=\delta_{jj}\ne 0$, as desired. So, in order to finish the
proof we must check that $k\notin J_l(l)$ and $j\notin\bigl\{k, c_k(A_{\chi}(l,l))\bigr\}$. But $k \notin J_l$, because $A_{\chi}(l,l)_{ku}=1$
implies that $A_{\chi}(l,l)_{kk}=0$; $j\ne k$, because, by Theorem~\ref{proto quasi estandar}(1), if $j=k$, then
$A_{\chi}(s,s)_{kv}=\delta_{ss}\delta_{kv}=0$; and $j\ne c_k\bigl(A_{\chi}(l,l)\bigr)$, since $A_{\chi}(l,l)_{kj}=0$.
\end{proof}

\begin{proposition}\label{construccion de columnas cuasiestandars} Each quasi-standard column $A(i,l_0)_{i\in \mathds{N}^*_m}$ of a pre-twisting of
$K^m$ with $K^n$ can be obtained in the following way:

\begin{enumerate}

\smallskip

\item Take a matrix $A\in M_n(K)$, which is equivalent via identical permutations in rows and columns to a standard idempotent $0,1$-matrix, and
set $A(l_0,l_0)\coloneqq A$.

\smallskip

\item Set $J_{l_0}\coloneqq \bigl\{k\in \mathds{N}_n^*: A(l_0,l_0)_{kk}=1\bigr\}$ and $J_{l_0}^c\coloneqq \mathds{N}^*_n \setminus J_{l_0}$.

\smallskip

\item For all $i\in\mathds{N}_m^*\setminus\{l_0\}$ choose $J_i\subseteq\mathds{N}_n^*\setminus J_{l_0}$ such that
$$
\bigcup_{i=1}^m J_i=\mathds{N}_n^*\quad\text{and}\quad J_i\cap J_{i'}=\emptyset\quad\text{if $i\ne i'$}.
$$

\smallskip

\item Set $\digamma\coloneqq \{i\in\mathds{N}_m^*:J_i\ne \emptyset\}$ and choose $D_{(i)}^{ij}\!\in M_{J_i\times J_j}(K)$ for $i\ne j$ in
$\digamma\setminus\{l_0\}$, such that

\begin{enumerate}[label=\emph{(}\alph*\emph{)}]

\smallskip

\item $D_{(r)}^{ri}D_{(i)}^{ij} = 0$ for all $r\ne i\ne j$,

\smallskip

\item $\#\Supp \bigl((D_{(i)}^{ij})_{k*}\bigr)\le 1$ for all $i\ne j$ and $k\in J_i$,

\smallskip

\item If $d\in \Supp \bigl((D_{(i)}^{ij})_{k*}\bigr)$, then $c_d = c_k$, where $c_d\coloneqq c_d(A(l_0,l_0))$ and $c_k\coloneqq
c_k(A(l_0,l_0))$.

\end{enumerate}

\smallskip

\item Set

\begin{enumerate}[label=\emph{(}\alph*\emph{)}]

\smallskip

\item $D_{(j)}^{ij}\coloneqq -D_{(i)}^{ij}$ for all $i\ne j$ in $\digamma\setminus\{l_0\}$,

\smallskip

\item $D_{(i)}^{ii}\coloneqq \ide_{J_i}$ for all $i\in \digamma\setminus\{l_0\}$,

\smallskip

\item $D_{(i)}^{rj}\coloneqq 0$ for all $i,j,r\in \digamma\setminus \{l_0\}$ such that $i\notin\{j,r\}$,

\smallskip

\end{enumerate}

\smallskip

\item For each $i\in \digamma\setminus\{l_0\}$ define $W^{(i)}\in M_{J_{l_0}^c\times J_{l_0}^c}(K)$ by
$$
W^{(i)}_{kj} \coloneqq \bigl(D_{(i)}^{uv}\bigr)_{kj}\qquad\text{for $k\in J_u$ and $j\in J_v$ (Note that $u,v\ne l_0$).}
$$

\smallskip

\item Set $C\coloneqq A(l_0,l_0)|_{J_{l_0}^c\times J_{l_0}}$. For each $i\in \digamma\setminus\{l_0\}$, define $A(i,l_0)$ to be the unique matrix
satisfying
$$
\qquad \quad A(i,l_0)|_{J_{l_0}^c\times J_{l_0}^c}=W^{(i)},\quad A(i,l_0)|_{J_{l_0}^c\times J_{l_0}}=-W^{(i)} C\quad\text{and}\quad A(i,l_0)|_{J_{l_0}\times \mathds{N}_n^*}=0.
$$

\smallskip

\item For $i\notin \digamma$ set $A(i,l_0)\coloneqq 0$.

\end{enumerate}
\end{proposition}

\begin{proof} We first prove that the construction yields a quasi-standard column of a pre-twisting. We begin by checking that conditions~(1), (2)
and~(3) of Corollary~\ref{caracterizacion en terminos de las A(i,j)s} are satisfied in the $l_0$-th column. By Remark~\ref{cuando son idempotentes
ortogonales} and Proposition~\ref{reduccion al cuadrado inferior} for this it suffices to prove that $\sum_{i\in \digamma\setminus\{l_0\}}
W^{(i)}=\ide_{J_{l_0}^c}$ and $\cramped{{W^{(i)}}^2=W^{(i)}}$ for all $i\in \digamma\setminus\{l_0\}$. But the first quality follows from item~(5),
while the second one, from items~(4)(a) and~(5). It remains to check that conditions~(1) and~(4) of Definition~\ref{quasi standard} and
conditions~(3') and~(4') of Remark~\ref{simplificacion} are satisfied. Condition~(1) is clear; condition~(2) follows from~(5)(b),~(5)(c) and~(7);
condition~(3'), from~(5)(c); and condition~(4'), from~(4)(b), (4)(c) and~(5)(b).

\smallskip

Now we are going to check that any quasi-standard column of idempotent matrices can be constructed as above. For this note that applying an
identical permutations in rows and columns we can assume that $A(l_0,l_0)$ is a standard idempotent $0,1$-matrix. Using Proposition~\ref{reduccion
al cuadrado inferior}, Def\-i\-ni\-tion~\ref{quasi standard} and Remarks~\ref{columnas que se anulan en l_0},~\ref{identidad en la diagonal}
and~\ref{cero fuera}, a straightforward verification shows that the $A(i,l_0)$'s can be constructed following the given receipt.
\end{proof}

\begin{remark}\label{construccion quasi} Suppose we have performed the steps indicated in items (1)--(3) of Proposition~\ref{construccion de
columnas cuasiestandars}. An algorithm for the construction of matrices $D_{(i)}^{ij}$ satisfying item~(4) of the previous proposition is the
following:
\begin{itemize}

\smallskip

\item[-] Set $\ov{\digamma}\coloneqq \digamma\setminus\{l_0\}$ and fix a total order in $\overline{\Delta}_{\ov{\digamma}} \coloneqq
(\ov{\digamma}\times \ov{\digamma})\setminus \{(x,x): x\in \ov{\digamma}\}$.

\smallskip

\item[-] For increasing $(i,j)\in \ov{\Delta}_{\ov{\digamma}}$ perform the following construction for all $k\in J_i$, which produce the matrix
$D_{(i)}^{ij}$:

\begin{enumerate}[label=(\alph*)]

\smallskip

\item If $k\in\Supp\bigl(\bigl(D^{ri}_{(r)}\bigr)_{t*}\bigr)$ for some $t\in J_r$ and $(r,i)<(i,j)$, then set
$\bigl(D_{(i)}^{ij}\bigr)_{k*}\coloneqq 0$.

\smallskip

\item Let $\mathscr{D}_i^j\coloneqq \bigl\{d\in J_j: c_d=c_k \text{\hspace{1pt} and \hspace{1pt}}
\bigl(D_{(j)}^{jr}\bigr)_{d*}=0\text{\hspace{1pt} for all \hspace{1pt}} (j,r)<(i,j)\bigr\}$. If $\mathscr{D}_i^j=\emptyset$, then set
$\bigl(D_{(i)}^{ij}\bigr)_{k*}\coloneqq 0$. Else choose $d\in \mathscr{D}_i^j$ and $\lambda\in K$ and set $\bigl(D_{(i)}^{ij}\bigr)_{kv}
\coloneqq \lambda \delta_{vd}$ for all $v\in J_j$.

\end{enumerate}

\smallskip

\end{itemize}
It is clear that the above construction guarantees that for a given $(i,j)$ we have $\cramped{D_{(r)}^{ri}D_{(i)}^{ij}=0}$ for all $(r,i)<(i,j)$
and $D_{(i)}^{ij}D_{(v)}^{jr}=0$ for $(j,r)<(i,j)$. Also it is clear that this construction, performed with $d$ and $\lambda$ arbitraries, gives
all the possible families $\cramped{\bigl(D_{(i)}^{ij}\bigr)_{(i,j)\in \ov{\Delta}_{\ov{\digamma}}}}$ that satisfy item~(4) of
Proposition~\ref{construccion de columnas cuasiestandars}.
\end{remark}

Let $\chi\colon K^m\ot K^n \longrightarrow K^n \ot K^m$ be a twisting map and let $r<m$. By Proposition~\ref{extension de twistings} we know that
there exists a twisting map
$$
\check{\chi}\colon K^r\ot K^n \longrightarrow K^n \ot K^r
$$
such that $\mathcal{A}_{\check{\chi}}=\bigl(A_{\chi}(i,l)\bigr)_{1\le i,l\le r}$ if and only if $A_{\chi}(i,l)=0$ for all $i>r$ and $l\le r$. Now
suppose that we have a twisting map
$$
\check{\chi}\colon K^r\ot K^n \longrightarrow K^n \ot K^r.
$$
Let $\mathcal{A}=(A(i,l))_{1\le i,l\le m}$ be a pre-twisting which is a extension of the family $\mathcal{A}_{\check{\chi}}=
(A_{\check{\chi}}(i,l))_{1\le i,l\le r}$ such that
\begin{itemize}

\smallskip

\item[-] $A(i,l)=0$ if $i>r$ and $l\le r$,

\smallskip

\item[-] for $l>r$, the $l$-th column of $\mathcal{A}$ is a quasi-standard column.

\smallskip

\end{itemize}
In the following theorem we give necessary and sufficient conditions in order that $\mathcal{A}$ defines a twisting map.

\begin{theorem}\label{teorema principal} Let $\mathcal{A}$ be as above. For all $u,v,l\in \mathds{N}_m^*$ with $l>r$, set $D_{(i,l)}^{uv}\coloneqq
A(i,l)|_{J_u\times J_v}$. The family $\mathcal{A}$ defines a twisting map if and only if
\begin{enumerate}

\smallskip

\item for all $i\in\mathds{N}^*_m$,
$$
\bigcup_{l>r} F(A(i,l)) \subseteq F_0(\mathcal{A},i).
$$

\smallskip

\item  If $\bigl(D_{(u,l)}^{uv}\bigr)_{kd}\ne 0$, with  $u\ne v\ne l$, then

\begin{enumerate}[label=\emph{(}\alph*{\emph{)}}]

\smallskip

\item $A(u,v)_{kj} = \delta_{kj}-\delta_{jd}$ for all $j$,

\smallskip

\item $A(v,v)_{kj} = \delta_{jd}$ for all $j$,

\smallskip

\item $A(i,v)_{kj}=0$ for $i\notin\{u,v\}$ and for all $j$.

\smallskip

\end{enumerate}
\end{enumerate}
Moreover there exist $u\ne v\ne l$ such that $D_{(u,l)}^{uv}\ne 0$ if and only if the $l$-th column is not a~stan\-dard column.
\end{theorem}

\begin{proof} The last assertion follows immediately from the definition of standard column. Next we prove the main part of the theorem.

\smallskip

\noindent $\Leftarrow$) We only must show that condition~(4) of Corollary~\ref{caracterizacion en terminos de las A(i,j)s} is satisfied. For $l\le
r$ this is true since
$$
\sum_{h=1}^m A(i,h)_{kj}A(h,l)_{kj'}=\sum_{h=1}^r A(i,h)_{kj}A(h,l)_{kj'}=\delta_{jj'}A(i,l)_{kj},
$$
because $A(h,l)=0$ if $h>r$ and $\check{\chi}$ is a twisting map; while for $l>r$ this follows from Theorem~\ref{proto quasi estandar}

\smallskip

\noindent $\Rightarrow)$ This follows immediately from Theorem~\ref{proto quasi
estandar}.
\end{proof}

\begin{proposition}\label{casi todas id} Let $\mathcal{A}$ be a pre-twisting of $K^m$ with $K^n$. Assume that $A(i,i)=\ide$ for all $i\in
\mathds{N}^*_{m-1}$.  Then $\mathcal{A}$ is the family $\mathcal{A}_{\chi}$ of matrices associated with a twisting map~$\chi$ if and only if
$\bigl(A(l,m)\bigr)_{l\in \mathds{N}^*_m}$ is a standard column.
\end{proposition}

\begin{proof} $\Rightarrow$)\enspace By the assumptions it is clear that the rank matrix $\Gamma_{\chi}$ introduced in Definition~\ref{matriz de
rangos} has the form
\begin{equation}\label{matriz standard}
\Gamma_{\chi} = \begin{pmatrix} n\ide_{m-1} & *\\ 0 & *\end{pmatrix},
\end{equation}
Consequently, $\Gamma_\chi$ satisfies the hypothesis of Proposition~\ref{forma general casi triangular} for $l=m$, and so $A(m,m)$ is a
$0,1$-matrix. It remain to check that item~(2) of Definition~\ref{definicion columna estandar} is fulfilled for $l_0=m$, i.e., that
$$
A(k,m)_{ij}\ne 0 \Rightarrow A(m,m)_{ij}\ne 0\qquad\text{for $k<m$ and $i\ne j$;}
$$
but this follows immediately from the fact that
\begin{equation*}\label{rango de Bij}
A(m,m)_{ij}=\sum_t A(t,t)_{ij}=\sum_t B_{\chi}(j,i)_{tt}= \rk((B_{\chi}(j,i))\qquad\text{for all $i\ne j$,}
\end{equation*}
and $B_{\chi}(j,i)_{mk}=A(k,m)_{ij}$.

\smallskip

\noindent $\Leftarrow$)\enspace This follows from Theorem~\ref{teorema principal}, since by Remark~\ref{columnas que se anulan en l_0} we know that
$F(A(m,m)) = F_0(\mathcal{A},m)$ and $A(i,i)=\ide_n$ implies that $F_0(\mathcal{A},i)=\mathds{N}_n^*$ for all $i<m$.
\end{proof}

\section[Reduced rank~$1$]{Reduced rank~$\bm{1}$}\label{Seccion rango reducido 1}
In \cite{JLNS} the case of twisting maps $\chi$ in which all the columns of $\mathcal{A}_{\chi}$ have reduced rank less than or equal to $1$ (see
Definition~\ref{columna de rango reducido r}) is analysed. In this section we use our tools, that are completely different to the ones used in
\cite{JLNS}, in order to describe these twisting maps.

\begin{proposition}\label{rango reducido 1} Let $\chi\colon K^m\ot K^n\longrightarrow K^n\ot K^m$ be a twisting map. Assume that $\rrank_\chi(l)=1$ and $A_{\chi}(i,l)\ne 0$ where $i\ne l$.  The following facts hold:

\begin{enumerate}

\smallskip

\item If $A_{\chi}(l,i)=0$, then the $l$-th column of $\mathcal{A}_{\chi}$ is standard. Moreover, if $A_{\chi}(l,l)_{kk}=0$, then
$A_{\chi}(i,i)_{kj}=\delta_{kj}$ for all $j$.

\smallskip

\item If $A_{\chi}(l,i)\ne 0$ and $\rrank_\chi(i)=1$, then there is a twisting map $\psi\colon K^2\otimes K^n\longrightarrow K^n\otimes K^2$ with
$A_{\psi}(a,b)\coloneqq A_{\chi}(f(a),f(b))$, where $a,b\in \{1,2\}$, $f(1)\coloneqq i$ and $f(2)\coloneqq l$.

\end{enumerate}
\end{proposition}

\begin{proof}(1)\enspace By Proposition~\ref{forma general casi triangular} we know that $A_{\chi}(l,l)$ is a $0,1$-matrix, and clearly
$$
A_{\chi}(l,l)+A_{\chi}(i,l)=\ide \Rightarrow \Supp\bigl(A_{\chi}(i,l)\bigr)\subseteq \Supp\bigl(A_{\chi}(l,l)\bigr)\cup \Supp(\ide).
$$
So the $l$-th column of $\mathcal{A}_{\chi}$ is standard. The last assertion follows from Proposition~\ref{influencia de columna estandar}, since
$A_{\chi}(l,l)_{kk}=0$ implies that $k\in F(A_{\chi}(i,l))$.

\smallskip

\noindent (2)\enspace The family of matrices $\left(A_{\psi}(a,b)\right)_{1\le a,b\le 2}$ satisfies the conditions of
Corollary~\ref{caracterizacion en terminos de las A(i,j)s}. In fact, this is clear for the three first conditions, whereas the last one follows
easily from the fact that
$$
\sum_{h=1}^m A_{\chi}(u,h)_{kj}A_{\chi}(h,v)_{kj'}=A_{\chi}(u,i)_{kj}A_{\chi}(i,v)_{kj'}+A_{\chi}(u,l)_{kq}A_{\chi}(l,v)_{kj'},
$$
if $v\in\{i,l\}$.
\end{proof}

\begin{proposition}\label{clasificacion rango reducido 1} Let $\mathcal{A}=\left(A(i,l)\right)_{1\le i,l\le m}$ be a pre-twisting of $K^m$ with
$K^n$. For each $l$ whose reduced rank is~$1$, let $i(l)$ denote the unique $i(l)\ne l$ such that $A(i(l),l)\ne 0$. If $\rrank_{\mathcal{A}}(l) \le
1$ for all $l$, the there exists a twisting map $\chi\colon K^m\ot K^n \longrightarrow K^n \ot K^m$ with $\mathcal{A}_{\chi}=\mathcal{A}$ if and
only if for each $l\in \mathds{N}^*_m$ such that $\rrank_{\mathcal{A}}(l)=1$ the following facts hold:

\begin{enumerate}

\smallskip

\item If $A(l,i(l))=0$, then:

\smallskip

\begin{enumerate}[label=\emph{(}\alph*\emph{)}]

\item $A(l,l)$ is equivalent to a standard idempotent $0,1$-matrix via identical permutations in rows and columns,

\smallskip

\item $A(i(l),i(l))_{kj}=\delta_{kj}$ for all $j$, whenever $A(j,j)_{kk}=0$.

\end{enumerate}

\smallskip

\item If $A(l,i(l))\ne 0$, then there is a twisting map $\psi \colon K^2\otimes K^n\longrightarrow K^n\otimes K^2$ with $A_{\psi}(a,b)\coloneqq
A_{\chi}(f(a),f(b))$, where $a,b\in\{1,2\}$, $f(1)\coloneqq i$ and $f(2)\coloneqq j$.

\end{enumerate}
\end{proposition}

\begin{proof} The conditions are necessary by Proposition~\ref{rango reducido 1} and Corollary~\ref{lema 01 matrix previo}. On the other hand, it
is straightforward to check that if $\mathcal{A}$ satisfies items~(1) and (2), then it also fulfills condition~(4) of
Corollary~\ref{caracterizacion en terminos de las A(i,j)s}.
\end{proof}

We associate a quiver $Q_{\chi}$ with a twisting map $\chi\colon K^m\ot K^n \longrightarrow K^n \ot K^m$ in the following way. The vertices are
$1,\dots, m$ and the adjacency matrix of $Q_{\chi}$ is the $0,1$-matrix with $1$ in the entry $(i,l)$ if and only if $i\ne l$ and $A_{\chi}(i,l)\ne
0$.

\begin{remark} Proposition~\ref{clasificacion rango reducido 1} allows to construct all the twisting maps $\chi\colon K^m\ot K^n \longrightarrow
K^n \ot K^m$ of reduced rank~$1$ (this means that each column of $\mathcal{A}_{\chi}$ has reduced rank lesser than or equal to~$1$, and at least
one of its columns has reduced rank~$1$). For this it suffices to consider twisting maps with connected quivers, since every twisting map is the
direct sum of the twisting maps restricted to the connected components. Each connected component of the quiver $Q_{\chi}$ has at most one proper
oriented cycle. This follows from the fact that each vertex of the quiver is the head of at most one arrow from another vertex, since the reduced
rank of $\chi$ is~$1$. So, in order to construct such a twisting map $\chi$ take a quiver~$Q$  fulfilling this condition and fix a connected
component. There are three possible cases: the connected component is a $2$-cycle, the connected component contains no $2$-cycle or the connected
component contains properly a $2$-cycle. The two first cases were treated in \cite{JLNS}, and in our setting are very easy to describe: In the
first one $\chi$ is obtained from a twisting map $\psi\colon K^2\otimes K^n\longrightarrow K^n\otimes K^2$, as in Proposition~\ref{clasificacion
rango reducido 1}(2). In the second one by Proposition~\ref{rango reducido 1} all columns are standard, so it suffices to  consider standard
twisting maps compatible with the chosen quiver in the sense that $A_{\chi}(i,l)\ne 0$ if and only if the adjacency matrix of $Q$ has $1$ in the
entry $(i,l)$.

In the third case assume that the $2$-cycle is at the vertices $i,j$. Suppose that there is a reduced rank~$1$ twisting map $\chi$ such that
$Q_{\chi}=Q$. By Proposition~\ref{rango reducido 1} we know that the $l$-th columns of $\mathcal{A}_{\chi}$ is standard for all $l\notin\{i,j\}$.
This implies that if $\chi$ has an arrow from $i$ to $l$, then $F_0(\mathcal{A}_{\chi},i)\ne\emptyset$ (and similarly for $j$). In fact, we have
$$
\emptyset \ne F(A(i,l))\subseteq F_0(\mathcal{A}_{\chi},i),
$$
where inequality holds since $A(i,l) = \ide- A(l,l)$ and $A(l,l)$ is an idempotent $(0,1)$-matrix, while the inclusion is true by
Proposition~\ref{influencia de columna estandar}. Thus, in order to obtain such a twisting map $\chi$ we first construct a twisting map
$$
\psi \colon K^2\otimes K^n\longrightarrow K^n\otimes K^2,
$$
such that

\begin{itemize}

\smallskip

\item[-] $A_{\psi}(1,2)\ne 0\ne A_{\psi}(2,1)$,

\smallskip

\item[-] $F_0(\mathcal{A}_{\chi},i)\ne\emptyset$ if $Q$ has an arrow that starts at $i$ and does not end at $j$,

\smallskip

\item[-] $F_0(\mathcal{A}_{\chi},j)\ne\emptyset$ if $Q$ has an arrow that starts at $j$ and does not end at $i$.

\smallskip

\end{itemize}
Then we set $A_{\chi}(h,i)\coloneqq 0$ and $A_{\chi}(h,j)\coloneqq 0$ for $h\notin\{i,j\}$, and $A_{\chi}(f(a),f(b))\coloneqq A_{\psi}(a,b)$, where
$f(1)\coloneqq i$ and $f(2)\coloneqq j$. After that, for each vertex $l\in Q_0\setminus\{i,j\}$, we take a standard column $(A_{\chi}(u,l))_{u\in
Q_0}$ such that

\begin{itemize}

\smallskip

\item[-] $A_{\chi}(u,l)\ne 0$ if and only $Q$ has an arrow from $u$ to $l$,

\smallskip

\item[-] $F(A(v,l))\subseteq F_0(\mathcal{A},v)$ for all $v\in Q_0$ and $l\in Q_0\setminus\{i,j\}$.

\smallskip

\end{itemize}
By Proposition~\ref{influencia de columna estandar}, Corollary ~(4) is satisfied for all $l\notin\{i,j\}$. Since an straightforward computation
shows that it is satisfied for also for $i$ and $j$, this method produces all the twisting maps of reduced rank~$1$ with quiver $Q$.
\end{remark}

\section{Quiver associated with standard and quasi-standard twisting maps} In this section we will construct quivers that characterize completely the standard twisting maps. Moreover, the quiver indicates how one could possibly generate quasi-standard twisting maps out of a standard one.

\subsection{Characterization of standard twisted tensor products}
The aim of this section is to completely characterize the standard twisted tensor products of $K^n$ with $K^m$.  In particular we will prove that
they are algebras with square zero Jacobson radical. Our main result gen\-er\-al\-izes~\cite{C}*{Theorem 4.2}. Let
$$
\chi \colon K^m \ot K^n \longrightarrow K^n \ot K^m
$$
be a standard twisting map. As in Remark~\ref{tabla e ideales}, for each $j\in \mathds{N}^*_n$ and $i\in \mathds{N}^*_m$ we let $x_{ji}$ denote
$f_j\ot e_i$ . In that remark we saw that
$$
x_{ki}x_{jl} = A_{\chi}(i,l)_{kj}x_{kl}.
$$

\begin{remark}\label{la matriz Ail en el caso estandar} By Remark~\ref{Como son las columnas estandar} we know that
$$
A_{\chi}(i,l)_{kj}=\begin{cases}
\phantom{-}1& \text{if $k\in J_l(l)$, $i=l$ and  $j=k$,}\\
\phantom{-}1& \text{if $k \notin J_l(l)$, $i=l$ and $j=c_k(A_{\chi}(l,l))$,}\\
\phantom{-}1& \text{if $k\notin J_l(l)$, $i=i(k,l,\mathcal{A}_{\chi})$ (which means that $k\in F(A_{\chi}(i,l))$) and $j=k$,}\\
-1& \text{if $k\notin J_l(l)$, $i=i(k,l,\mathcal{A}_{\chi})$  and  $j=c_k(A_{\chi}(l,l))$,}\\
\phantom{-}0& \text{otherwise,}
            \end{cases}
$$
which implies that
$$
x_{ki}x_{jl}=\begin{cases}
\phantom{-}x_{kl}& \text{if $k\in J_l(l)$, $i=l$ and  $j=k$,}\\
\phantom{-}x_{kl}& \text{if $k \notin J_l(l)$, $i=l$ and $j=c_k(A(l,l))$,}\\
\phantom{-}x_{kl}& \text{if $k\notin J_l(l)$, $i=i(k,l,\mathcal{A})$ and $j=k$,}\\
-x_{kl}& \text{if $k\notin J_l(l)$, $i=i(k,l,\mathcal{A})$  and  $j=c_k(A(l,l))$,}\\
\phantom{-}0& \text{otherwise.}
\end{cases}
$$
\end{remark}

\begin{remark}\label{calculo del radical} If $j\notin J_l(l)$, then
\begin{align*}
& x_{ki}x_{jl}=A_{\chi}(i,l)_{kj} x_{kl}=0\quad\text{for all $k\ne j$ and all $i$}
\shortintertext{and}
& x_{jl}x_{ki}=B_{\chi}(k,j)_{il} x_{ji}=0\quad\text{for all $i\ne l$ and all $k$,}
\end{align*}
since $A_{\chi}(i,l)_{kj}\ne 0$ for some $i$ if and only if $j=c_k\bigl(A_{\chi}(l,l)\bigr)$, $B(k,j)_{il}\ne 0$ for some $k$ if and only if
$l=c_i\bigl(B_{\chi}(j,j)\bigr)$, $c_k\bigl(A_{\chi}(l,l)\bigr)$ belongs to $J_l(l)$ and $c_i\bigl(B_{\chi}(j,j)\bigr)$ belongs to
$\tilde{J}_j(j)$. From these facts it follows that $I\coloneqq \bigoplus_{j\notin J_l(l)} K x_{jl}$ is a square zero two-sided ideal of
$K^n\ot_{\chi} K^m$. Furthermore, each two-sided ideal including properly $I$ has an idempotent element~$x_{jl}$, and therefore it is not a
nilpotent ideal. So, $I$ is the Jacobson ideal $\J(K^n\ot_{\chi} K^m)$ of $K^n\ot_{\chi} K^m$.
\end{remark}

Let $\prescript{\chi}{}{Q}^{}_{}$ be the quiver with set of vertices $\prescript{\chi}{}{Q}^{}_{0}\coloneqq \{(j,i)\in \mathds{N}^*_n \times
\mathds{N}^*_m: j\in J_i(i)\}$, set of arrows $\prescript{\chi}{}{Q}^{}_{1} \coloneqq \{\alpha_{jl}: l\in \mathds{N}^*_m \text{ and } j\in
\mathds{N}^*_n \setminus J_l(l)\}$, and source and target maps $s,t \colon \prescript{\chi}{}{Q}^{}_{1} \longrightarrow
\prescript{\chi}{}{Q}^{}_{0}$ given by
$$
s(\alpha_{jl})\coloneqq (j,i(j,l, \mathcal{A}_{\chi}))\quad \text{and}\quad t(\alpha_{jl})\coloneqq (i(l,j, \mathcal{B}_{\chi}),l)
=(c_j(A_{\chi}(l,l)),l).
$$
Note that $s$ and $t$ are well defined by~Proposition~\ref{influencia de columna estandar} and Remark~\ref{c sub k esta en J sub k}, respectively.

\begin{remark}\label{hay flecha} By Remarks~\ref{c sub k esta en J sub k} and~\ref{Como son las columnas estandar}(4), Proposition~\ref{influencia
de columna estandar} and the definition of $\prescript{\chi}{}{Q}^{}_{}$, in $\prescript{\chi}{}{Q}^{}_{}$ there is an arrow from $(k,i)$ to
$(j,l)$ if and only if $A_{\chi}(i,l)_{kj}=-1$.
\end{remark}

Now we compute the products $x_{ki}x_{jl}$ in terms of the maps $s$ and $t$. By the computations made in Remark~\ref{la matriz Ail en el caso
estandar}, the following facts hold:
\begin{enumerate}[label=(\alph*)]

\smallskip

\item If $k\in J_i(i)$ and $j\in J_l(l)$, then $x_{ki}x_{jl}=\begin{cases}
                                                                        \phantom{-}x_{jl} & \text{if $(k,i)=(j,l)$,}\\
                                                                        -x_{kl} & \text{if $(k,i)=s(\alpha_{kl})$ and $(j,l)=t(\alpha_{kl})$,}\\
                                                                        \phantom{-}0& \text{otherwise.}
                                                             \end{cases}$

\smallskip

\item If $k\notin J_i(i)$ and $j\in J_l(l)$, then $x_{ki}x_{jl}=\begin{cases}
                                                                            x_{ki} & \text{if $(j,l)=t(\alpha_{ki})$,}\\
                                                                            0& \text{otherwise.}
                                                                \end{cases}$

\smallskip

\item If $k\in J_i(i)$ and $j\notin J_l(l)$, then $x_{ki}x_{jl}=\begin{cases}
                                                                          x_{jl} & \text{if $(k,i)=s(\alpha_{jl})$,}\\
                                                                          0& \text{otherwise.}
                                                                \end{cases}$

\smallskip

\item If $k\notin J_i(i)$ and $j\notin J_l(l)$, then $x_{ki}x_{jl}=0$.

\smallskip

\end{enumerate}

\begin{theorem}\label{caracterizacion de los productos tensoriales torcidos estandars} The twisted tensor product~$K^n\ot_{\chi} K^m$ is isomorphic
to the radical square zero algebra~$K \prescript{\chi}{}{Q}^{}_{}/\langle \prescript{\chi}{}{Q}^{2}_{1} \rangle$.
\end{theorem}

\begin{proof} An algebra morphism from $K \prescript{\chi}{}{Q}^{}_{}$ to $K^n\ot_{\chi} K^m$ is determined by a coherent choice of images of the
vertices and the arrows of $\prescript{\chi}{}{Q}$, since $K\prescript{\chi}{}{Q}$ is a tensor algebra on the vertices set algebra of the arrows
bimodule. For each $(j,l)\in \prescript{\chi}{}{Q}^{}_{0}$ set $\In(j,l)\coloneqq \{\alpha_{ki}\in \prescript{\chi}{}{Q}^{}_{1}:
(j,l)=t(\alpha_{ki})\}$. A straight\-for\-ward computation using (a)--(d) shows that
$$
\phi((j,l))\coloneqq x_{jl}+\sum_{\In(j,l)} x_{ki}\quad\text{if $j\in J_l(l)$}\qquad\text{and}\qquad \phi(\alpha_{jl})\coloneqq x_{jl}\quad\text{if
$j\notin J_l(l)$}
$$
is a coherent choice and hence defines an algebra morphism $\phi\colon K \prescript{\chi}{}{Q}^{}_{} \longrightarrow K^n\ot_{\chi} K^m$. Since the
elements $x_{jl}$'s generate linearly $K^n\ot_\chi K^m$, the morphism $\phi$ is surjective. Clearly a path of length two of $\prescript{\chi}{}{Q}$
has zero image. Since both algebras $K \prescript{\chi}{}{Q}^{}_{}/\langle \prescript{\chi}{}{Q}^{2}_{1} \rangle$ and $K^n\ot_\chi K^m$ have the
same dimension, the induced map
$$
\ov{\phi}\colon K \prescript{\chi}{}{Q}^{}_{}/\langle \prescript{\chi}{}{Q}^{2}_{1} \rangle \longrightarrow  K^n\ot_{\chi} K^m
$$
is an algebra isomorphism, as desired.
\end{proof}

The following remark generalizes the correct version of \cite{C}*{Theorem~4.6}.

\begin{remark}\label{construccion de los twisting estandar} The quiver $\prescript{\chi}{}{Q} = (\prescript{\chi}{}{Q}_0,\prescript{\chi}{}{Q}_1)$
associated with a standard twisting map $\chi$ of $K^m$ with $K^n$ fulfill the following properties:

\begin{enumerate}

\smallskip

\item $\prescript{\chi}{}{Q}_0\subseteq \mathds{N}_n^*\times \mathds{N}_m^*$ and for all $l\in \mathds{N}_m^*$ there exists $j\in \mathds{N}_n^*$
such that $(j,l)\in \prescript{\chi}{}{Q}_0$,

\smallskip

\item $\prescript{\chi}{}{Q}^1 = \{\alpha_{jl}:(j,l)\in (\mathds{N}_n^*\times \mathds{N}_m^*)\setminus \prescript{\chi}{}{Q}_0\}$,

\smallskip

\item for all $(j,l)\in (\mathds{N}_n^*\times \mathds{N}_m^*)\setminus \prescript{\chi}{}{Q}_0$ there exist $i\in \mathds{N}_m^*$ and $k\in
\mathds{N}_n^*$ such that $s(j,l) = (j,i)$ and $t(j,l) = (k,l)$.
\end{enumerate}
Conversely if $Q = (Q_0,Q_1)$ is a quiver that satisfies conditions~(1), (2) and (3), then there exists a unique standard twisting map $\chi$ of
$K^m$ with $K^n$, such that $Q = \prescript{\chi}{}{Q}$. Indeed, by Theorem~\ref{A y B determina el stm} in order to construct $\chi$ out of $Q$ it
suffices to determine families $(A(l))_{l\in N^*_m}$ and $(B(j))_{j\in N^*_n}$ of idempotent $0,1$-matrices $A(l)\in M_n(K)$ and $B(j)\in M_m(K)$
satisfying   conditions~(1) and (2) of that theorem. For this we define the $j$th row of $A(l)$ and the $l$th row of $B(j)$ as follows:

\begin{enumerate}

\smallskip

\item If $(j,l)\in Q_0$, then we set $A(l)_{jh} \coloneqq  \delta_{jh}$,

\smallskip

\item if $(j,l)\notin Q_0$, then we set $A(l)_{jh} \coloneqq  \delta_{kh}$, where $k$ is defined by $t(j,l) = (k,l)$,

\smallskip

\item if $(j,l)\in Q_0$, then we set $B(j)_{lh} \coloneqq  \delta_{lh}$,

\smallskip

\item if $(j,l)\notin Q_0$, then we set $B(j)_{lh} \coloneqq  \delta_{ih}$, where $i$ is defined by $s(j,l) = (j,i)$.

\smallskip

\end{enumerate}

\end{remark}

\subsection{Iterative construction of quasi-standard twisted tensor products}\label{subseccion construccion de quasi estandars}
The aim of this subsection is to give a method to construct the quasi-standard twisting tensor products of $K^m$ with $K^n$.  Through it we use
the notations of the previous sections, specially those introduced in the fifth one.  By Theorem~\ref{A y B determina el stm} we can associate in
an evident way a standard twisting map $\hat{\chi}$ to each quasi-standard twisting map~$\chi$. This allow us to associate a quiver
$\prescript{\chi}{}{Q}\coloneqq \prescript{\hat{\chi}}{}{Q}$ with each quasi-standard twisting tensor product $\chi \colon K^m\ot K^n
\longrightarrow K^n \ot K^m$ (Actually it is clear that the definition of $\prescript{\chi}{}{Q}^{}_{}$ introduced below Remark~\ref{calculo del
radical} has perfect sense for a  quasi-standard twisting map $\chi$ and that $\prescript{\chi}{}{Q} = \prescript{\hat{\chi}}{}{Q}$).

\begin{proposition}\label{ver un parametro} Let $\chi \colon K^m\ot K^n \longrightarrow K^n \ot K^m$ be a quasi-standard twisting map and let $k,d
\in \mathds{N}^*_n$. Assume that $\lambda \coloneqq A_{\chi}(u,l)_{kd}\ne 0$, or, which is equivalent, that there exist $w,v\in \mathds{N}^*_m$
such that $d\in \Supp\bigl((D_{(u,l)}^{wv})_{k*}\bigr)$. If $A_{\chi}(u,l)_{kk}=1$ and $A_{\hat{\chi}}(u,l)_{kd}=0$, then
\begin{equation}\label{igualdades quasi}
A_{\chi}(u,l)_{kd}=-A_{\chi}(v,l)_{kd}=A_{\chi}(v,l)_{kc_k},
\end{equation}
where $c_k=c_k(A_{\chi}(l,l))$. Moreover, there are the following arrows in the quiver of $\hat{\chi}$:
\begin{itemize}

\smallskip

\item[-] $\alpha_{kv}$, from $(k,u)$ to $(d,v)$,

\smallskip

\item[-] $\alpha_{kl}$, from $(k,u)$ to $(c_k,l)$,

\smallskip

\item[-] $\alpha_{dl}$, from $(d,v)$ to $(c_d,l)$, (where $c_d=c_d(A_{\chi}(l,l))$).

\smallskip

\end{itemize}
\end{proposition}

\begin{proof} In order to prove this result it suffices to verify that~$u\ne l$, $w=u$, $d\notin \{k,c_k\}$ and $v\ne u$. In fact, by
Lemma~\ref{condicion}, from the fact that~$d\notin \{k,c_k\}$ it follows that~$v\ne l$, and hence equalities~\eqref{igualdades quasi} hold by
Remark~\ref{cero fuera} and Definition~\ref{quasi standard}(3). Moreover, $\alpha_{kv}$, $\alpha_{kl}$ and $\alpha_{dl}$ are arrows of
$\prescript{\chi}{}{Q}^{}_{}$ since  $k\notin J_v(v)$ by Theorem~\ref{proto quasi estandar}, $k\notin J_l(l)$ by Theorem~\ref{proto quasi estandar}
and the fact that $A_{\chi}(u,l)_{kd}$, and $d\notin J_l(l)$ by Lemma~\ref{condicion}, and the starting and target vertices of these arrow are those
ones given in the statement because:

\begin{itemize}

\smallskip

\item[-] $i(k,v, \mathcal{A}_{\hat{\chi}})=u$, since $A_{\hat{\chi}}(u,v)_{kk}= B_{\hat{\chi}}(k,k)_{vu}= A_{\chi}(u,v)_{kk}=1$,

\smallskip

\item[-] $c_k\bigl(A_{\hat{\chi}}(v,v)\bigr)=d$, since $A_{\hat{\chi}}(v,v)_{kd}=A_{\chi}(v,v)_{kd}=1$,

\smallskip

\item[-] $i(k,l, \mathcal{A}_{\hat{\chi}})=u$, since $A_{\hat{\chi}}(u,l)_{kk} = B_{\chi}(k,k)_{lu} = A_{\chi}(u,l)_{kk}=1$,

\smallskip

\item[-] $c_k\bigl(A_{\hat{\chi}}(l,l)\bigr)=c_k$, since $A_{\hat{\chi}}(l,l)=A_{\chi}(l,l)$,

\smallskip

\item[-] $i(d,l, \mathcal{A}_{\hat{\chi}})=v$, since $A_{\hat{\chi}}(v,l)_{dd} = B_{\chi}(d,d)_{lv} = A_{\chi}(v,l)_{dd}=1$,

\smallskip

\item[-] $c_d\bigl(A_{\hat{\chi}}(l,l)\bigr)=c_d$, since $A_{\hat{\chi}}(l,l)=A_{\chi}(l,l)$,

\smallskip

\end{itemize}
where in the first and second item the last equality hold by Theorem~\ref{proto quasi estandar}.

So, we are reduced to prove that the facts pointed out at the beginning of the proof are true. But $u\ne l$, because otherwise
$$
A_{\chi}(l,l)_{kd} = \delta_{kd} \Rightarrow d=k \Rightarrow A_{\hat{\chi}}(l,l)_{kd}=A_{\chi}(l,l)_{kk}=1,
$$
which contradicts $A_{\hat{\chi}}(u,l)_{kd}=0$;  the equality $A_{\chi}(u,l)_{kk}=1$ say that $w=u$; the equalities
$$
A_{\hat{\chi}}(u,l)_{kk} =1\quad\text{and}\quad A_{\hat{\chi}}(u,l)_{kc_k}=-1
$$
imply that $d\notin \{k,c_k\}$ because~$A_{\hat{\chi}}(u,l)_{kd}=0$; and by Remark~\ref{identidad en la diagonal} we have $u\ne v$.
\end{proof}

\begin{definition}\label{lambda}
Let $\chi \colon K^m\ot K^n \longrightarrow K^n \ot K^m$ be a quasi-standard twisting map, $u, v, l\in \mathds{N}^*_m$,  $k\in J_u(u)$ and $d\in
J_v(v)$. Assume that there are the following arrows in the quiver of~$\hat{\chi}$:
\begin{itemize}

\smallskip

\item[-] $\alpha_{kv}$, from $(k,u)$ to $(d,v)$,

\smallskip

\item[-] $\alpha_{kl}$, from $(k,u)$ to $(c_k,l)$ (where $c_k=c_k(A_{\chi}(l,l))$),

\smallskip

\item[-] $\alpha_{dl}$, from $(d,v)$ to $(c_d,l)$ (where $c_d=c_d(A_{\chi}(l,l))$).

\smallskip

\end{itemize}
If $\Supp\bigl(D_{(u,l)}^{uv}\bigr)=\emptyset$, then for each $\lambda\in K$ we define the map $\chi_{_1}\colon K^m \ot K^n \longrightarrow K^n
\ot K^m$, by
$$
\setlength{\tabcolsep}{0.5pt}
\begin{tabular}{ >{$}l<{$}  >{$}c<{$}  >{$}l<{$} r}
A_{\chi_{_1}}(u,l)_{kd}  & \coloneqq &\phantom{-}\lambda,\\
A_{\chi_{_1}}(v,l)_{kd}  & \coloneqq &-\lambda,\\
A_{\chi_{_1}}(v,l)_{kc_k}& \coloneqq &\phantom{-}\lambda,\\
A_{\chi_{_1}}(u,l)_{kc_k}& \coloneqq &\phantom{-} A_{\chi}(u,l)_{kc_k}-\lambda,\\
A_{\chi_{_1}}(i,t)_{js}  & \coloneqq &\phantom{-} A_{\chi}(i,t)_{js} & \text{\, if $(i,t,j,s)\notin\{(u,l,k,d),(v,l,k,d),(v,l,k,c_k),(u,l,k,c_k)\}$.}\\
\end{tabular}
$$
If necessary to be more precise the map $\chi_{_1}$ will be denoted by $\Lambda^{\lambda}_{(k,u),(d,v),(c_k,l)}(\chi)$.
\end{definition}

\begin{remark}\label{lambda=0}
Note that if $\lambda=0$, then $\chi_{_1}=\chi$.
\end{remark}

\begin{remark}\label{no es necesario pedir A(u,l)kk=1} By Remark~\ref{hay flecha}, if the twisting map $\chi$ satisfies the assumptions made in
Definition~\ref{lambda}, then $A_{\hat{\chi}}(u,l)_{kc_k}=-1$, which by Remark~\ref{condicion para twisting standard} implies that
$A_{\chi_{_1}}(u,l)_{kk}=A_{\hat{\chi}}(u,l)_{kk}=1$. Moreover,  by the very definition of $\prescript{\chi}{}{Q}^{}_{}$, the existence of the
arrows $\alpha_{kv}$, $\alpha_{kl}$ and $\alpha_{dl}$ implies that $u$, $v$ and $l$ are three different elements of $\mathds{N}^*_m$. Since $k \in
J_u(u)\setminus (J_v(v) \cup J_l(l))$, $c_k \in J_l(l)$ and $d \in J_v(v)\setminus J_l(l)$, from this fact it follows that $k$, $c_k$ and $d$ are
three different elements of $\mathds{N}^*_n$, which implies that $A_{\hat{\chi}}(u,l)_{kd}=0$, because $\Supp(A_{\hat{\chi}}(u,l)_{k*})=
\{k,c_k\}$.  So, the hypothesis of Proposition~\ref{ver un parametro} are satisfied by $\chi_{_1}$, provided that it is a quasi-standard twisting
map.
\end{remark}

\begin{remark}\label{conservacion de rangos} Assume that the twisting map $\chi$ satisfies the assumptions made in Definition~\ref{lambda}. If
$\chi_{_1}$ is a (quasi-standard) twisting map, then $\Gamma_{\chi_{_1}}=\Gamma_{\chi}$ and $\tilde{\Gamma}_{\chi_{_1}}=\tilde{\Gamma}_{\chi}$.
\end{remark}

\begin{proposition}\label{construccion de quasiestandar} Let $\chi$, $\chi_{_1}$, $u$, $v$, $l$, $k$, $d$, $\alpha_{kv}$, $\alpha_{kl}$,
$\alpha_{kd}$ and $\lambda$ be as in Definition~\ref{lambda}. Assume that $\lambda\ne 0$. If $\chi_{_1}$ is a quasi-standard twisting map, then
\begin{enumerate}

\smallskip

\item $A_{\chi_{_1}}(i,u)_{kj}=A_{\hat{\chi}}(i,u)_{kj}$ and $A_{\chi_{_1}}(i,v)_{kj}=A_{\hat{\chi}}(i,v)_{kj}$ for all $i,j$,

\smallskip

\item $A_{\chi_{_1}}(u,l)$, $A_{\chi_{_1}}(v,l)$, $B_{\chi_{_1}}(d,k)$ and $B_{\chi_{_1}}(c_k,k)$ are idempotent matrices.

\smallskip

\end{enumerate}
Moreover, condition~(2) implies that $\chi_{_1}$ is a (quasi-standard) twisting map.
\end{proposition}

\begin{proof} By Remark~\ref{no es necesario pedir A(u,l)kk=1} we know that~$A_{\chi_{_1}}(u,l)_{kk}=A_{\chi}(u,l)_{kk}=1$, and so, by
Theorem~\ref{proto quasi estandar}(1),
$$
A_{\chi_{_1}}(i,u)_{kj} = \delta_{iu}\delta_{kj}= A_{\hat{\chi}}(i,u)_{kj}\quad\text{ for all $i,j$.}
$$
On the other hand, since $A_{\chi_{_1}}(u,l)_{kd}\ne 0$, by Theorem~\ref{proto quasi estandar}(2) we have
$$
A_{\chi_{_1}}(i,v)_{kj}=\begin{cases}
                                        \delta_{kj} - \delta_{jd} &\text{if $i=u$,}\\
                                        \delta_{jd}               &\text{if i=v,}\\
                                        0                         &\text{otherwise,}
                        \end{cases}
$$
and a direct computation using Theorem~\ref{A y B determina el stm} shows that $A_{\hat{\chi}}(i,v)_{kj}$ is given by the same formula (for this
computation is can be useful to see the proof of Proposition~\ref{ver un parametro}). This finishes the proof of condition~(1). By items~(1)
and~(2) of Proposition~\ref{caracterizacion}, condition~(2) is also satisfied. Finally, by Remark~\ref{cuando son idempotentes ortogonales}
condition~(2) is sufficient for $\chi_{_1}$ to be a twisting map, since $\sum_i A_{\chi_{_1}}(i,l)=\ide_{K^n}$ and $\sum_j
B_{\chi_{_1}}(j,k)=\ide_{K^m}$.
\end{proof}

\begin{corollary} Under the assumptions  made in Definition~\ref{lambda}, if $\chi$ is standard, then $\chi_{_1}$ is a
quasi-standard twisting map.
\end{corollary}

\begin{proof} When $\lambda =0$  this is evident, whereas when it is different from $0$ a straightforward computation shows that
$A_{\chi_{_1}}(u,l)$, $A_{\chi_{_1}}(v,l)$, $B_{\chi_{_1}}(d,k)$ and $B_{\chi_{_1}}(c_k,k)$ are idempotent matrices.
\end{proof}

\begin{remark}\label{deformacion} A straightforward computation shows that if $\chi_1$ of Definition 7.7 is a twisting map, then the construction
$\chi_{_1}$ out of $\chi$ corresponds to a formal deformation in the sense of Gerstenhaber. To be more precise, the multiplication map
$\mu_{\chi_{_1}}$ of $\chi_{_1}$ is given by
$$
\mu_{\chi_{_1}}(a\ot b)=\mu_0(a\ot b)+\lambda \mu_1(a\ot b),
$$
where $\mu_0$ is the multiplication in $D\coloneqq K^n\ot_{\chi} K^m$ and $\mu_1\colon D\ot D \longrightarrow D$ is the map defined by
\begin{align*}
&\mu_1( x_{ku}\ot \hspace{0.5pt} x_{dl}\hspace{0.5pt}) \hspace{0.5pt} = \phantom{-}1\\
&\mu_1(x_{kv}\ot \hspace{-0.8pt} x_{c_k l}\hspace{-0.8pt})\hspace{-0.6pt}=\phantom{-}1\\
&\mu_1( x_{kv} \ot \hspace{0.8pt} x_{dl}\hspace{0.8pt}) \hspace{0.7pt} =-1\\
&\mu_1(x_{ku}\ot\hspace{-0.5pt} x_{c_k l}\hspace{-0.5pt})\hspace{-0.5pt}=-1
\shortintertext{and}
&\mu_1(x_{pq}\ot \hspace{0.5pt} x_{rs}\hspace{0.5pt})\hspace{0.5pt} =\phantom{-}0 && \text{if $(x_{pq},x_{rs})\notin\{(x_{ku},x_{dl}),(x_{kv},x_{dl}),(x_{ku},x_{c_k l}),(x_{kv},x_{c_k l})\}.$}
\end{align*}
\end{remark}

\begin{remark} Each quasi-standard twisting map can be obtained from a standard twisting map by applying repeatedly the construction of
Definition~\ref{lambda}, thus adding parameters $\lambda_1,\lambda_2,\lambda_3,\dots$, obtaining quasi-standard twisting maps
$\chi_{_1},\chi_{_2},\chi_{_3},\dots$.
\end{remark}

\begin{remark} Let $\chi \colon K^m\ot K^n \longrightarrow K^n \ot K^m$ be a quasi-standard twisting map. For each $u,l,v\in \mathds{N}^*_m$ and
$k,d\in \mathds{N}^*_n$ such that $u$, $l$ and $v$ are three different elements of $\mathds{N}^*_m$, $k\in J_u(u)$, $d\in J_v(v)$ and
$A_{\chi}(u,l)_{kd}\ne 0$, the quiver of $\hat{\chi}$ has a triangle with vertices $(k,u)$, $(d,v)$ and $(c_k,l)$, and arrows $\alpha_{kv}$,
$\alpha_{kl}$ and $\alpha_{dl}$, from $(k,u)$ to $(d,v)$,  $(k,u)$ to $(c_k,l)$, and $(d,v)$ to $(c_k,l)$, respectively. In fact, this follows from
the previous results and the fact that $c_k=c_d$ by Definition~\ref{quasi standard}(4).
\end{remark}

\subsection{Jacobson radical of quasi-standard twisted tensor products}
Let $\chi\colon K^m\ot K^n \longrightarrow K^n \ot K^m$ be a quasi-standard twisting map. For each $j\in \mathds{N}^*_n$ and $l\in \mathds{N}^*_m$,
let $x_{jl}$ be as in Remark~\ref{tabla e ideales}. In this subsection we prove that, as in the case when $\chi$ is standard, the Jacobson ideal
$\J(C)$ of $C\coloneqq K^n\ot_{\chi} K^m$ is the ideal $I\coloneqq \bigoplus_{j\notin J_l(l)} K x_{jl}$ of $C$ (however unlike what happens in the
standard case, when $\chi$ is not standard $I$ can be not a square zero ideal). As a consequence there exists a subalgebra
$A\simeq \frac{C}{\J(C)}$ of $C$ such that $C= A\bigoplus \J(C)$.

\begin{theorem}\label{radical en el caso quasiestandar} Let $\chi$, $C$ and $I$ be as above. Then $I$ is the Jacobson ideal of $C$.
\end{theorem}

\begin{proof} For each $j,k\in \mathds{N}^*_n$ and $i,l\in \mathds{N}^*_m$. If $i\ne l$ or $k\ne j$, then
$$
x_{ki}x_{jl}=A_{\chi}(i,l)_{kj}x_{kl}\in I.
$$
In fact, if $k\notin J_l(l)$, then this is true by the very definition of $I$, and if $k\in J_l(l)$, then it is true because
$A_{\chi}(i,l)_{kj}=0$ by Theorem~\ref{proto quasi estandar}(1). So, $I$ is a two-sided ideal of $C$. To finish the proof it suffices to show that
$$
x_{j_1l_1}\cdots x_{j_{n+1}l_{n+1}} =0\qquad\text{for each $x_{j_1l_1},\dots, x_{j_{n+1}l_{n+1}}\in I$.}
$$
By the above argument this is true if there exist $s<t$ such that $j_s\in J_{l_t}(l_t)$. So, assume that this is not the case. Then, since
$j_1,\dots, j_{n+1}\in \mathds{N}^*_n$ there exist $u<v$ such that $j_u=j_v$, and so
$$
x_{j_ul_u}\cdots x_{j_vl_v}=\prod_{h=0}^{v-u-1} A_{\chi}(l_{u+h},l_{u+h+1})_{j_uj_{u+h+1}}x_{j_ul_v}=0,
$$
because  $A_{\chi}(l_{v-1},l_v)_{j_uj_v}=0$ by the fact that $j_u\notin J_{l_{v-1}}(l_{v-1})$ and Theorem~\ref{proto quasi estandar}(1).
\end{proof}

\begin{corollary}\label{caracterizacion quasiestandar} Under the hypothesis of Theorem\ref{radical en el caso quasiestandar}, the quotient algebra
$\frac{C}{\J(C)}$ is isomorphic to the direct product $\prod_{j\in J_l(l)} K x_{jl}$ of fields, and there exists a subalgebra $A\simeq
\frac{C}{\J(C)}$ of $C$ such that $C= A\bigoplus \J(C)$.
\end{corollary}

\begin{proof} The first assertion follows from the fact that for each $i,l\in \mathds{N}^*_m$, $k\in J_i(i)$ and $j\in J_l(l)$, with $j\ne k$,
\begin{align*}
& x_{jl}x_{jl}=A(l,l)_{jj}x_{jl}= x_{jl},\\
&x_{ki}x_{jl}=A(i,l)_{kj}x_{kl}\in I \text{ if $k\notin J_l(l)$}
\shortintertext{and}
&x_{ki}x_{jl}= A(i,l)_{kj}x_{kl}=0 \text{ if $k\in J_l(l)$.}
\end{align*}
The second assertion follows now by a direct application of the Principal Theorem of Wedderburn-Malcev (\cite{Pierce}*{Chapter 11}).
\end{proof}

\begin{remark}\label{radical quasiestandar no estandar}
It is easy to check that if $\chi\colon K^m\ot K^n \longrightarrow K^n \ot K^m$ is a quasi-standard twisting map that it is not standard, then
$\J(C)^2\ne 0$.
\end{remark}

\section{Low dimensional cases}
In this section we determine the twisting maps of $K^3$ with $K^3$. To achieve this it is convenient first to describe in detail the twisting maps of $K^2$ with $K^2$ and the twisting map of $K^2$ with $K^3$.

\subsection[Twisting maps of $K^2$ with $K^2$]{Twisting maps of $\bm{K^2}$ with $\bm{K^2}$}\label{K2 por K2}
We first use our results to obtain a classification of all twisting maps $\chi\colon K^2\otimes K^2 \longrightarrow K^2\otimes K^2$ in a direct
way. This classification was already obtained by~\cite{LN}. By Corollary~\ref{propiedades de la matriz de rangos} and
Proposition~\ref{matrices de rango equivalentes} we can assume that the $\mathcal{A}_{\chi}$-rank matrix is one of the following:
$$
\Gamma_1=\begin{pmatrix}
2&0\\
0&2
\end{pmatrix},\quad
\Gamma_2=\begin{pmatrix}
2&1\\
0&1
\end{pmatrix}\quad\text{or}\quad
\Gamma_3=\begin{pmatrix}
1&1\\
1&1
\end{pmatrix}.
$$

\paragraph{First case}
If the $\mathcal{A}_{\chi}$-rank matrix is $\Gamma_1$, then $A(1,1)=A(2,2)=\ide$. Consequently $\chi$ is the flip and
$K^2\otimes_\chi K^2\cong K^4$.

\paragraph{Second case}
If the $\mathcal{A}_{\chi}$-rank matrix is $\Gamma_2$, then $\chi$ is a standard twisting map (use Proposition~\ref{forma general casi triangular}),
and one verifies readily that $\chi$ is equivalent via identical permutations
in rows and columns to the twisting map $\chi'$ with quiver
\begin{center}
\begin{tikzpicture}[scale=0.8]
\draw (-1.5,1) node[right,text width=1cm]{$\prescript{\chi'}{}{Q}^{}_{}=$};
\draw [-stealth] (0.1,0.1) -- (1.9,1.9);
\draw (0.9,0.8) node[right,text width=2cm]{$\alpha_{22}$};
\filldraw [black]  (0,0)    circle (2pt)
[black]  (0,2)    circle (2pt)
[black]  (2,2)    circle (2pt);
\draw (2,0) circle (3pt);
\draw (0,2) node[right,text width=0.8cm]{$(1,1)$};
\draw (2,2) node[right,text width=0.8cm]{$(1,2)$};
\draw (0,0) node[right,text width=0.8cm]{$(2,1)$};
\draw[gray] (2,0) node[right,text width=0.8cm, black]{$(2,2)$};
\draw (3,1) node[right,text width=0.1cm]{.};
\end{tikzpicture}
\end{center}
Here the bullets represent the vertices of $\prescript{\chi'}{}{Q}^{}_{}$, and the white circle in the coordinate $(2,2)$, indicates that the arrow
$\alpha_{22}$ starts at the $2$-th row and ends at the $2$-th column. It is easy to recover the matrices of $\mathcal{A}_{\chi'}$ from $\prescript{\chi'}{}{Q}^{}_{}$. We have:
$$
A(1,1)=\ide,\quad A(2,1)=0,\quad A(2,2)=\begin{pmatrix}
                        1&0\\
                        1&0
                       \end{pmatrix}\quad\text{and}\quad A(1,2)=\begin{pmatrix}
                        0&0\\
                        -1&1
                       \end{pmatrix}.
$$
In the sequel we will simply represent the quivers of this twisting map and of its equivalent twisting  maps as
$$
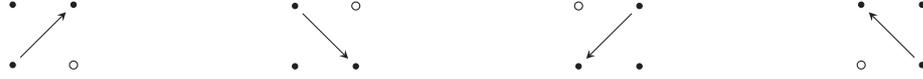
\begin{tikzpicture}
\draw [-stealth] (0.1,0.1) -- (0.7,0.7);
\filldraw [black]  (0,0)    circle (1pt)
[black]  (0,0.8)    circle (1pt)
[black]  (0.8,0.8)    circle (1pt);
\draw (0.8,0) circle (1.5pt);
\end{tikzpicture}\qquad \qquad\qquad\qquad
\begin{tikzpicture}
\draw [-stealth] (0.1,0.7) -- (0.7,0.1);
\filldraw [black]  (0,0)    circle (1pt)
[black]  (0,0.8)    circle (1pt)
[black]  (0.8,0)    circle (1pt);
\draw (0.8,0.8) circle (1.5pt);
\end{tikzpicture} \qquad\qquad\qquad\qquad
\begin{tikzpicture}
\draw [-stealth] (0.7,0.7) -- (0.1,0.1);
\filldraw [black]  (0,0)    circle (1pt)
[black]  (0.8,0.8)    circle (1pt)
[black]  (0.8,0)    circle (1pt);
\draw (0,0.8) circle (1.5pt);
\end{tikzpicture}\qquad\qquad\qquad\qquad
\begin{tikzpicture}
\draw [-stealth] (0.7,0.1) -- (0.1,0.7);
\filldraw [black]  (0.8,0)    circle (1pt)
[black]  (0,0.8)    circle (1pt)
[black]  (0.8,0.8)    circle (1pt);
\draw (0,0) circle (1.5pt);
\end{tikzpicture}
$$
where there is a bullet in the position $(j,i)$ if (j,i) is a vertex (thus $j\in J_i(i))$; and there is a white circle  in the position $(j,l)$ if
the quiver has an arrow $\alpha_{jl}$ that starts at the $j$-th row and ends at the $l$-th column (it is unique). The quivers associated with
standard twisting maps of $K^3$ with $K^2$ and of $K^3$ with $K^3$ will be represented by diagrams constructed following the same instructions.

\paragraph{Third case}
If the $\mathcal{A}_{\chi}$-rank matrix is $\Gamma_3$, then by Remark~\ref{rango uno}, there exist $a,a'\in K$, such that
\begin{alignat*}{2}
   & A_{\chi}(1,1)=\begin{pmatrix}
                             a&1-a\\
                             a&1-a
            \end{pmatrix}, && \qquad A_{\chi}(2,1)=\begin{pmatrix}
                                                            1-a&a-1\\
                                                            -a&a
                                           \end{pmatrix}, \\[3pt]
   & A_{\chi}(1,2)=\begin{pmatrix}  1-a'&a'-1\\
                             -a'&a'
            \end{pmatrix}, && \qquad   A_{\chi}(2,2)=\begin{pmatrix}
                                                           a'&1-a'\\
                                                           a'&1-a'
                                             \end{pmatrix}.
\end{alignat*}
Thus, by~\eqref{eq:rel entre A(i,l) y B(j,k)} we have $B_{\chi}(1,1)= \begin{psmallmatrix} a&1-a\\ 1-a'&a'\end{psmallmatrix}$. Therefore $a'=1-a$  by Proposition~\ref{diagonales uno} and Remark~\ref{rango uno}, and so
\begin{alignat*}{2}
   & A_{\chi}(1,1)=\begin{pmatrix}
                             a&1-a\\
                             a&1-a
            \end{pmatrix}, && \qquad A_{\chi}(2,1)=\begin{pmatrix}
                                                            1-a&a-1\\
                                                            -a&a
                                           \end{pmatrix}, \\[3pt]
   & A_{\chi}(1,2)=\begin{pmatrix}  a & -a\\
                             a-1 & 1-a
            \end{pmatrix}, && \qquad   A_{\chi}(2,2)=\begin{pmatrix}
                                                           1-a & a\\
                                                           1-a & a
                                             \end{pmatrix}.
\end{alignat*}
Now a direct computation using~\eqref{eq:rel entre A(i,l) y B(j,k)} shows that $B_{\chi}(i,j)=A_{\chi}(i,j)$ for $i,j=1,2$, which enables one
to check easily that the conditions of  Proposition~\ref{caracterizacion} are satisfied. Hence we have a family of twisting maps parameterized by
$a\in K$. Applying Proposition~\ref{matrices de rango equivalentes} we see that the twisting maps corresponding to $a$ and $1-a$ are isomorphic.
Moreover, using again the same proposition, we check that these are the only isomorphisms between these twisting maps. If $a\in\{0,1\}$, then the
twisting map is standard and the quiver is one of
$$
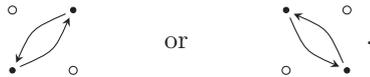
\begin{tikzpicture}
\begin{scope}[xshift=0,yshift=0]
\draw [-stealth] (0.1,0.05) .. controls (0.55,0.25) .. (0.75,0.7);
\draw[-stealth] (0.7,0.75) .. controls (0.25,0.55) .. (0.05,0.1);
\filldraw [black]  (0.8,0.8)    circle (1pt)
[black]  (0,0)    circle (1pt);
\draw (0,0.8) circle (1.5pt);
\draw (0.8,0) circle (1.5pt);
\end{scope}
\begin{scope}
\draw (2.2,0.37) node[text width=0.4cm]{or};
\end{scope}
\begin{scope}[xshift=3.6cm]
\draw [-stealth] (0.75,0.1) .. controls (0.55,0.55) .. (0.1,0.75);
\draw[-stealth] (0.05,0.7) .. controls (0.25,0.25) .. (0.7,0.05);
\filldraw [black]  (0,0.8)    circle (1pt)
[black]  (0.8,0)    circle (1pt);
\draw (0.8,0.8) circle (1.5pt);
\draw (0,0) circle (1.5pt);
\end{scope}
\begin{scope}
\draw (4.7,0.37) node[text width=0.1cm]{.};
\end{scope}
\end{tikzpicture}
$$
On the other hand, by Proposition~\ref{representacion} and
Remark~\ref{imagen de rho} we know that for $a\notin\{0,1\}$, the map  $\rho_1\colon K^2\ot_\chi K^2\longrightarrow M_2(K)$, given by
$$
\rho_1(f_j\otimes 1)\coloneqq E^{jj} \quad\text{and}\quad \rho_1 (1\otimes e_i)\coloneqq A(i,1),
$$
is an algebra isomorphism. So we obtain in this case, modulo isomorphism, four different algebras.

\subsection[Twisting maps of $K^3$ with $K^2$]{Twisting maps of $\bm{K^3}$ with $\bm{K^2}$}\label{de K3 con K2}
Now we use our results to classify all the twisting maps $\chi$ of $K^3$ with $K^2$ (By Remark~\ref{A sub Chi es B sub tau chi tau},
Proposition~\ref{A standard equivale a B standard}  and Proposition~\ref{A quasi standard equivale a B quasi standard}, this immediately gives a
similar classification for the twisting maps of $K^2$ with $K^3$). By Corollary~\ref{propiedades de la matriz de rangos} and
Proposition~\ref{matrices de rango equivalentes} we can assume that the $\mathcal{A}_{\chi}$-rank matrix is one of the following:
\begin{gather*}
\Gamma_1=\begin{pmatrix}
2&0&0\\
0&2&0\\
0&0&2
\end{pmatrix},\quad
\Gamma_2=\begin{pmatrix}
2&0&0\\
0&2&1\\
0&0&1
\end{pmatrix},\quad
\Gamma_3=\begin{pmatrix}
2&1&1\\
0&1&0\\
0&0&1
\end{pmatrix},
\quad
\Gamma_4=\begin{pmatrix}
2&1&0\\
0&1&1\\
0&0&1
\end{pmatrix},\\
\Gamma_5=\begin{pmatrix}
2&0&0\\
0&1&1\\
0&1&1
\end{pmatrix}, \quad
\Gamma_6=\begin{pmatrix}
1&0&0\\
0&1&1\\
1&1&1
\end{pmatrix},\quad
\Gamma_7=\begin{pmatrix}
1&0&1\\
1&1&0\\
0&1&1
\end{pmatrix}.
\end{gather*}
By Proposition~\ref{forma general casi triangular}, except perhaps in the cases $\Gamma_5$  and $\Gamma_6$, the matrices $A_{\chi}(l,l)$
are $0,1$-matrices, which (since the reduced rank of $\chi$ is less than or equal to~$1$) implies that $\chi$ is a standard twisting map.  So we list all the possible standard twisting maps (for this we use the method given in Remark~\ref{construccion de los twisting estandar}):

\begin{table}[H]
\caption{Standard twisting maps of $K^3$ with $K^2$}
\centering
\setlength{\tabcolsep}{8pt}
\ra{1.2}
\begin{tabu}{cccccc}
\toprule
$\#$ &$\sum \Tr$ & quiver& $\Gamma_{\chi}$ & $\tilde{\Gamma}_{\chi}$ & $\#$ equiv. \\
\midrule
1. & 6 & \raisebox{-3\height/8}{\begin{tikzpicture}[scale=0.65]
\filldraw [black]  (0,0)    circle (1pt)
[black]  (0,0.8)    circle (1pt)
[black]  (0.8,0)    circle (1pt)
[black]  (0.8,0.8)    circle (1pt)
[black]  (1.6,0)    circle (1pt)
[black]  (1.6,0.8)    circle (1pt);
\end{tikzpicture} }
& $\begin{psmallmatrix} 2&0&0\\ 0&2&0\\ 0&0&2 \end{psmallmatrix}$ & $\begin{psmallmatrix} 3&0\\ 0&3 \end{psmallmatrix}$& 1\\
&&\\
2. & 5 & \raisebox{-3\height/8}{\begin{tikzpicture}[scale=0.65]
\filldraw [black]  (0,0)    circle (1pt)
[black]  (0,0.8)    circle (1pt)
[black]  (0.8,0)    circle (1pt)
[black]  (0.8,0.8)    circle (1pt)
[black]  (1.6,0.8)    circle (1pt);
\draw  (1.6,0)    circle (1.5pt);
\draw[-stealth] (0.9,0.1)--(1.5,0.7);
\end{tikzpicture} }
& $\begin{psmallmatrix} 2&0&0\\ 0&2&1\\ 0&0&1 \end{psmallmatrix}$ & $\begin{psmallmatrix} 3&1\\ 0&2 \end{psmallmatrix}$& 12\\
&\\
3. & 4 & \raisebox{-3\height/8}{\begin{tikzpicture}[scale=0.65]
\filldraw [black]  (0,0)    circle (1pt)
[black]  (0,0.8)    circle (1pt)
[black]  (0.8,0.8)    circle (1pt)
[black]  (1.6,0.8)    circle (1pt);
\draw  (1.6,0)    circle (1.5pt);
\draw  (0.8,0)    circle (1.5pt);
\draw[-stealth] (0.1,0.05)--(1.5,0.75);
\draw[-stealth] (0.1,0.1)--(0.7,0.7);
\end{tikzpicture} }
& $\begin{psmallmatrix} 2&1&1\\ 0&1&0\\ 0&0&1 \end{psmallmatrix}$ & $\begin{psmallmatrix} 3&2\\ 0&1 \end{psmallmatrix}$& 6\\
&\\
4. & 4 & \raisebox{-3\height/8}{\begin{tikzpicture}[scale=0.65]
\filldraw [black]  (0,0)    circle (1pt)
[black]  (0,0.8)    circle (1pt)
[black]  (0.8,0.8)    circle (1pt)
[black]  (1.6,0)    circle (1pt);
\draw  (1.6,0.8)    circle (1.5pt);
\draw  (0.8,0)    circle (1.5pt);
\draw[-stealth] (0.9,0.7)--(1.5,0.1);
\draw[-stealth] (0.1,0.1)--(0.7,0.7);
\end{tikzpicture} }
& $\begin{psmallmatrix} 2&1&0\\ 0&1&1\\ 0&0&1 \end{psmallmatrix}$ & $\begin{psmallmatrix} 2&1\\ 1&2 \end{psmallmatrix}$& 12\\
&\\
5. & 4 & \raisebox{-3\height/8}{\begin{tikzpicture}[scale=0.65]
\filldraw [black]  (0,0)    circle (1pt)
[black]  (0,0.8)    circle (1pt)
[black]  (0.8,0.8)    circle (1pt)
[black]  (1.6,0)    circle (1pt);
\draw  (1.6,0.8)    circle (1.5pt);
\draw  (0.8,0)    circle (1.5pt);
\draw[-stealth] (0.9,0.75)..controls (1.35,0.55)..(1.55,0.1);
\draw[-stealth] (1.5,0.05)..controls (1.05,0.25)..(0.85,0.7);
\end{tikzpicture} }
& $\begin{psmallmatrix} 2&0&0\\ 0&1&1\\ 0&1&1 \end{psmallmatrix}$ & $\begin{psmallmatrix} 2&1\\ 1&2 \end{psmallmatrix}$& 6\\
&\\
6. & 3 & \raisebox{-3\height/8}{\begin{tikzpicture}[scale=0.65]
\filldraw [black]  (0,0.8)    circle (1pt)
[black]  (0.8,0.8)    circle (1pt)
[black]  (1.6,0)    circle (1pt);
\draw  (1.6,0.8)    circle (1.5pt);
\draw  (0.8,0)    circle (1.5pt);
\draw  (0,0)    circle (1.5pt);
\draw[-stealth] (0.9,0.75)..controls (1.35,0.55)..(1.55,0.1);
\draw[-stealth] (1.5,0.05)..controls (1.05,0.25)..(0.85,0.7);
\draw[-stealth] (1.5,-0.0)..controls (1.15,0.1)..(0.1,0.75);
\end{tikzpicture} }
& $\begin{psmallmatrix} 1&0&0\\ 0&1&1\\ 1&1&1 \end{psmallmatrix}$ & $\begin{psmallmatrix} 2&2\\ 1&1 \end{psmallmatrix}$& 12\\

\bottomrule
\end{tabu}
\end{table}
\noindent Here $\sum \Tr \coloneqq \sum_i \Tr(A(i,i))=\sum_j \Tr(B(j,j))$ and $\#$ equiv. indicates how many equivalent standard twisting maps there are (Here and in the sequel we say that two standard twisting maps $K^m$ with $K^n$ are equivalent if they are isomorphic).

\smallskip

If $\Gamma_{\chi}=\Gamma_5$, then $\chi$ is a direct sum of two twisting maps, and the twisted tensor product algebra is isomorphic
to $K^2\oplus A$, where $A$ is a twisted tensor product $K^2\otimes_{\chi'} K^2$ with $\Gamma_{\chi'}=\begin{psmallmatrix} 1 & 1 \\ 1 & 1 \end{psmallmatrix}$, so either it is standard (recovering the case $\#5$ in the list), or it corresponds to a value of $a\notin\{0,1\}$ in the third case of Subsection~\ref{K2 por K2}, and we obtain an algebra isomorphic to $K^2\oplus M_2(K)$.

If $\Gamma_{\chi}=\Gamma_6$, then by Proposition~\ref{forma general casi triangular} the first column of $\mathcal{A}_{\chi}$ is a
standard column, so that either $A_{\chi}(1,1)= \begin{psmallmatrix} 1 & 0\\ 1 & 0 \end{psmallmatrix}$ and
$A_{\chi}(3,1)= \begin{psmallmatrix*}[r] 0 & 0\\ -1 & 1 \end{psmallmatrix*}$, or
$A_{\chi}(1,1)= \begin{psmallmatrix} 0 & 1\\ 0 & 1 \end{psmallmatrix}$ and
$A_{\chi}(3,1)= \begin{psmallmatrix*}[r] 1 & -1\\ 0 & 0\end{psmallmatrix*}$  (by Proposition~\ref{matrices de rango equivalentes} we can assume,
and we do it, that $A_{\chi}(1,1)= \begin{psmallmatrix} 1 & 0\\ 1 & 0 \end{psmallmatrix}$).
Moreover, by Proposition~\ref{extension de twistings}
the matrices $A_{\chi}(i,j)$ for $i,j\in\{2,3\}$ define a $2$ times $2$ twisting map $\chi'$ with
$\Gamma_{\chi'}=\begin{psmallmatrix} 1&1\\ 1&1 \end{psmallmatrix}$, which is either standard, or has $1-a,a\notin\{0,1\}$ on the diagonal
of $A_{\chi'}(3,3)=A_{\chi}(3,3)$. But Theorem~\ref{teorema principal} shows that
$$
\{2\}=F(A(3,1))\subseteq F_0(\mathcal{A},3).
$$
So $A(3,3)_{22}=1$  and the twisting map is standard, corresponding to the sixth case  on the list.

If $\Gamma_{\chi}=\Gamma_7$, then the twisting map should be standard, but no standard twisting map $\chi$ yields $\Gamma_{\chi}=\Gamma_7$, so there is no twisting map in this case.

\subsection[Twisting maps of $K^3$ with $K^3$]{Twisting maps of $\bm{K^3}$ with $\bm{K^3}$}
We next aim is to construct (up to isomorphisms) all the twisting maps of $K^3$ with $K^3$.  Since in the appendix we list all standard and quasi-standard twisting maps of $K^3$ with $K^3$, for this purpose in this section we only need construct the twisting maps that are not quasi-standard.
In order to carry out this task in addition to the previous results, we will use the following ones:

\begin{remark}\label{en rango dos es standard}
Let $\mathcal{A}=\left(A(i,l)\right)_{i,l\in \mathds{N}^*_m}$ be a pre-twisting of $K^m$ with $K^n$. If the $l$-th column of $\mathcal{A}$ has reduced rank~$1$ and $A(l,l)$ is a $0,1$-matrix, then the $l$-th column of $\mathcal{A}$ is standard.
\end{remark}

\begin{proposition}\label{columna de unos no quasi-standard}
Let $\chi\colon K^m\ot K^3 \longrightarrow K^3\ot K^m$ be a twisting map and let $i_1$, $i_2$ and~$i_3$ be three different elements of $\mathds{N}^*_m$ such that $A_{\chi}(i_2,i_1)\ne 0\ne A_{\chi}(i_3,i_1)$ and $A_{\chi}(i_1,i_1)$ is equivalent to the matrix
$$
\begin{pmatrix} 1&0&0\\1&0&0\\1&0&0 \end{pmatrix}
$$
via identical permutations in rows and columns. If the $i_1$-th column of~$\mathcal{A}_{\chi}$ is not quasi-stan\-dard, then the following facts hold:

\begin{enumerate}

\smallskip

\item $A_{\chi}(i_2,i_3)\ne 0\ne A_{\chi}(i_3,i_2)$ and neither the $i_2$-th nor the $i_3$-th column of~$\mathcal{A}_{\chi}$ are quasi-standard columns.

\smallskip

\item If $A_{\chi}(i_1,i_1)=\begin{psmallmatrix} 1&0&0\\1&0&0\\1&0&0 \end{psmallmatrix}$, then
$$
\quad\qquad A_{\chi}(i_2,i_1)_{22}=A_{\chi}(i_2,i_2)_{22}=A_{\chi}(i_2,i_3)_{22}, \qquad A_{\chi}(i_3,i_1)_{22}=A_{\chi}(i_3,i_2)_{22}=A_{\chi}(i_3,i_3)_{22},
$$
and there exist $z\in K^{\times}$ and $\alpha\in K^{\times}\setminus\{1\}$ such that
\begin{align*}
&\quad\qquad A_{\chi}(i_2,i_1)=\begin{pmatrix}
                            0 & 0 & 0\\
                            -\alpha-z &\alpha &z \\
                            \alpha-1-\frac{\alpha(1-\alpha)}{z}&\frac{\alpha(1-\alpha)}{z}& 1-\alpha
             \end{pmatrix}
\shortintertext{and}
&\quad\qquad A_{\chi}(i_3,i_1)=\begin{pmatrix}
                            0 & 0 & 0\\
                            \alpha+z-1 &1-\alpha &-z \\
                            \frac{\alpha(1-\alpha)}{z}-\alpha& -\frac{\alpha(1-\alpha)}{z}& \alpha
             \end{pmatrix}.
\end{align*}
\end{enumerate}
\end{proposition}

\begin{proof}
Without loss of generality we can assume that $A_{\chi}(i,1)=0$ for $i>3$ and that $i_1=1$, $i_2=2$ and $i_3=3$.
By items~(1) and (3) of Corollary~\ref{caracterizacion en terminos de las A(i,j)s} we know that $A_{\chi}(1,1)A_{\chi}(i,1)=0$ for all $i> 1$ and
that $A_{\chi}(1,1)+A_{\chi}(2,1)+A_{\chi}(3,1)=\ide_3$.  Hence there exists $\alpha\in K$ such that
$$
A_{\chi}(2,1)=\begin{pmatrix}0&0&0\\ * &\alpha &* \\ *&*& 1-\alpha\end{pmatrix}\quad\text{and}\quad
A_{\chi}(3,1)=\begin{pmatrix}0&0&0\\ * &1-\alpha &* \\ *&*& \alpha\end{pmatrix},
$$
where the $*$'s denote arbitrary elements of $K$. Moreover,  by Proposition~\ref{cuando es cuasiestandard}(2) we know that $\alpha\notin\{0,1\}$.
Let $z\coloneqq A_{\chi}(3,1)_{23}$. Since the sum of the entries of each row of $A_{\chi}(2,1)$ and~$A_{\chi}(3,1)$ is zero,
$$
A_{\chi}(2,1)=\begin{pmatrix}0&0&0\\ -\alpha-z &\alpha &z \\ *&*& 1-\alpha\end{pmatrix}\quad\text{and}\quad
A_{\chi}(3,1)=\begin{pmatrix}0&0&0\\ \alpha+z-1 &1-\alpha &-z \\ *&*& \alpha\end{pmatrix}.
$$
Furthermore, since lower triangular idempotent matrices have~$0$ or~$1$ in each diagonal entry, necessarily $z\ne 0$.
Now it is clear that, since $\rk\bigr(A_{\chi}(2,1)\bigr)=\rk\bigl(A_{\chi}(3,1)\bigr)=1$, both matrices have the desired form. But then the first
row of $B_{\chi}(2,2)$ is $(0,\alpha,1-\alpha,0,\dots,0)$, the first row of $B_{\chi}(1,2)$ is $(1,-(\alpha+z),(\alpha+z)-1,0,\dots,0)$ and the
first row of $B_{\chi}(3,2)$ is $(0,z,-z,0,\dots,0)$. An easy computation using these facts, that by Remark~\ref{corolario con las B} we have $B(1,2)+B(2,2)+B(3,2)=\ide_3$, and that Proposition~\ref{caracterizacion}(2)  the columns of $B(2,2)$ are orthogonal to the first rows of $B(1,2)$ and $B(3,2)$,  shows that
$$
B(2,2)=\begin{pmatrix} 0&\alpha&1-\alpha&0&\dots&0\\0&\alpha&1-\alpha&0&\dots&0\\0&\alpha&1-\alpha&0&\dots&0\\ *&*&*&*& &*\\ \vdots&&&\vdots&&\vdots\\
*&*&*&*& &*\\ \end{pmatrix},
$$
which finishes the proof of item~(2) via \eqref{eq:rel entre A(i,l) y B(j,k)}.

Item~(1) follows from the fact that $A(3,2)_{22}=1-\alpha\notin\{0,1\}$ and $A(2,3)_{22}=\alpha\notin\{0,1\}$.
\end{proof}

Our next aim is to determine up to isomorphisms all twisting maps $\chi\colon K^3\ot K^3\longrightarrow K^3\ot K^3$ which are not quasi-standard.
For this, we can and we will assume that the values of the diagonal of $\Gamma_{\chi}$ are non increasing. So in the rest of this subsection $\chi$
denotes an arbitrary twisting map satisfying this restriction and we look for conditions in order that $\chi$ be not quasi-standard. We organize
our search according to the values of $\sum \Tr\coloneqq \sum_i \Tr(A_{\chi}(i,i))=\sum_j\Tr(B_{\chi}(j,j))$.

\subsubsection[$\sum \Tr=9,8$ or $7$]{$\bm{\sum \Tr=9,8}$ or $\bm{7}$}
Here the values of the diagonal of $\Gamma_{\chi}$ may be $(3,3,3)$, $(3,3,2)$, $(3,3,1)$  or $(3,2,2)$. By Proposition~\ref{casi todas id}, in the first three cases necessarily $\chi$ is a standard twisting map. In the last case $\Gamma_{\chi}$ is  equivalent via identical permutations in rows
and columns to one of the following matrices:
$$
\begin{pmatrix}3&1&1\\ 0&2&0\\ 0&0&2\end{pmatrix},\quad \begin{pmatrix}3&1&0\\ 0&2&1\\ 0&0&2\end{pmatrix}\quad\text{or}\quad
\begin{pmatrix}3&0&0\\ 0&2&1\\ 0&1&2\end{pmatrix}.
$$
By Proposition~\ref{forma general casi triangular} in the two first cases the diagonal matrices are $0,1$-matrices,
and so by Remark~\ref{en rango dos es standard} the obtained twisting maps are standard. In the last one $\chi$ is a direct sum of the flip of $K$
with $K^3$ and a twisting map $\chi'\colon K^2\ot K^3 \longrightarrow K^3\ot K^2$. Moreover, the analysis made out in
Subsection~\ref{de K3 con K2} shows that if $\chi'$ is not quasi-standard, then $K^3\ot_{\chi'} K^2$  is isomorphic to $K^2\times M_2(K)$.
Thus, in this case $K^3\ot_{\chi} K^3 \simeq K^5\times M_2(K)$.

\subsubsection[$\sum \Tr=6$]{$\bm{\sum \Tr=6}$}
The diagonal of $\Gamma_{\chi}$ is either $(2,2,2)$ or $(3,2,1)$. We treat each case separately:

\medskip

\noindent $\pmb{\Diag\bigl(\Gamma_{\chi}\bigr)=(2,2,2)}$ \enspace By Proposition~\ref{matrices de rango equivalentes} we can assume that the
first column is $(2,1,0)^\bot$, or, in other words, that $\Gamma_{\chi}$ it is one of the following matrices:
$$
\begin{pmatrix}2&1&1\\ 1 &2&0\\  0&0&2\end{pmatrix},\quad \begin{pmatrix}2&1&0\\ 1 &2&1\\  0&0&2\end{pmatrix},
\quad \begin{pmatrix}2&0&0\\ 1 &2&1\\  0&1&2\end{pmatrix}\quad \text{or}\quad
 \begin{pmatrix}2&0&1\\ 1 &2&0\\  0&1&2\end{pmatrix}.
$$
Moreover, by Proposition~\ref{forma general casi triangular} and Remark~\ref{en rango dos es standard}, each twisting map whose rank matrix is the
last one is standard, and, again by Proposition~\ref{matrices de rango equivalentes}, each twisting map whose rank matrix is the first or the
second one is isomorphic to one twisting map whose rank matrix is the third one. So we only must consider the case
$$
\Gamma_{\chi}=\begin{pmatrix}2&0&0\\ 1 &2&1\\  0&1&2\end{pmatrix}.
$$
Since, by Proposition~\ref{forma general casi triangular} and Remark~\ref{en rango dos es standard}, the first column is standard, the hypothesis of
Theorem~\ref{teorema principal} are satisfied. By this theorem and Proposition~\ref{extension de twistings}, we know that $\chi$ is twisting map if
and only if the first column of $\mathcal{A}_{\chi}$ is standard and the matrices $A_{\chi}(2,2)$, $A_{\chi}(3,2)$, $A_{\chi}(3,3)$ and
$A_{\chi}(2,3)$ define a twisting map of $K^2$ with $K^3$ such that $F(A_{\chi}(2,1))\subseteq F_0(\mathcal{A}_{\chi},2)$ (In fact, we also need
that $F(A(i,1))\subseteq F_0(\mathcal{A},i)$ for $i\in \{1,3\}$, but for $i=3$ this is trivial and for $i=1$ its follows from
Remark~\ref{columnas que se anulan en l_0}). Since we are looking for non quasi-standard twisting maps, by the discussion in Subsection~\ref{de K3
con K2} we may assume that
\begin{alignat*}{2}
&A_{\chi}(2,2)=\begin{pmatrix}1&0&0\\ 0 &a&1-a\\  0&a&1-a\end{pmatrix},\qquad && A_{\chi}(3,2)=\begin{pmatrix}0&0&0\\ 0 &1-a&a-1\\
0&-a&a\end{pmatrix},\\[3pt]
&A_{\chi}(3,3)=\begin{pmatrix}1&0&0\\ 0 &1-a&a\\  0&1-a&a\end{pmatrix},\qquad && A_{\chi}(2,3)=\begin{pmatrix}0&0&0\\ 0 &a&-a\\
0&a-1&1-a\end{pmatrix}.
\end{alignat*}
But then $F_0(\mathcal{A}_{\chi},2)=\{1\}$, so necessarily
$$
A_{\chi}(2,1)=\begin{pmatrix}1&-1&0\\ 0 &0&0\\  0&0&0\end{pmatrix}\quad\text{or}\quad
A_{\chi}(2,1)=\begin{pmatrix}1&0&-1\\ 0 &0&0\\  0&0&0\end{pmatrix}.
$$
In both cases setting $A_{\chi}(1,1)\coloneqq \ide_3-A_{\chi}(2,1)$, $A_{\chi}(3,1)=0$, $A_{\chi}(1,2)=0$ and $A_{\chi}(1,3)=0$ (which is forced),
we obtain a twisting map which is not quasi-standard. In the first one
$$
\tilde{\Gamma}_{\chi}=\begin{pmatrix}2&0&0\\ 1 &2&1\\  0&1&2\end{pmatrix},
$$
whereas in the second one
$$
\tilde{\Gamma}_{\chi}=\begin{pmatrix}2&0&0\\ 0 &2&1\\  1&1&2\end{pmatrix}.
$$
Taking into account Proposition~\ref{A quasi standard equivale a B quasi standard}, and applying the same arguments to $\tilde{\chi}$, we conclude that $\chi$ is a non quasi-standard twisting map with $\mathrm{\Diag\bigl(\Gamma_{\chi}\bigr)=(2,2,2)}$ if and only if $\tilde{\chi}$ is.

\medskip

\noindent $\pmb{\Diag\bigl(\Gamma_{\chi}\bigr)=(3,2,1)}$ \enspace  Assume that $\chi$ is a not quasi-standard twisting map. Then, by the last assertion we know that $\Diag\bigl(\tilde{\Gamma}_{\chi}\bigr)=(3,2,1)$. The rank matrix $\Gamma_{\chi}$ is one of the following matrices:
$$
\begin{pmatrix}3&0&0\\ 0 &2&2\\  0&1&1\end{pmatrix}, \quad
\begin{pmatrix}3&0&2\\ 0 &2&0\\  0&1&1\end{pmatrix}, \quad
\begin{pmatrix}3&0&1\\ 0 &2&1\\  0&1&1\end{pmatrix}, \quad
\begin{pmatrix}3&1&0\\ 0 &2&2\\  0&0&1\end{pmatrix},\quad
\begin{pmatrix}3&1&1\\ 0 &2&1\\  0&0&1\end{pmatrix},\quad
\begin{pmatrix}3&1&2\\ 0 &2&0\\  0&0&1\end{pmatrix}.
$$
By Proposition~\ref{Gamma en caso de rango de Aii igual a 1}, both $\Gamma_{\chi}$ and $\tilde{\Gamma}_{\chi}=\Gamma_{\tilde{\chi}}$ must be one of
the last two matrices. But by Corollary~\ref{forma general triangular}, Proposition~\ref{A standard equivale a B standard} and Remark~\ref{en rango
dos es standard}, if $\Gamma_{\chi}$ or $\tilde{\Gamma}_{\chi}$ is the last matrix, then $\chi$ is a standard twisting map. So the only chance of
being not standard for the twisting map~$\chi$  is that both $\Gamma_{\chi}$ and $\tilde{\Gamma}_{\chi}$ be the second last matrix. In that case
by Propositions~\ref{forma general casi triangular} and~\ref{cuando es cuasiestandard}(2) we recover the family of
quasi-standard twisting maps listed in number~$20$ in the appendix.

\subsubsection[$\sum \Tr=5$]{$\bm{\sum \Tr=5}$}
The diagonal of $\Gamma$ is either $(2,2,1)$ or $(3,1,1)$.  We treat each case separately:

\medskip

\noindent $\pmb{\Diag\bigl(\Gamma_{\chi}\bigr)=(2,2,1)}$ \enspace By Proposition~\ref{matrices de rango equivalentes} we can assume that the rank matrix $\Gamma_{\chi}$ is one of the following matrices:
\begin{equation}\label{posibles matrices}
\begin{gathered}
\begin{pmatrix}2&1&1\\ 1 &2&1\\  0&0&1\end{pmatrix},\quad
\begin{pmatrix}2&1&1\\ 0 &2&1\\  1&0&1\end{pmatrix},\quad
\begin{pmatrix}2&0&1\\ 0 &2&1\\  1&1&1\end{pmatrix},\quad
\begin{pmatrix}2&1&0\\ 1 &2&2\\  0&0&1\end{pmatrix}\\[3pt]
\begin{pmatrix}2&1&0\\ 0 &2&2\\  1&0&1\end{pmatrix},\quad
\begin{pmatrix}2&0&0\\ 1 &2&2\\  0&1&1\end{pmatrix},\quad
\begin{pmatrix}2&0&0\\ 0 &2&2\\  1&1&1\end{pmatrix}.
\end{gathered}
\end{equation}
Since $\tilde{\Gamma}$ has at least one $1$ in the diagonal, By Proposition~\ref{Gamma en caso de rango de Aii igual a 1} the rank matrix
$\Gamma_{\chi}$ can not be the first of the second row. Assume first that
\begin{equation}\label{primera matriz}
\Gamma_{\chi}=\begin{pmatrix}2&1&1\\ 1 &2&1\\  0&0&1\end{pmatrix}
\end{equation}
and that the twisting map $\chi$ is not standard. So, $\chi$ is an extension of a twisting map $\chi'$ of $K^2$ with $K^3$. Clearly, if the third
column of $\mathcal{A}_{\chi}$ is quasi-standard, then $\chi'$ must be a non quasi-standard twisting map. But by
Proposition~\ref{columna de unos no quasi-standard}(1) this is also the case if the third column of $\mathcal{A}_{\chi}$ is quasi-standard. Thus,
by the analysis made out in subsection~\ref{de K3 con K2}, we can assume that there exists $a\in K\setminus \{0,1\}$, such that
\begin{alignat*}{2}
& A_{\chi'}(1,1)=\begin{pmatrix}1&0&0\\ 0 &a&1-a\\  0&a&1-a\end{pmatrix}, \qquad && A_{\chi'}(1,2)=\begin{pmatrix}0&0&0\\ 0 &a&-a\\
0&a-1&1-a\end{pmatrix}\\[3pt]
& A_{\chi'}(2,1)=\begin{pmatrix}0&0&0\\ 0 &1-a&a-1\\  0&-a&a\end{pmatrix},\qquad && A_{\chi'}(2,2)=\begin{pmatrix}1&0&0\\ 0 &1-a&a\\  0&1-a&a\end{pmatrix}.
\end{alignat*}
If the third column of $\mathcal{A}_{\chi}$ is quasi-standard, then by Theorem~\ref{teorema principal}, we have
$$
\{1\}=F_0(\mathcal{A}_{\chi},1)\supseteq F(A_{\chi}(1,3))\ne F(A_{\chi}(2,3))\subseteq F_0(\mathcal{A}_{\chi},2)=\{1\},
$$
a contradiction. Hence it is not quasi-standard. Moreover, by Proposition~\ref{forma general casi triangular} necessarily $A_{\chi}(3,3)$ is one of the following matrices:
\begin{equation}\label{posibilidades}
\begin{pmatrix}1&0&0\\ 1&0&0\\  1&0&0\end{pmatrix},\quad
\begin{pmatrix}0&1&0\\ 0 &1&0\\  0&1&0\end{pmatrix},\quad
\begin{pmatrix}0&0&1\\ 0 &0&1\\  0&0&1\end{pmatrix}.
\end{equation}
In the two last cases a straightforward computation using Propositions~\ref{matrices de rango equivalentes} and~\ref{columna de unos no
quasi-standard}(2) leads to the contradiction $A(1,2)_{11}\ne 0$. Hence $A_{\chi}(3,3)$ is the first matrix. Now applying Proposition~\ref{columna
de unos no quasi-standard}, we obtain a family of not quasi-standard twisting maps parameterized by $\alpha\in K\setminus\{0,1\}$ and $z\in
K^\times$. Moreover, we have
\begin{equation}\label{como es gammatilde 1er caso}
\tilde{\Gamma}_{\chi}=\begin{pmatrix}3&1&1\\ 0 &1&1\\  0&1&1\end{pmatrix}.
\end{equation}
If $\Gamma_{\chi}$ is not the matrix at the right side of equality~\eqref{primera matriz} , then $\tilde{\Gamma}_{\chi}$ can not be that matrix
because~$\Gamma_{\chi}$ would be the matrix at the right side of equality~\eqref{como es gammatilde 1er caso}. But all the other possible matrices
for $\tilde{\Gamma}_{\chi}$ (including those with diagonal $(3,1,1)$) have exactly one row without zeroes, and so, by
Propositions~\ref{Gamma en caso de rango de Aii igual a 1} and~\ref{matrices de rango equivalentes} we can assume that $A_{\chi}(3,3)$ is the first matrix in~\eqref{posibilidades}. By Proposition~\ref{columna de unos no quasi-standard}, if
$$
\Gamma_{\chi}=\begin{pmatrix}2&1&1\\ 0 &2&1\\  1&0&1\end{pmatrix}\quad\text{or}\quad
\Gamma_{\chi}=\begin{pmatrix}2&0&1\\ 0 &2&1\\  1&1&1\end{pmatrix},
$$
then $\chi$ is a quasi-standard twisting map. If
$$
\Gamma_{\chi}=\begin{pmatrix}2&1&0\\ 1 &2&2\\  0&0&1\end{pmatrix},
$$
then~$\chi$ must be quasi-standard. In fact, otherwise it is an extension of a not quasi-standard twisting map of $K^3$ with $K^2$, and we know
that in this case~$\# F_0(\mathcal{A}_{\chi},2)=1$, which contradict the fact that $\# F(A_{\chi}(2,3))=2$ and $F(A_{\chi}(2,3))\subseteq
F_0(\mathcal{A}_{\chi},2)$ by Theorem~\ref{teorema principal}. Finally, if
$$
\Gamma_{\chi}=\begin{pmatrix}2&0&0\\ 1 &2&2\\  0&1&1\end{pmatrix}\quad\text{or} \quad \Gamma_{\chi}=\begin{pmatrix}2&0&0\\ 1 &2&2\\  0&1&1\end{pmatrix},
$$
then $\chi$ must be a standard twisting map, since the first column is standard and it is the extension of a standard twisting map of $K^3$ with $K^2$ (which follows from Proposition~\ref{A standard equivale a B standard} and the analysis made out in Subsection~\ref{de K3 con K2}).

\medskip

\noindent $\pmb{\Diag\bigl(\Gamma_{\chi}\bigr)=(3,1,1)}$ \enspace By Proposition~\ref{matrices de rango equivalentes} we can assume that the rank matrix $\Gamma_{\chi}$ is one of the following matrices:
$$
\begin{pmatrix}3&2&2\\ 0 &1&0\\  0&0&1\end{pmatrix},\quad
\begin{pmatrix}3&2&0\\ 0 &1&2\\  0&0&1\end{pmatrix},\quad
\begin{pmatrix}3&0&0\\ 0 &1&2\\  0&2&1\end{pmatrix},\quad
\begin{pmatrix}3&0&1\\ 0 &1&1\\  0&2&1\end{pmatrix},\quad
\begin{pmatrix}3&2&1\\ 0 &1&1\\  0&0&1\end{pmatrix},\quad
\begin{pmatrix}3&1&1\\ 0 &1&1\\  0&1&1\end{pmatrix}.
$$
By Proposition~\ref{forma general casi triangular}, if $\Gamma_{\chi}$ is the first matrix, then $\chi$ is a standard twisting map. Assume that
$\Gamma_{\chi}$ is not the first matrix, which by Propositions~\ref{Gamma en caso de rango de Aii igual a 1} implies that $\tilde{\Gamma}_{\chi}$
has one row without zeroes. By the arguments given above, if $\Gamma_{\chi}$ is not the last matrix, then $\tilde{\Gamma}_{\chi}$ can not be
equivalent via identical permutations in rows and columns to the first matrix in the second row of~\eqref{posibles matrices}. Hence,
$\tilde{\Gamma}_{\chi}$ has exactly one row without zeroes, and by Propositions~\ref{Gamma en caso de rango de Aii igual a 1}
and~\ref{columna de unos no quasi-standard}, if $\chi$ is not quasi-standard, then~$A_{\chi}(2,1)\ne 0\ne A_{\chi}(1,2)$. So, necessarily $\chi$
is quasi-standard, and in fact there is quasi-standard twisting maps with $\Gamma_{\chi}$  the fifth matrix (see the appendix). But if
$\Gamma_{\chi}$ is the second, third or fourth matrix, then condition~(1) in Theorem~\ref{teorema principal} is not fulfilled, and thus
there is not twisting maps in these cases. Finally, there is a family of not quasi-standard twisting maps~$\chi$ with $\Gamma_{\chi}$ the last
matrix, dual to the family found above, when analyzing the case $\Diag\bigl(\Gamma_{\chi}\bigr)=(2,2,1)$.

\subsubsection[$\sum \Tr=4$]{$\bm{\sum \Tr=4}$} We claim that in this case all twisting maps are quasi-standard. By Proposition~\ref{matrices de rango equivalentes} in order to prove this it suffices to check that $\chi$ is  quasi-standard if its rank matrix $\Gamma_{\chi}$ is one of the following matrices:
\begin{alignat*}{3}
&\begin{pmatrix}2&2&2\\ 1 &1&0\\  0&0&1\end{pmatrix},\quad &&
\begin{pmatrix}2&0&2\\ 1 &1&0\\  0&2&1\end{pmatrix},\quad &&
\begin{pmatrix}2&2&0\\ 1 &1&2\\  0&0&1\end{pmatrix},\\[3pt]
&\begin{pmatrix}2&0&0\\ 1 &1&2\\  0&2&1\end{pmatrix},\quad &&
\begin{pmatrix}2&2&1\\ 1 &1&1\\  0&0&1\end{pmatrix},\quad &&
\begin{pmatrix}2&0&1\\ 1 &1&1\\  0&2&1\end{pmatrix},\\[3pt]
&\begin{pmatrix}2&1&0\\ 1 &1&2\\  0&1&1\end{pmatrix},\quad &&
\begin{pmatrix}2&1&2\\ 1 &1&0\\  0&1&1\end{pmatrix},\quad &&
\begin{pmatrix}2&1&1\\ 1 &1&1\\  0&1&1\end{pmatrix}.
\end{alignat*}
If $\Gamma_{\chi}$ is the first or the third matrix of the first row, then $\chi$ is a extension of a standard twisting map $\chi'$ of $K^2$ with
$K^3$, whose added column is standard, and so it is standard. If $\Gamma_{\chi}$ is the second matrix of the second row,
then $\chi$ is an extension of a standard twisting map of~$K^2$ with~$K^3$. Moreover,  by Proposition~\ref{forma general casi triangular} we know
that $A_{\chi}(3,3)$ is equivalent via identical permutations in rows and columns
to a standard idempotent $(0,1)$-matrix, and hence,  by proposition~\ref{columna de unos no quasi-standard}, the twisting map
$\chi$ is quasi-standard. By Proposition~\ref{Gamma en caso de rango de Aii igual a 1}, the rank matrix $\Gamma_{\chi}$ can not be the second
matrix in the first row. Also $\Gamma_{\chi}$ can not be the first matrix in the second row,  because otherwise it would be the extension of a
twisting map $\chi'$ of $K^2$ with $K^3$ with $\Gamma_{\chi'}=\begin{psmallmatrix}1&2\\ 2&1 \end{psmallmatrix}$, but $\sum\Tr=2$ is impossible. So,
we are left with the last four matrices. Assume first that $\Gamma_{\chi}$ is the last one. By Proposition~\ref{matrices de rango equivalentes}
we also can assume that $\Diag(\Gamma_{\tilde{\chi}})=(2,1,1)$. In this case $B_{\chi}(2,1)=0$ or $B_{\chi}(3,1)=0$, both cases being equivalent
via Proposition 2.4 with $\sigma=\ide$ and $\tau=(2,3)$. So, assume that $B_{\chi}(2,1)=0$, which by Remark~\ref{rango uno generalizado} implies
that
$$
A_{\chi}(2,2)=\begin{pmatrix} \ast &0 &\ast\\ \ast &0 &\ast\\ \ast &0 &\ast\end{pmatrix}.
$$
Moreover, again by Remark~\ref{rango uno generalizado}, $A_{\chi}(3,1)=0$ implies that
$$
B_{\chi}(3,3)=\begin{pmatrix} \ast&\ast &0\\ \ast&\ast &0\\ \ast&\ast &0\end{pmatrix}\quad\text{and}\quad
B_{\chi}(2,2)=\begin{pmatrix} \ast&\ast &0\\ \ast&\ast &0\\ \ast&\ast &0\end{pmatrix}.
$$
Hence $\Diag(A_{\chi}(3,2))=(*,0,0)$ and
$$
A_{\chi}(3,3)=\begin{pmatrix} 1&0 &0\\ 1&0 &0\\ 1&0 &0\end{pmatrix},
$$
where for the last equality we use once more Remark~\ref{rango uno generalizado}. Since $\rk(A_{\chi}(3,2))=1$ it follows that
$\Diag(A_{\chi}(3,2))=(1,0,0)$, and therefore
$$
A_{\chi}(3,2)=\begin{pmatrix} 1 &0 &-1\\ \ast &0 &\ast\\ 0 &0 &0\end{pmatrix}.
$$
Hence,
$$
A_{\chi}(1,2)=\ide - A_{\chi}(2,2) - A_{\chi}(3,2) = \begin{pmatrix} \ast &0 &\ast\\ \ast &1 &\ast\\ \ast &0 &\ast\end{pmatrix} = \begin{pmatrix} 0 &0 & 0\\ \ast &1 &\ast\\ 0 &0 & 0\end{pmatrix},
$$
where for the last equality we use that~$\rk\bigl(A_{\chi}(1,2)\bigr)=1$, and so
$$
A_{\chi}(2,2)=\begin{pmatrix} 0&0&1\\ 0&0&1\\ 0&0&1\end{pmatrix}.
$$
Now, from Proposition~\ref{columna de unos no quasi-standard}(1) it follows that $\chi$ is quasi-standard. For the remaining three cases, the
only way that $\chi$ can be not quasi-standard is that $\tilde{\Gamma}_{\chi}= \Gamma_{\tilde{\chi}}$ has exactly one row without zeroes.
But then Propositions~\ref{Gamma en caso de rango de Aii igual a 1} and~\ref{columna de unos no quasi-standard}
shows that the twisting map $\chi$ is quasi-standard.

\subsubsection[$\sum \Tr=3$]{$\bm{\sum \Tr=3}$}
By Proposition~\ref{diagonales uno} we know that $\Gamma_{\chi}=\tilde{\Gamma}_{\chi}=\mathfrak{J}_3$.  We have the following
possibilities for each matrix $A_{\chi}(i,i)$ and each $B_{\chi}(j,j)$. It is equivalent to a standard idempotent $0,1$-matrix via identical
permutations in rows and columns, it has all entries non-zero, or it has two non-zero columns and one zero column. If two of $A_{\chi}(1,1)$,
$A_{\chi}(2,2)$ and $A_{\chi}(3,3)$ are $0,1$-matrices, then all the matrices $B_{\chi}(j,k)$ have zeroes and ones in its diagonal entries, and,
moreover, by Remark~\ref{rango uno generalizado} each $B_{\chi}(j,j)$ is a $(0,1)$-matrix. Therefore the hypothesis of Proposition~\ref{cuando es
cuasiestandard} are fulfilled, and we have a quasi-standard twisting map. On the other hand, if one of the $A_{\chi}(i,i)$ (say $A_{\chi}(1,1)$)
has all its entries non-zero, then $\chi$ is a non quasi-standard twisting map which yields a tensor
product algebra isomorphic to $M_3(K)$. In fact, the existence follows from Proposition~\ref{proposicion determinante} and
Theorem~\ref{existencia, mas manejable} (with $\mathbf{v}_2$ and $\mathbf{v}_3$ vectors that generate the images of $A_{\chi}(2,1)$ and
$A_{\chi}(3,1)$, respectively), and the uniqueness follows from Remark~\ref{unicidad mejorada}).
So there are two cases left:
\begin{itemize}

\smallskip

\item[-] All three matrices $A_{\chi}(1,1)$, $A_{\chi}(2,2)$ and $A_{\chi}(3,3)$ have exactly one zero column.

\smallskip

\item[-] One of them (for example $A_{\chi}(1,1)$) is a $0,1$-matrix, the other two have exactly one zero column.

\smallskip

\end{itemize}
In the first case a straightforward computation shows that the resulting twisting map (up to an isomorphism) is given by
\begin{alignat*}{3}
& A(1,1)= \begin{pmatrix}
                        a & b & 0 \\
                        a & b & 0 \\
                        a & b & 0
          \end{pmatrix}, \quad &&
  A(1,2)= \begin{pmatrix}
                        a & -a & 0 \\
                        -b & b & 0 \\
                        -b & b & 0
          \end{pmatrix},\quad &&
 A(1,3)=  \begin{pmatrix}
                         a & -a & 0 \\
                        -b & b & 0 \\
                         a & -a & 0
          \end{pmatrix}\\[3pt]
& A(2,1)= \begin{pmatrix}
                         b & 0 & -b \\
                        -a & 0 & a \\
                        -a & 0 & a
          \end{pmatrix}, \quad &&
 A(2,2)=  \begin{pmatrix}
                         b & 0 & a \\
                         b & 0 & a \\
                         b & 0 & a
          \end{pmatrix},\quad &&
 A(2,3)=  \begin{pmatrix}
                         b & 0 & -b \\
                         b & 0 & -b \\
                        -a & 0 & a
          \end{pmatrix}\\[3pt]
& A(3,1)= \begin{pmatrix}
                         0 & -b & b \\
                         0 & a & -a \\
                         0 & -b & b
          \end{pmatrix}, \quad &&
  A(3,2)= \begin{pmatrix}
                         0 & a & -a \\
                         0 & a & -a \\
                         0 & -b & b
          \end{pmatrix},\quad &&
  A(3,3)= \begin{pmatrix}
                         0 & a & b \\
                         0 & a & b \\
                         0 & a & b
          \end{pmatrix},
\end{alignat*}
for some $a\notin\{0,1\}$ and $b\coloneqq 1-a$, which gives a family of twisting maps parameterised by $a\in K$. In the second case we can
assume that
$$
A(1,1)=\begin{pmatrix} 1&0&0\\ 1&0&0\\ 1&0&0\end{pmatrix}
$$
and that the first column is not quasi-standard. A  straightforward computation along the lines of the proof
of~Proposition~\ref{columna de unos no quasi-standard} show that there exists $a\notin\{0,1\}$, $b\coloneqq 1-a$ and $x,y\in k^{\times}$ such that
\begin{alignat*}{3}
& A(1,1)= \begin{pmatrix}
                         1 & 0 & 0 \\
                         1 & 0 & 0 \\
                         1 & 0 & 0
          \end{pmatrix}, \quad &&
  A(1,2)= \begin{pmatrix}
                         1 & p & q \\
                         0 & 0 & 0 \\
                         0 & 0 & 0 \\
          \end{pmatrix},\quad &&
 A(1,3)=  \begin{pmatrix}
                         1 & t & v \\
                         0 & 0 & 0 \\
                         0 & 0 & 0 \\
          \end{pmatrix}\\[3pt]
& A(2,1)= \begin{pmatrix}
                          0 & 0 & 0 \\
                          r & a & y \\
                          u & \frac{a b}{y} & b
          \end{pmatrix}, \quad &&
 A(2,2)=  \begin{pmatrix}
                         0 & a & b \\
                         0 & a & b \\
                         0 & a & b
          \end{pmatrix},\quad &&
 A(2,3)=  \begin{pmatrix}
                         0 & \phantom{-}\frac{a b}{x} & -\frac{a b}{x} \\
                         0 & \phantom{-}a & -a \\
                         0 & -b & \phantom{-}b
          \end{pmatrix}\\[3pt]
& A(3,1)= \begin{pmatrix}
                         0 & 0 & 0 \\
                         s & \phantom{-}b & -y \\
                         w & -\frac{a b}{y} & \phantom{-}a
          \end{pmatrix}, \quad &&
  A(3,2)= \begin{pmatrix}
                          0 & \phantom{-}x & -x \\
                          0 & \phantom{-}b & -b \\
                          0 & -a & \phantom{-}a
          \end{pmatrix},\quad &&
  A(3,3)= \begin{pmatrix}
                          0 & b & a \\
                          0 & b & a \\
                          0 & b & a
          \end{pmatrix},
\end{alignat*}
where $p\coloneqq -a-x$, $q\coloneqq x-b$, $r\coloneqq -a-y$, $s\coloneqq y-b$, $t\coloneqq -\frac{b (a+x)}{x}$, $u\coloneqq -\frac{b (a+y)}{y}$,
$v\coloneqq \frac{a b}{x}-a$ and $w\coloneqq \frac{a b}{y}-a$.

\setcounter{secnumdepth}{0}
\setcounter{section}{1}
\setcounter{theorem}{0}
\section[Appendix: Quasi-standard twisting maps of $K^3$ with $K^3$]{Appendix: Quasi-standard twisting maps of $\bm{K^3}$ with $\bm{K^3}$}
\renewcommand\thesection{\Alph{section}}
\renewcommand\sectionmark[1]{}

Next we list the quasi-standard twisting maps of $K^3$ with $K^3$. For this, first we construct the standard ones using the method given in Remark~\ref{construccion de los twisting estandar}, and then we construct the remaining quasi-standard twisting maps using the recursive method developed in Subsection~\ref{subseccion construccion de quasi estandars}.  Is not possible iterate arbitrarily the steps in this construction  because the conditions in item~(2) of Proposition~\ref{construccion de quasiestandar} would not be satisfied (see for instance the last item in following list).
\setlength{\tabcolsep}{8pt}
\ra{1.2}
\begin{longtabu}{ccccccc}
\caption{Quasi-standard twisting maps of $K^3$ with $K^3$} \\
\toprule
$\#$ &$\sum \Tr$ & quiver& $\Gamma_{\chi}$ & $\tilde{\Gamma}_{\chi}$ & $\#$ equiv. & quasi-st.\\
\midrule
\endfirsthead
$\#$ &$\sum \Tr$ & quiver& $\Gamma_{\chi}$ & $\tilde{\Gamma}_{\chi}$ & $\#$ equiv. & quasi-st. \\
\midrule
\endhead
\midrule
\multicolumn{7}{r}{\emph{continued on next page \ldots}}
\endfoot
\bottomrule
\endlastfoot
1.  &9 & \raisebox{-3.5\height/8}
{
 }
& $\begin{psmallmatrix} 3&1&1\\ 0&1&1\\ 0&1&1 \end{psmallmatrix}$ & $\begin{psmallmatrix} 2&1&1\\ 1&2&1\\ 0&0&1 \end{psmallmatrix}$& 18
&$\begin{array}{c}\chi_{_1}=\Lambda^{\lambda_1}_{(3,1)(1,2)(2,3)}(\chi)\\[2pt] \chi_{_2}=\Lambda^{\lambda_2}_{(3,1)(2,3)(1,2)}(\chi)\end{array}%
$  \\ &\\
    \end{longtabu}


\begin{bibdiv}
\begin{biblist}

\bib{B-M1}{article}{
   author={Brzezi{\'n}ski, Tomasz},
   author={Majid, Shahn},
   title={Coalgebra bundles},
   journal={Comm. Math. Phys.},
   volume={191},
   date={1998},
   number={2},
   pages={467--492},
   issn={0010-3616},
   review={\MR{1604340}},
   doi={10.1007/s002200050274},
}

\bib{B-M2}{article}{
   author={Brzezi{\'n}ski, Tomasz},
   author={Majid, Shahn},
   title={Quantum geometry of algebra factorisations and coalgebra bundles},
   journal={Comm. Math. Phys.},
   volume={213},
   date={2000},
   number={3},
   pages={491--521},
   issn={0010-3616},
   review={\MR{1785427}},
   doi={10.1007/PL00005530},
}

\bib{C-S-V}{article}{
   author={Cap, Andreas},
   author={Schichl, Hermann},
   author={Van{\v{z}}ura, Ji{\v{r}}{\'{\i}}},
   title={On twisted tensor products of algebras},
   journal={Comm. Algebra},
   volume={23},
   date={1995},
   number={12},
   pages={4701--4735},
   issn={0092-7872},
   review={\MR{1352565}},
   doi={10.1080/00927879508825496},
}

\bib{Ca}{article}{
   author={Cartier, Pierre},
   title={Produits tensoriels tordus},
   journal={Expos\'{e} au S\'{e}minaire des groupes quantiques de l' \'{E}cole Normale
Sup\'{e}rieure, Paris},
   date={1991-1992},
}

\bib{C-I-M-Z}{article}{
   author={Caenepeel, S.},
   author={Ion, Bogdan},
   author={Militaru, G.},
   author={Zhu, Shenglin},
   title={The factorization problem and the smash biproduct of algebras and
   coalgebras},
   journal={Algebr. Represent. Theory},
   volume={3},
   date={2000},
   number={1},
   pages={19--42},
   issn={1386-923X},
   review={\MR{1755802}},
   doi={10.1023/A:1009917210863},
}

\bib{C}{article}{
   author={Cibils, Claude},
   title={Non-commutative duplicates of finite sets},
   journal={J. Algebra Appl.},
   volume={5},
   date={2006},
   number={3},
   pages={361--377},
   issn={0219-4988},
   review={\MR{2235816}},
   doi={10.1142/S0219498806001776},
}

\bib{G}{article}{
   author={Gerstenhaber, Murray},
   title={On the deformation of rings and algebras},
   journal={Ann. of Math. (2)},
   volume={79},
   date={1964},
   pages={59--103},
   issn={0003-486X},
   review={\MR{0171807}},
}

\bib{G-G}{article}{
   author={Guccione, Jorge A.},
   author={Guccione, Juan J.},
   title={Hochschild homology of twisted tensor products},
   journal={$K$-Theory},
   volume={18},
   date={1999},
   number={4},
   pages={363--400},
   issn={0920-3036},
   review={\MR{1738899}},
   doi={10.1023/A:1007890230081},
}

\bib{JLNS}{article}{
   author={Jara, P.},
   author={L{\'o}pez Pe{\~n}a, J.},
   author={Navarro, G.},
   author={{\c{S}}tefan, D.},
   title={On the classification of twisting maps between $K^n$ and $K^m$},
   journal={Algebr. Represent. Theory},
   volume={14},
   date={2011},
   number={5},
   pages={869--895},
   issn={1386-923X},
   review={\MR{2832263}},
   doi={10.1007/s10468-010-9222-x},
}

\bib{J-L-P-V}{article}{
   author={Jara Mart{\'{\i}}nez, Pascual},
   author={L{\'o}pez Pe{\~n}a, Javier},
   author={Panaite, Florin},
   author={van Oystaeyen, Freddy},
   title={On iterated twisted tensor products of algebras},
   journal={Internat. J. Math.},
   volume={19},
   date={2008},
   number={9},
   pages={1053--1101},
   issn={0129-167X},
   review={\MR{2458561}},
   doi={10.1142/S0129167X08004996},
}

\bib{Ka}{book}{
   author={Kassel, Christian},
   title={Quantum groups},
   series={Graduate Texts in Mathematics},
   volume={155},
   publisher={Springer-Verlag, New York},
   date={1995},
   pages={xii+531},
   isbn={0-387-94370-6},
   review={\MR{1321145}},
   doi={10.1007/978-1-4612-0783-2},
}

\bib{LN}{article}{
   author={L{\'o}pez Pe{\~n}a, Javier},
   author={Navarro, Gabriel},
   title={On the classification and properties of noncommutative duplicates},
   journal={$K$-Theory},
   volume={38},
   date={2008},
   number={2},
   pages={223--234},
   issn={0920-3036},
   review={\MR{2366562}},
   doi={10.1007/s10977-007-9017-y},
}

\bib{Ma1}{article}{
   author={Majid, Shahn},
   title={Physics for algebraists: noncommutative and noncocommutative Hopf
   algebras by a bicrossproduct construction},
   journal={J. Algebra},
   volume={130},
   date={1990},
   number={1},
   pages={17--64},
   issn={0021-8693},
   review={\MR{1045735}},
   doi={10.1016/0021-8693(90)90099-A},
}

\bib{Ma2}{article}{
   author={Majid, Shahn},
   title={Algebras and Hopf algebras in braided categories},
   conference={
      title={Advances in Hopf algebras},
      address={Chicago, IL},
      date={1992},
   },
   book={
      series={Lecture Notes in Pure and Appl. Math.},
      volume={158},
      publisher={Dekker, New York},
   },
   date={1994},
   pages={55--105},
   review={\MR{1289422}},
}

\bib{Mo}{book}{
   author={Montgomery, Susan},
   title={Hopf algebras and their actions on rings},
   series={CBMS Regional Conference Series in Mathematics},
   volume={82},
   publisher={Published for the Conference Board of the Mathematical
   Sciences, Washington, DC; by the American Mathematical Society,
   Providence, RI},
   date={1993},
   pages={xiv+238},
   isbn={0-8218-0738-2},
   review={\MR{1243637}},
   doi={10.1090/cbms/082},
}

\bib{Pierce}{book}{
   author={Pierce, Richard S.},
   title={Associative algebras},
   series={Graduate Texts in Mathematics},
   volume={88},
   note={Studies in the History of Modern Science, 9},
   publisher={Springer-Verlag, New York-Berlin},
   date={1982},
   pages={xii+436},
   isbn={0-387-90693-2},
   review={\MR{674652}},
}

\bib{Tam}{article}{
   author={Tambara, D.},
   title={The coendomorphism bialgebra of an algebra},
   journal={J. Fac. Sci. Univ. Tokyo Sect. IA Math.},
   volume={37},
   date={1990},
   number={2},
   pages={425--456},
   issn={0040-8980},
   review={\MR{1071429}},
}

\bib{VD-VK}{article}{
   author={Van Daele, A.},
   author={Van Keer, S.},
   title={The Yang-Baxter and pentagon equation},
   journal={Compositio Math.},
   volume={91},
   date={1994},
   number={2},
   pages={201--221},
   issn={0010-437X},
   review={\MR{1273649}},
}

\end{biblist}
\end{bibdiv}
\end{document}